\documentclass[11]{preprint}
\usepackage[english]{babel}
\usepackage{amsmath,amssymb,xcolor,latexsym}

\usepackage{tommasos-commands, mhequ, comment}

\usepackage[top=3.4cm, bottom=3.4cm, left=3.1cm, right=3.1cm]{geometry}

\newcommand{\assign}{\stackrel{\mathrm{def}}{=}}

\newenvironment{proof}{\noindent\textbf{Proof\ }}{\hspace*{\fill}$\Box$\medskip}

\begin{document}

\title{Lyapunov exponents in a slow environment}

\author{Tommaso Rosati}

\institute{Imperial College London, \email{t.rosati@imperial.ac.uk}} 

\maketitle

\begin{abstract}
	Motivated by the evolution of a population in a slowly varying random
environment, we consider the 1D Anderson model on finite volume, with viscosity $ \kappa > 0 $:
\begin{equs}
\partial_{t} u(t,x) = \kappa \Delta u(t,x) + \xi(t, x) u(t,x), \quad u(0, x) =
u_{0}(x), \qquad t > 0, x \in \TT.
\end{equs} 
The noise $ \xi $ is chosen constant on time intervals of length $ \tau >0 $ and sampled 
independently after a time $ \tau $. We prove that the Lyapunov
exponent $ \lambda (\tau) $ is positive and near $ \tau= 0 $ follows a power law that depends
on the regularity on the driving noise. As $ \tau \to \infty $ the Lyapunov
exponent converges to the average top eigenvalue of the associated time-independent
Anderson model. The proofs make use of a solid control of the projective
component of the solution and build on the
Furstenberg--Khasminskii and Bou\'e--Dupuis formulas, as well as on Doob's
H-transform and on tools from singular stochastic PDEs.
\end{abstract}

\section{Introduction}

In this work, we study a parabolic Anderson model with viscosity $ \kappa >0 $ and periodic
boundary conditions:
\begin{equ}[eqn:time-Anderson]
   \partial_t u (t, x)  =  \kappa \Delta u (t, x) + \xi (t, x) u (t, x),
\qquad u (0, x) = u_0 (x) \geqslant 0, \qquad t > 0, \ x \in \TT,
\end{equ}
where $ \TT $ is the 1D torus.
We understand $ u $ as the density of a population evolving in the environment $
\xi $. Our main assumption is that the environment varies on a slower timescale
with respect to the population, so that one expects the underlying
population to adapt, more or less rapidly and efficiently, to the surrounding environment.
This is a natural setting for
evolutionary models. For instance, adaptation to
seasonal influences was recently observed in certain kinds of fruit flies
\cite{Bergland2014Fly} and many other
dynamical effects in the study of biodiversity are linked fluctuating
environments that are colored in time
\cite{bell2010fluctuating, Vasseur2004Color}, see also \cite{Benaim}. 
In this context \eqref{eqn:time-Anderson} arises as the limit of a number of
microscopic, finite-dimensional evolution models.

The aim of this article is to show how the
long term effect of the environment changes both as a function of the timescale at which the noise varies and of
the spatial variability of the noise.
We will therefore introduce a parameter $ \tau > 0 $ and choose $
\xi(t, x)$ such that $ \xi $ is constant on time intervals of length $ \tau
$, after which it is chosen independently according to a fixed centered distribution.
Then the capacity of the population to adapt to and exploit the environment can be measured
by the Lyapunov exponent $ \lambda(\tau) $ of the system. 
We expect that if $ \tau \gg 1 $ the population can
optimize its distribution around favorable parts of the potential, where
it can grow rapidly, while if the environment varies quickly it averages out and the population
is not able to exploit it.
Following our heuristic we will prove that the Lyapunov exponent is strictly
positive and for $ \tau \to \infty $, $ \lambda(\tau) \to \EE
[\lambda_{\mathrm{stat}}] > 0 $, with $
\lambda_{\mathrm{stat}} $ being the largest eigenvalue of the associated
time-independent dynamic, whereas $ \lambda (\tau) \to 0 $ for $ \tau \to 0 $.

In addition, simulations show that $ \lambda (\tau) $ stabilizes relatively
quickly around $ \EE [\lambda_{\mathrm{stat}}] $, so the behaviour near $ \tau
=0 $ is particularly interesting and gives a measure of the speed of adaptation
of the population. Here we prove that the Lyapunov exponent
follows a power law that depends on the regularity of the noise. We consider two archetypal settings:
in the first one the noise is function valued, in the second one it is
space white noise. In the first
setting we show that $ \lambda(\tau) $ grows linearly, while in the second
setting $ \lambda(\tau) \simeq \sqrt{ \kappa^{-1} \pi\tau}. $ 

This behaviour can be explained by considering a different scaling. If we multiply $ \xi
$ with a factor $ \tau^{- \frac{1}{2} } $ we expect to see a Stratonovich
equation in the limit $ \tau \to 0 $ and by scaling we also expect that $\lambda(\tau) $ gets roughly multiplied by $
\tau^{-1} $ -- since in the limiting equation the Lyapunov exponent depends
linearly on the second moment of the noise -- and after this rescaling we expect its
convergence to the Lyapunov exponent of the Stratonovich
equation. This is the case is if the noise is regular: instead for irregular noise the Stratonovich equation makes sense only after
renormalization (i.e. after subtracting an It\^o correction), which accounts for the blow-up of order $
\tau^{- \frac{1}{2}}$ in the new scaling.

However, we will not follow precisely the approach we just outlined. Instead,
on the one hand, our scaling slightly simplifies the setting, in the sense that we can work
entirely without making use of Wong--Zakai
results: it will be sufficient to establish averaging to the heat equation as
$ \tau \to 0 $. On the other hand, since in our scaling we are essentially
performing a Taylor expansion, we will need to focus on certain additional moment estimates.

As a notable side effect of our scaling, we can recover the
renormalization constant for the multiplicative stochastic heat equation purely from the large scale dynamical properties of
the equation, without considering the small-scale behaviour that usually
motivates the use of tools from singular SPDEs \cite{Hairer2014, GubinelliImkellerPerkowski2015}. These tools will still be
required, though in a different setting and in combination with different
(indeed fewer and simpler) stochastic estimates than in the case of the
multiplicative stochastic heat equation.

It is also crucial to observe that Wong--Zakai results for the associated SPDEs with Stratonovich or
white noise, such as \cite{HairerPardoux2015WongZakai}, are not per se sufficient to derive convergence of the
Lyapunov exponents, since they consider convergence on compact time intervals and do not cover
the large scales behavior. The key additional ingredient in the
control of the longtime behaviour is the projective invariant measure
associated to the equation. The latter is the limit of the ``projective''
process $ u_{t} / \| u_{t} \| $, for some suitable norm $ \| \cdot \| $. If one establishes convergence of the solution map
as well as of the named invariant measure in a suitable space and with
sufficient moment estimates, then the convergence of the Lyapunov exponent will
follow. This is, in a nutshell, the approach discussed for SDEs in the
monograph \cite{KhasminskiiBook}. In full generality there is no infinite-dimensional
extension of this theory, due to a lack of understanding
of the projective component. To the best of our knowledge the present order
preserving case is the only one in which a spectral gap for the projective
component is available and has been studied in a series of papers
\cite{Mierczynski2013First, Mierczynski2013Second, Mierczynski2016Third} and
\cite{Mierczinsky2016LyapFormula}, see also the survey
\cite{Mierczynski2014Survey} and the book \cite{Hess1991periodic} for the
time-periodic case.

At the heart of our arguments lies the strict positivity
of the solution map to \eqref{eqn:time-Anderson}, together with classical
approaches for products of random matrices.
We will decompose $ u_{t} = \|
u_{t} \|_{L^{1}} z_{t} $, where $ z_{t} > 0 $ integrates to one and is
the projective component of $ u_{t} $. It is then useful to
endow the projective space in which $
z_{t} $ takes values with a particular topology under which positive linear
maps satisfy a contraction principle. This guarantees the existence of a spectral gap
for the process $ z_{t} $, as was first observed for random matrices
\cite{ArnoldGundlachDemetrius1994, Hennion1997} -- but the result extends immediately to
SPDEs \cite{Mierczynski2016Third, Rosati2019synchronization}.
We observe that in the existing literature unique ergodicity of the
projective component has been studied in many different forms over the past
years. A seminal work by Sinai based on the representation of $ z_{t} $ via
directed polymers \cite{Sinai1991Buergers} has been extended to cases
without viscosity and on infinite volume
\cite{WeinanKhaninSinai2000BurgersInviscous, BakhtinCatorKhanin2014Burgers, DunlapGrahamRyzhik2019stationary} in the context of
ergodicity for Burgers equation: we highlight here the work 
\cite{Bakhtin2016Kick} where the noise has a comparable structure to our case
and a recent article \cite{gu2021kpz} that uses a somewhat similar approach
to establish Gaussian fluctuations.

In view of the mentioned results it would appear particularly interesting to
extend the present study to
infinite volume. However, in such setting the results for the projective
component we mentioned 
are weaker and the picture we presented changes drastically, as the growth rate can be
super-exponential in cases of so called strongly catalytic environments.
In particular the regularity of the environment can determine the exact
super-exponential order of growth of the population (see e.g. \cite{Konig2020LongtimeAO, Chen2014} for the
time independent case or \cite{Gaertner2006Catalytic, Cranston2005Levy}, as
well as the monograph \cite[Chapter 8]{Konig2016}, for time dependent
problems). 

Next let us discuss our proof methods. We will make use of the spectral gap of $ z_{t} $ to derive the Furstenberg--Khasminskii formula for $
\lambda(\tau) $. We then study the
Lyapunov exponent near zero via a Taylor expansion of the latter formula, in the spirit of similar results for products of random
matrices close to the identity.
These Taylor expansions build on the convergence of certain stochastic
quantities, which give rise to the leading order terms. Much unlike the
finite-dimensional case, here, if the noise is rough, we use 
paracontrolled distributions \cite{GubinelliPerkowski2017KPZ} to identify
such terms.
The positivity of $ \lambda(\tau) $ in the bulk is proven instead with an
application of the Bou\'e--Dupuis formula, by constructing a suitable control.
We observe that there is a vast literature on lower bounds for Lyapunov
exponents for parabolic Anderson models. Although to the best of our knowledge there is no result that covers
our minimal assumption (the Lyapunov exponent is always positive, unless the
noise is constant in space), we believe that our proof, based in a different
way than above on the contraction property in the projective space and on a
perturbation expansion, is of independent interest. Eventually, the convergence for $ \tau \to
\infty $ is studied via Doob's $ H- $transform. In all cases, the backbone of
our analysis consists of moment estimates for the invariant projective
component: an essential tool in this respect is the use of certain quantitative lower
bounds to the fundamental solution of SPDEs, as were developed recently in the context of SDEs with singular drift
\cite{perkowski2020quantitative}, as well as some precise estimates on the
moments of the solution map to \eqref{eqn:time-Anderson}.

In conclusion, this work shows in a novel way how the tools presented in the cited
works can be extended to
obtain a strong quantitative control of the Lyapunov exponent, with a
particular attention to the interplay with the regularity of the noise and
theories from singular SPDEs.

\subsection{Structure of the article}

In Section \ref{sec:main-results} we will collect our main results,
Theorems~\ref{thm:lyapunov-smooth} and~\ref{thm:lyapunov-white}, after
having introduced the setting of the article. In the next three sections we
then prove the results concerning, respectively, the behaviour near zero, the
bulk behaviour and the averaging as $ \tau \to \infty $. The focus will be on
the crucial points of the proofs: we sometimes provide an intuitive proof in the simpler case of regular
noise and then concentrate on the added difficulties of white noise
(or sometimes we just treat the latter, more complicated case). We leave the most
technical calculations to later sections. In particular, in Section
\ref{sec:proj-invariant-measures} we recall the required properties of the
projective space and prove some moment estimates on the invariant measures.
These in turn build on a quantitative analytic bound that we present in
Section~\ref{sec:an-analytic-estimate}. In
Section~\ref{sec:stochastic-estimates} we prove the stochastic estimates
required for the Taylor expansion near $ \tau = 0 $ and in
Appendix~\ref{sec:paraproducts} we define some functions spaces and recall basic
constructions involving paraproducts.

\subsection{Acknowledgments}
The author gratefully acknowledges support through the Royal Society
grant $RP \backslash R1 \backslash 191065$. Many thanks
to Florian Bechtold, Ilya Goldsheid, Martin Hairer and Nicolas Perkowski for
some fruitful discussions, ideas and support.

\subsection{Notations}

Let $ \NN = \{ 0, 1, 2, \dots \} $. We will work on the torus $ \TT = \bigslant{\RR}{\ZZ}. $
We denote with $M_b (\TT)$ the space of measurable and bounded
functions $\varphi : \TT  \rightarrow \RR$ with the uniform norm:
\[ \| \varphi \|_{\infty} = \sup_{x \in \TT } | \varphi (x) | . \]
Then let $ \mS(\TT), \mS^{\prime}(\TT) $ be respectively the space of Schwartz
test functions (i.e. smooth) and its topological dual, the space Schwartz distributions.
For $ \alpha \in \RR, p \in [1, \infty] $, we denote with
\begin{equs}
\| \varphi \|_{\mC^{\alpha}} = \| \varphi \|_{B^{\alpha}_{\infty,
\infty}}, \qquad \| \varphi \|_{\mC^{\alpha}_{p}} = \| \varphi
\|_{B^{\alpha}_{p, \infty}},
\end{equs}
the spaces constructed in Appendix~\ref{sec:paraproducts}. For $ \alpha \in
(0, \infty] \setminus \NN $, $ \mC^{\alpha} $ coincides with the typical space of
$ \alpha - $H\"older continuous functions. For time-dependent functions $
\varphi \colon [0, T] \to X $ (for some $ T > 0$) and a Banach space $ X $, we introduce the
norms, for $ p \in [1, \infty]$ and the usual modification if $ p = \infty $:
\begin{equs}
\| \varphi \|_{L_{T}^{p}X} = \bigg( \int_{0}^{T} \| \varphi(s) \|_{X}^{p} \ud s
\bigg)^{\frac{1}{p}}.
\end{equs}
For a set $ \mX $ and two functions $ f, g \colon \mX \to \RR $ we write $
f \lesssim g$ if there exists a constant \( c > 0 \) such that $ f(x) \leqslant
c g(x) $ for all $ x \in \TT $.

\section{Main results}\label{sec:main-results}

As mentioned in the introduction, we will work in two distinct settings: in the first one every
realization of the noise, for fixed time, is assumed to be a function
(Assumption~\ref{assu:smooth-noise}), either piecewise constant or with some
continuity requirement; in the second one we consider space-white noise. We
start by describing precisely the first setting: we add the index $
\mathrm{stat} $ to indicate that the noise we describe is time independent.

\begin{assumption}[Regular noise]\label{assu:smooth-noise}
  Consider a probability space $(\Omega, \mathcal{F}, \PP)$ supporting a random function:
  \[ \xi_{\mathrm{stat}} : \Omega \to M_b (\TT), \]
	such that the following requirements are satisfied.
  \begin{enumerate}
    \item \textbf{(Centeredness \& nontriviality)} For every $x \in \TT$
    \[ \mathbf{E} [\xi_{\mathrm{stat}} (x)] = 0, \qquad \int_{\TT^{2}} \EE | \xi(x) -
\xi(y) |^{2} \ud x \ud y \in (0, \infty). \]
    \item \textbf{(Moment bound)} For some  $ \sigma > 0$
    \[ \mathbf{E} [e^{ \sigma \| \xi_{\mathrm{stat}} \|_{\infty}^{2}}] < \infty . \]
    \item \textbf{(Regularity)} One of the following two holds true:
\begin{enumerate}
\item There exists an $ \alpha \in (0, 1) $ such that 
\begin{equs}[eqn:holder-noise]
\EE \| \xi_{\mathrm{stat}} \|_{C^{\alpha}} < \infty.
\end{equs}
\item There exist deterministic disjoint intervals\footnote{Here an interval is a set of the
  form $ [a,b],[a,b),(a,b],(a,b) \subseteq \TT $ for some $ -1/2 \leqslant a <
b \leqslant  1/2 $.} $
A_{i} \subseteq \TT, \ i = 1, \dots,
\mf{n} $ for some $ \mf{n} \in \NN $ such that $ \TT = \bigcup_{i =
1}^{\mf{n}} A_{i} $ and
\begin{equs}[eqn:piecewise-noise]
\xi_{\mathrm{stat}} (x) = \sum_{i = 1}^{\mf{n}}  X_{i} 1_{A_{i}}(x),
\end{equs}
with $ X_{i} \colon \Omega \to \RR $ random variables.
\end{enumerate}
   \end{enumerate}
\end{assumption}

\begin{remark}\label{rem:on-reg-ass}
Most of the calculations we will present will work for general potential in $
L^{\infty} $ with exponential moments. The additional regularity assumption
will be used for certain stochastic estimates that are based on the
Feynman--Kac formula. The quadratic exponential moments are instead required to
study the behaviour in the bulk.
\end{remark}
In the second setting, below, we consider space white noise, as
an archetype for more irregular noises for which every realization is a
distribution rather than a function.
\begin{assumption}[White noise]\label{assu:white-noise}
  Consider a probability space $( \Omega,
  \mathcal{F}, \PP)$ supporting a random distribution
  $ \xi_{\mathrm{stat}} : \Omega \rightarrow \mS^{\prime} (\TT) $
	such that for any $ \varphi \in C^{\infty}(\TT) $ the random variable
$ \langle \xi, \varphi \rangle $ is a centered Gaussian with
covariance:
\begin{equs}
\EE \big[ \langle \xi, \varphi \rangle \langle \xi, \psi \rangle \big] =
\langle \varphi, \psi \rangle.
\end{equs}
\end{assumption}
Now we define our time-dependent potential.

\begin{definition}\label{def:tau-potential}
  Consider a probability space $ (\Omega, \mF, \PP) $ such that either
Assumption~\ref{assu:smooth-noise} or Assumption~\ref{assu:white-noise} is
satisfied and supporting a sequence $ \{ \xi^{i}_{\mathrm{stat}} \}_{i \in \NN}
$ of i.i.d. random fields $ \xi^{i}_{\mathrm{stat}} \colon \Omega \to
\mS^{\prime} ( \TT ) $ such
that
$\xi^{i}_{\mathrm{stat}} = \xi_{\mathrm{stat}}$ in distribution.
Then for any $\tau > 0$ define:
  $ \xi^{\tau} : \Omega \times [0, \infty) \rightarrow \mS^{\prime} ( \TT)$
  by
  \[ \xi^{\tau}  (\omega, t, \cdot) = \xi^{\lfloor t / \tau \rfloor}_{\mathrm{stat}}
     (\omega, \cdot). \]
\end{definition}
We will consider the following random Hamiltonians, naturally associated to \(
\xi^{\tau} \).

\begin{definition}\label{def:hamiltonians}
  In the same setting as the previous definition, define for every $\omega \in
  \Omega$:
  \[ H (\omega) = \kappa \Delta + \xi_{\mathrm{stat}} (\omega), \qquad H^i (\omega) =
     \kappa \Delta + \xi_{\mathrm{stat}}^i (\omega).\]
  In the case of space white noise (Assumption~\ref{assu:white-noise}) the
Hamiltonian is defined in the sense of Fukushima and Nakao \cite{Fukushima1976}.
\end{definition}
To study the longtime behaviour of
\eqref{eqn:time-Anderson} as $\tau$ we recall the Furstenberg formula for the
Lyapunov exponent $\lambda (\tau)$. Here we will write
$u_0 > 0$ if $u_0 (x) \geqslant 0$ for almost all $x \in
\TT$ and it holds $ 0 < \int_{\TT} u_{0}(x) \ud x < \infty$.

\begin{lemma}[Furstenberg formula]\label{lem:furst}
  Consider $\tau > 0$ and let $u$ be the solution
  to \eqref{eqn:time-Anderson} with initial condition $u_{0} \in
L^{1}(\TT), \ u_0 > 0$ and with $\xi = \xi^{\tau} $ as in Definition
  \ref{def:tau-potential}. Then there exists a $\lambda (\tau)
  \in \RR $ such that $ \PP$-almost surely the
  following limit holds (and is independent of the choice of $u_0$):
  \[ \lambda (\tau) = \lim_{t \rightarrow \infty} \frac{1}{t} \log \left( 
     \int_{\TT} u (t, x) \ud x \right) \in \RR. \]
In addition, consider $ H(\omega) $ as in
Definition~\ref{def:hamiltonians}. Then
  \begin{equation}
    \lambda (\tau) = \frac{1}{\tau} \int_{\Omega \times \Omega} \log
    \left( \int_{\TT} e^{\tau H (\omega)} (z_{\infty} (\tau, \omega')) (x)
    \ud x \right) \ud \PP (\omega) \ud \PP (\omega'),
    \label{eqn-furstenberg-formula}
  \end{equation}
  where $z_{\infty}$ is the projective invariant measure constructed in
Proposition~\ref{prop:existence-proj-invariant-measure}.

\end{lemma}
We provide a
proof of the lemma in Section~\ref{sec:proj-invariant-measures}. Instead now we
pass to the main results of this work.
The next result describes the behavior of the Lyapunov under the assumption
that the noise is regular. Here we denote with $ \sigma (H) $ the spectrum of
a closed operator $ H $. We note that all operators we consider have compact
resolvents, so $ \sigma(H) $ consists of the pure point spectrum of the
operator.

\begin{theorem}\label{thm:lyapunov-smooth}
Under Assumption~\ref{assu:smooth-noise} the map \( \lambda : (0, \infty)
\rightarrow \RR, \) with $ \lambda(\tau) $ as in Lemma~\ref{lem:furst}, is
continuous, strictly positive (i.e. $ \lambda(\tau)> 0 $ for all $ \tau > 0
$) and satisfies:
  \begin{enumerate}
    \item The limit $\lim_{\tau \rightarrow 0^+} \lambda (\tau) =
0$ holds, with:
    \[ \lim_{\tau \rightarrow 0^+} \frac{\lambda (\tau)}{\tau} =
\frac{1}{4} \int_{\TT} \int_{\TT} \EE | \xi_{\mathrm{stat}}(x) -
\xi_{\mathrm{stat}}(y) |^{2} \ud x \ud y  . \]
    \item For large values of $\tau$:
    \[ \lim_{\tau \rightarrow \infty} \lambda (\tau) =\mathbf{E}
       [\mathrm{max} \ \sigma (\Delta + \xi_{\mathrm{stat}})] \in (0, \infty) . \]
  \end{enumerate}
\end{theorem}

\begin{proof}
The continuity of $ \lambda $ follows from Lemma~\ref{lem:continuity} and the
positivity from Lemma~\ref{lem:positivity}. Then the
first statement is proven in Lemma~\ref{lem:regular-noise-small-time}, while
the second statement follows from Proposition~\ref{prop:averaging-invariant-measure}
as well as Lemma~\ref{lem:positive-avrg-lyap}.
\end{proof} \\
Instead in the case in which $ \xi $ has the
law of space white noise the behavior near zero follows a different power law.
\begin{theorem}\label{thm:lyapunov-white}
Under Assumption~\ref{assu:white-noise} the map \( \lambda : (0, \infty) \rightarrow \RR, \)
with $ \lambda(\tau) $ as in Lemma~\ref{lem:furst}, is continuous, strictly positive (i.e. $ \lambda(\tau)> 0 $ for all $ \tau > 0
$)  and satisfies:  
\begin{enumerate}
    \item The limit $\lim_{\tau \rightarrow 0^+} \lambda (\tau) = 0$ holds,
    with:
    \[ \lim_{\tau \rightarrow 0^+} \frac{\lambda (\tau)}{\sqrt{\tau} } =
    \sqrt{\frac{\pi}{\kappa}}. \]
    \item For large values of $\tau$:
    \[ \lim_{\tau \rightarrow \infty} \lambda (\tau) =\mathbf{E}
       [\mathrm{max} \ \sigma (\Delta + \xi_{\mathrm{stat}})] \in (0, \infty) . \]
  \end{enumerate}
\end{theorem}

\begin{proof}
The continuity of $ \lambda $ follows from Lemma~\ref{lem:continuity}.
Similarly to above, the
first statement follows from Lemma~\ref{lem:taylor-white-noise}, while
the second statement is a consequence of Proposition~\ref{prop:averaging-invariant-measure}
as well as Lemma~\ref{lem:positive-avrg-lyap}.
\end{proof}
In the next sections we will collect all the results needed to prove the
previous two claims.

\section{Behavior near zero}

This section is devoted to the proof of the small $ \tau $ behavior stated in
Theorems~\ref{thm:lyapunov-smooth}, \ref{thm:lyapunov-white}.
We start with the simpler setting of Assumption~\ref{assu:smooth-noise}.

\begin{lemma}\label{lem:regular-noise-small-time}
  For $\tau > 0$ and under Assumption~\ref{assu:smooth-noise} consider $\lambda (\tau)$ as in
Lemma~\ref{lem:furst}. Then
  \[ \lim_{\tau \rightarrow 0^+} \lambda (\tau) = 0, \qquad
     \lim_{\tau\rightarrow 0^+} \frac{\lambda
     (\tau)}{\tau} = \frac{1}{4} \int_{\TT} \int_{\TT} \EE | \xi_{\mathrm{stat}} (x) -
\xi_{\mathrm{stat}}(y) |^{2}\ud y \ud x.\]
\end{lemma}

\begin{proof}
  By Lemma \ref{lem:furst} we have for any $ \tau >0 $:
\begin{equs}
\lambda (\tau)  & =  \frac{1}{\tau} \int_{\Omega \times \Omega} \log \left( \int_{\TT}
      e^{\tau H (\omega)} (z_{\infty} (\tau,
      \omega')) (x) \ud x \right) \ud \PP (\omega) \ud
      \PP (\omega').
\end{equs}
  Using the definition of the semigroup $e^{\tau H
  (\omega)}$, we can rewrite the quantity inside the logarithm as:
\begin{equs}
\int_{\TT} e^{\tau H (\omega)} & (z_{\infty} (\tau, \omega')) (x) \ud x  \\
      &  = \int_{\TT}z_{\infty} (\tau, \omega') (x) \ud x + \int_{0}^{\tau}
      \int_{\TT} H (\omega) e^{s H (\omega)} (z_{\infty} (\tau ,
      \omega')) (x) \ud x \ud s \\
 & = 1 + \int_{0}^{\tau} \int_{\TT} H (\omega) e^{s H (\omega)} (z_{\infty} (\tau ,
      \omega')) (x) \ud x \ud s,
\end{equs}
where in the last step we used that $ \int_{\TT} z_{\infty}(\tau,
\omega^{\prime} , x) \ud x = 1 $ by construction (cf.
Proposition~\ref{prop:existence-proj-invariant-measure}). Now let us define
\begin{equs}
\zeta(\tau, \omega, \omega^{\prime}) & = \int_{0}^{\tau} \int_{\TT} H
     (\omega) e^{s H (\omega)} (z_{\infty} (\tau, \omega')) (x) \ud
     x \ud s \\
& =\int_{0}^{\tau} \int_{\TT} \xi_{\mathrm{stat}}(\omega, x)
      e^{s H (\omega)} (z_{\infty} (\tau, \omega')) (x) \ud x \ud s,
\end{equs}
where we used integration by parts to remove the Laplacian.
With this definition we observe that
\begin{equs}
\lambda(\tau) = \frac{1}{\tau} \int_{\Omega \times \Omega} \log{ \big( 1 +
\zeta(\tau, \omega, \omega^{\prime})\big)} \ud \PP(\omega)
\ud \PP(\omega^{\prime}),
\end{equs}
and now our result will follow by a Taylor expansion of the logarithm. The key
observation is that although for $ \tau \to 0, \ \zeta(\tau, \omega, \omega^{\prime}) \simeq \tau $, as
the potential is centered \(\int_{\Omega} \zeta(\tau, \omega, \omega^{\prime})
\ud \PP(\omega) \simeq \tau^{2}, \) so that we obtain a term of the correct
order (here we will use Lemma~\ref{lem:moment-estimate-for-zeta}). 

To rigorously motivate the Taylor expansion we start with some moment bounds
for $ \zeta $. By a maximum principle we can bound:
  \[ | \zeta (\tau, \omega, \omega') | \leqslant \tau \|
     \xi_{\mathrm{stat}} (\omega) \|_{\infty} e^{\tau \|
     \xi_{\mathrm{stat}} (\omega) \|_{\infty}} \| z_{\infty} (\tau,
     \omega') \|_{\infty} . \]
  In particular by Lemma~\ref{lem:moment-bound-invariant-measure} on the
projective invariant measure and the finite exponential moments of
Assumption~\ref{assu:smooth-noise}, we obtain that for any $p \geqslant 1$ there exists a
  $C (p) > 0$ such that
  \[ \sup_{\tau \in (0, 1)} \mathbf{E}^{\PP \otimes \PP}
     [| \zeta (\tau) / \tau |^p] < C (p) . \]
  We can conclude that the set $A_{\tau} = \left\{ | \zeta (\tau) |
  \leqslant \frac{1}{2} \right\} \subseteq \Omega \times \Omega$ satisfies for any $p \geqslant 1$
  \[ \PP (A_{\tau}) \geqslant 1 - 2^p \mathbf{E} [| \zeta
     (\tau) |^p] \geqslant 1 - 2^p C (p) \tau^p, \]
  so that
  \[ \lambda (\tau) = \frac{1}{\tau} \int_{A_{\tau}} \log
     \left( 1 + \zeta(\tau, \omega, \omega^{\prime}) \right) \ud
     \PP (\omega) \ud \PP (\omega') +\mathcal{O}
     (\tau^2), \]
  since
\begin{equs}
 \frac{1}{\tau} \int_{A_{\tau}^c} \log \left( 1 +
    \zeta(\tau, \omega, \omega^{\prime}) \right)  \ud \PP
    (\omega) \ud \PP (\omega^{\prime} ) \leqslant & \frac{1}{\tau} \PP (A_{\tau}^c)
    \mathbf{E}^{\PP \otimes \PP} [\log (1 + \zeta
    (\tau))^2]^{\frac{1}{2}}  \\
    \leqslant & \frac{1}{\tau} \PP
    (A_{\tau}^c)^{\frac{1}{2}} \mathbf{E}^{\PP \otimes  
    \PP} [\zeta (\tau)^2]^{\frac{1}{2}} \lesssim  \frac{1}{\tau} \tau^{\frac{p}{2}}
    \tau \lesssim \tau^2  
\end{equs}
  by choosing $p = 4$. Now we expand the logarithm to
  obtain:
  \[ \lambda (\tau) = \frac{1}{\tau} \int_{A_{\tau}}
     \zeta (\tau, \omega, \omega') - \frac{1}{2} \zeta^2 (\tau,
     \omega, \omega') + R (\tau, \omega, \omega') \ud \PP 
     (\omega) \ud \PP (\omega') +\mathcal{O} (\tau^2), \]
  where the rest $R$ is the Lagrange rest of for the Taylor expansion
  \[ R (\tau, \omega, \omega') = \frac{1}{3 \kappa (\tau,
     \omega, \omega')^3} (\zeta (\tau, \omega, \omega'))^3, \]
  for some $\kappa$ which, using the definition of $A_{\tau},$
  satisfies the bound
  $ \kappa (\tau, \omega, \omega') \in [1/2, 3/2].$
  Hence the rest term is controlled by
  $ | R (\tau, \omega, \omega') | 1_{A_{\tau}}(\omega, \omega^{\prime})  \lesssim | \zeta (\tau,
     \omega, \omega') |^3, $
  so that from the bound $\mathbf{E}^{ \PP \otimes \PP} | \zeta
  (\tau) |^3 \lesssim \tau^3$ we obtain:
  \begin{equs}
    \lambda (\tau) & =  \frac{1}{\tau} \int_{A_{\tau}}
    \zeta (\tau, \omega, \omega') - \frac{1}{2} \zeta^2 (\tau,
    \omega, \omega') \ud \PP (\omega) \ud \PP (\omega')
    +\mathcal{O} (\tau^2)\\
    & =  \frac{1}{\tau} \int_{\Omega \times \Omega} \zeta
    (\tau, \omega, \omega') - \frac{1}{2} \zeta^2 (\tau, \omega,
    \omega') \ud \PP (\omega) \ud \PP (\omega')
    +\mathcal{O} (\tau^2) .
  \end{equs}
  In the last step we followed similar arguments to those already used to
  control the integral on $A_{\tau}^c$. At this point we have reduced
  the problem to an estimate for the first two terms in the Taylor expansion
  of the logarithm. Now we apply Lemma \ref{lem:moment-estimate-for-zeta} to
  obtain
\begin{equs}
    \frac{1}{\tau}\int_{\Omega \times \Omega} \zeta
    (\tau, \omega, \omega') & - \frac{1}{2} \zeta^2 (\tau, \omega,
    \omega')   \ud \PP (\omega) \ud \PP (\omega') = \tau
\mf{r}(\tau)+ \mathcal{O} \left( \tau^{1 +
\gamma} \right),
\end{equs}
for some $ \gamma > 0 $, with $ \lim_{\tau \to 0^{+}} \mf{r}(\tau) =
\frac{1}{4} \int_{\TT} \int_{\TT} \EE | \xi_{\mathrm{stat}}(x) -
\xi_{\mathrm{stat}}(y) |^{2} \ud x \ud y $. This concludes the proof of the lemma.
\end{proof} \\
Next we treat the behaviour near zero in the white noise case. Here we will
skip some parts of the proof that are identical to the arguments we just used.
\begin{lemma}\label{lem:taylor-white-noise}
  For $\tau > 0$ and under Assumption~\ref{assu:white-noise} consider $\lambda (\tau)$ as in
Lemma~\ref{lem:furst}. Then
  \[ \lim_{\tau \rightarrow 0^+} \lambda (\tau) = 0, \qquad
     \lim_{\tau \rightarrow 0^+} \frac{\lambda
   (\tau)}{\sqrt{\tau}} = \sqrt{\frac{\pi}{\kappa} }. \]
\end{lemma}

\begin{proof}
  Once again we use Lemma \ref{lem:furst} to write, for any $\tau > 0$:
\begin{equs}
\lambda (\tau) & =  \frac{1}{\tau} \int_{\Omega
       \times \Omega} \log \left( \int_{\TT} e^{\tau H (\omega)}
       (z_{\infty} (\tau, \omega')) (x) \ud x \right) \ud
       \PP (\omega) \ud \PP (\omega')
\end{equs}
Following the calculations for the regular case, we rewrite the quantity in
the logarithm by integration by parts
\begin{equs}
\int_{\TT} [e^{\tau H(\omega)} z_{\infty} ( \tau, \omega^{\prime}) ] (x) \ud x & = 1+ \int_{0}^{\tau}
\int_{\TT}  H(\omega) [ e^{s H(\omega)} z_{\infty} (\tau, \omega^{\prime} )] (x) \ud x \ud s \\
& =1+ \int_{0}^{\tau} \int_{\TT} \xi_{\mathrm{stat}}(\omega, x)[ e^{s H(\omega)}z_{\infty} (
\tau, \omega^{\prime})] (x) \ud x \ud s.
\end{equs}
Hence let us define
\begin{equs}
\eta(\tau, \omega, \omega^{\prime}) = \int_{0}^{\tau} \int_{\TT}
\xi_{\mathrm{stat}}(\omega, x)[ e^{s H(\omega)} z_{\infty}(\tau, \omega^{\prime})]( x)\ud x \ud s.
\end{equs}
To define the product inside the integral, since $
\xi_{\mathrm{stat}} $ is a distribution, we can use the product estimates in
Lemma~\ref{lem:paraproduct-estimate} for any $ \ve> 0 $:
\begin{equs}
| \eta(\tau, \omega, \omega^{\prime}) | & \leqslant \int_{0}^{\tau} \|
\xi_{\mathrm{stat}} (\omega) e^{s H(\omega)} z_{\infty}(\tau, \omega^{\prime})  \|_{\mC^{-
\frac{1}{2} - \ve}} \ud s \\
& \lesssim \tau \| \xi_{\mathrm{stat}} (\omega) \|_{\mC^{-
\frac{1}{2} - \ve}} \sup_{0 \leqslant s \leqslant \tau}  \| e^{s H (\omega)} z_\infty (\tau, \omega^{\prime}) \|_{\mC^{\frac{1}{2} + 2 \ve}}.
\end{equs}
Now, for any $ p \geqslant 1 $ and $ \gamma \in (0,1) $ we can estimate, via
Lemma~\ref{lem:moment-bound-solution-1D-wn}:
\begin{equs}
\int_{\Omega} \sup_{0 \leqslant s \leqslant 1} \| e^{s H(\omega)} z_{\infty}(\tau, \omega^{\prime})
\|_{\mC^{\frac{1}{2} + \gamma}}^{p} \ud \PP(\omega) \lesssim \|
z_{\infty}(\tau, \omega^{\prime}) \|_{\mC^{\frac{1}{2} + \gamma}}^{p},
\end{equs}
and by Lemma~\ref{lem:moment-bound-invariant-msr-1D-1n} 
\begin{equs}
\sup_{\tau \in (0, 1)} \EE^{\PP}  \|
z_{\infty}(\tau)\|_{\mC^{\frac{1}{2} + \gamma}} < \infty,
\end{equs}
so that overall, for any $ \ve \in (0, 1/2) $
\begin{equs}
\sup_{\tau \in (0, 1)} \EE^{\PP \otimes \PP} | \eta(\tau)/ \tau |^{p}
& \lesssim \big( \EE \| \xi_{\mathrm{stat}} \|_{\mC^{-\frac{1}{2} - \ve
}}^{2p}\big)^{\frac{1}{2}} \big( \EE^{\PP \otimes \PP} \sup_{0 \leqslant s
\leqslant 1} \| e^{s H} z_{\infty}(\tau)
\|_{\mC^{\frac{1}{2} + 2 \ve}}^{2p} \big)^{\frac{1}{2}} < \infty.
\end{equs}
As a consequence of this bound, following the same arguments as in the proof of
Lemma~\ref{lem:regular-noise-small-time} to motivate the Taylor expansion, we
obtain that
\begin{equs}
\lambda(\tau) = \frac{1}{\tau} \int_{\Omega \times \Omega} \eta(\tau, \omega,
\omega^{\prime}) \ud \PP( \omega) \ud \PP ( \omega^{\prime}) + \mathcal{O}(\tau), \qquad \forall \tau \in (0,
1).
\end{equs}
To establish the limiting behavior for $ \tau \to 0 $ we have to further
decompose $ \eta $ into a leading term of order $ \simeq \tau^{\frac{3}{2}} $
and a better behaved rest term. For fixed $ \varphi \in \mC^{\frac{1}{2}
+ \gamma} $ (for some $ \gamma \in (1/2, 1) $) and $ s>0 $ we write the solution
$ e^{s H(\omega)} \varphi $ as
\begin{equs}
e^{s H (\omega)} \varphi = (e^{ \cdot H(\omega)} \varphi) \ppara
\mathcal{I}(\xi_{\mathrm{stat}}(\omega)) (s) + (e^{s H(\omega)} \varphi)^{\sharp},
\end{equs}
where the modified paraproduct $ \ppara $ is defined as in
\eqref{eqn:modified-paraproduct} (it is in many ways equivalent to the
paraproduct $ \para $, with the crucial difference that the commutator $
[\psi \ppara (\cdot), (\partial_{t} - \Delta)] $, for \(\psi\) fixed, satisfies some nice
regularity estimates) and $ \mI(\xi_{\mathrm{stat}})(t) = \int_{0}^{t} P_{t - s}
\xi_{\mathrm{stat}}\ud s $ as in Lemma~\ref{lem:stochastic-product}.
Then by Lemma~\ref{lem:para-rest-estimate} we have that for any $ \gamma \in
(1/2, 1) $ there exists a $ \delta > 0 $ such that
\begin{equs}
\EE^{\PP} \|( e^{s H} \varphi )^{\sharp}  - P_{s} \varphi
\|_{\mC^{\frac{1}{2} + \delta}} \lesssim s^{\frac{1}{2} + \delta}  \| \varphi \|_{C^{\frac{1}{2} + \gamma}}.
\end{equs}
In this way we find via Lemma~\ref{lem:moment-bound-invariant-msr-1D-1n}:
\begin{equs}
& \EE^{\PP \otimes \PP}  \bigg[  \int_{0}^{\tau} \int_{\TT} \xi_{\mathrm{stat}}
e^{s H} z_{\infty}(\tau) \ud s \ud x \bigg] \\
& \, = \EE^{\PP \otimes \PP} \bigg[ \int_{0}^{t}
\int_{\TT} \xi_{\mathrm{stat}} \big( e^{\cdot H} z_{\infty}(\tau) \ppara
\mathcal{I}(\xi_{\mathrm{stat}}) \big) + \xi_{\mathrm{stat}}
P_{s} z_{\infty}(\tau) + \xi_{\mathrm{stat}} \big( (e^{s H}
z_{\infty}(\tau))^{\sharp} - P_{s} z_{\infty}(\tau) \big)  \ud x
\ud s\bigg] \\
& \,=  \EE^{\PP \otimes \PP} \bigg[ \int_{0}^{t} \int_{\TT} \xi_{\mathrm{stat}}
\big( e^{\cdot H} z_{\infty}(\tau) \ppara
\mathcal{I}(\xi_{\mathrm{stat}}) \big) + \xi_{\mathrm{stat}}
P_{s} z_{\infty}(\tau)  \ud x \ud s \bigg] + \mathcal{O}(\tau^{1 + \frac{1}{2} + \delta}).
\end{equs}
The second term in the decomposition above, $ \xi_{\mathrm{stat}} P_{s}
z_{\infty}(\tau) $, vanishes on average by using the independence of $
\xi_{\mathrm{stat}} $ and $ z_{\infty}(\tau) $. We are left with studying, for
$ s \in [0, t] $ :
\begin{equs}
\EE^{\PP \otimes \PP} \bigg[ \frac{1}{ \sqrt{s}} \int_{\TT}
\xi_{\mathrm{stat}} \big( e^{ \cdot H}z_{\infty}(\tau) \ppara
\mathcal{I}(\xi_{\mathrm{stat}})\big)(s)  \ud x \bigg] =
\EE^{\PP \otimes \PP} \bigg[ \frac{1}{ \sqrt{s}}\int_{\TT}
\xi_{\mathrm{stat}} \reso \big( e^{\cdot H}z_{\infty}(\tau) \ppara
\mathcal{I}(\xi_{\mathrm{stat}}) \big)(s) \ud x \bigg],
\end{equs}
where we used the resonant product as in Lemma~\ref{lem:paraproduct-estimate},
observing that $ \int_{\TT} f (x) g (x) \ud x = \mF(fg)(0) = \int_{\TT}  (f
\reso g)\, (x) \ud x$ (with $ \mF $ the Fourier transform).
Now we can further decompose
\begin{equs}
\frac{1}{\sqrt{s}} \xi_{\mathrm{stat}} \reso & \big[ (e^{ \cdot H(\omega)}
z_{\infty}(\tau)) \ppara
\mathcal{I}(\xi_{\mathrm{stat}}) \big](s)
= (e^{s H(\omega)} z_{\infty}(\tau)) \cdot \bigg(  \xi_{\mathrm{stat}} \reso
\frac{\mathcal{I}(\xi_{\mathrm{stat}} )(s)}{ \sqrt{s} } \bigg) \\
 & \qquad \qquad + \frac{1}{\sqrt{s}} \Big(  C^{\reso}(\xi_{\mathrm{stat}},
e^{s H}z_{\infty}(\tau),
\mathcal{I}(\xi_{\mathrm{stat}}) ) + \xi_{\mathrm{stat}} \reso
C^{\ppara}(e^{\cdot H}z_{\infty}(\tau) ,
\mathcal{I}(\xi_{\mathrm{stat}}))(s) \Big),
\end{equs}
with the commutators
\begin{equs}
C^{\reso}(f,g,h) = f \reso (g \para h) - g(f \reso h), \qquad
C^{\ppara}(f, g)(s) =( f \ppara g ) \, (s) - f(s) \para g(s).
\end{equs}
For the first resonant product we can use \cite[Lemma 2.4]{GubinelliImkellerPerkowski2015} to bound, for
$\delta, \gamma \in (0, 1) $:
\begin{equs}
\| C^{\reso}(\xi_{\mathrm{stat}}, e^{s H} z_{\infty}(\tau),
\mathcal{I}(\xi_{\mathrm{stat}} )(s)) \|_{\mC^{\frac{1}{2} + \gamma -
4\delta }} \lesssim \| \xi_{\mathrm{stat}} \|_{\mC^{- \frac{1}{2}
- \delta}} \| e^{s H} z_{\infty}(\tau)\|_{\mC^{\frac{1}{2} + \gamma }} \| \mathcal{I}(\xi_{\mathrm{stat}})(s)
\|_{\mC^{\frac{1}{2} - 3\delta}},
\end{equs}
where the parameters must satisfy $ \frac{1}{2} + \gamma - 4\delta> 0 $
(which is true for $ \delta$ sufficiently small: in particular, we can
assume that $ \delta $ is the same as chosen in the calculations above). Now we can estimate
\begin{equs}[eqn:prf-wn-small-tau-Ixi]
\| \mathcal{I}(\xi_{\mathrm{stat}})(t) \|_{\mC^{\frac{1}{2} -3\delta} } &\lesssim
\int_{0}^{t} s^{- \frac{1}{2} + \delta} \| \xi_{\mathrm{stat}}
\|_{\mC^{-\frac{1}{2} - \delta}}
\ud s  \lesssim t^{\frac{1}{2}+ \delta} \| \xi_{\mathrm{stat}}
\|_{\mC^{-\frac{1}{2} - \delta}}
\end{equs}
so that overall for some $ \delta > 0 $
\begin{equs}
\EE^{\PP \otimes \PP}\bigg\vert \int_{\TT} C^{\reso}(\xi_{\mathrm{stat}}, e^{s H}
z_{\infty}(\tau),
\mathcal{I}(\xi_{\mathrm{stat}})(s)) \ud x
\bigg\vert \lesssim s^{\frac{1}{2}+ \delta}.
\end{equs}
As for the second commutator term, we use \cite[Lemma 2.8]{GubinelliPerkowski2017KPZ} and \eqref{eqn:prf-wn-small-tau-Ixi} to find:
\begin{equs}
\| C^{\ppara} ( e^{\cdot H}z_{\infty}(\tau) ,
\mathcal{I}(\xi_{\mathrm{stat}}))(s)
\|_{\mC^{1 + \gamma - 3 \delta}} & \lesssim \| e^{\cdot H} z_{\infty}(\tau)
\|_{C^{\frac{1/2+ \gamma}{2}}_{s}L^{\infty} \cap L^{\infty}_{s}
\mC^{\frac{1}{2} + \gamma} } \|
\mathcal{I} (\xi_{\mathrm{stat}})(s) \|_{\mC^{\frac{1}{2} - 3\delta}} \\
& \lesssim s^{\frac{1 }{2} + \delta}\| e^{\cdot H} z_{\infty}(\tau)
\|_{C^{\frac{1/2+ \gamma}{2}}_{s}L^{\infty} \cap L^{\infty}_{s}
\mC^{\frac{1}{2} + \gamma} } \| \xi_{\mathrm{stat}} \|_{\mC^{-
\frac{1}{2} - \delta}},
\end{equs}
so that (assuming $ \delta $ is sufficiently small) we obtain:
\begin{equs}
\EE^{\PP \otimes \PP}\bigg\vert \int_{\TT} \xi_{\mathrm{stat}} \reso C^{\ppara}(e^{\cdot
H}z_{\infty}(\tau),
\mathcal{I}(\xi_{\mathrm{stat}})
)(s) \ud x \bigg\vert \lesssim s^{\frac{1}{2} + \delta}.
\end{equs}
We have thus deduced the following estimate on the Lyapunov exponent
\begin{equs}
\lambda(\tau) = \frac{1}{ \tau}\int_{0}^{\tau} \sqrt{s}\, \EE^{\PP \otimes \PP} \bigg[
\int_{\TT} (e^{s H}z_{\infty}(\tau)) \cdot \Big(
\xi_{\mathrm{stat}} \reso \frac{\mathcal{I}(\xi_{\mathrm{stat}})(s)}{\sqrt{s}} \Big)  \ud x \bigg] \ud s +
\mathcal{O}(\tau^{\frac{1}{2} + \delta}).
\end{equs}
The proof of the lemma is concluded if we show that
\begin{equs}
   \EE^{\PP \otimes \PP}\bigg\vert\int_{\TT} (e^{s
H}z_{\infty}(\tau)) \cdot \Big(
\xi_{\mathrm{stat}} \reso \frac{\mathcal{I}(\xi_{\mathrm{stat}})(s)}{\sqrt{s}}
\Big)  \ud x - \sqrt{\frac{\pi}{\kappa}}
\bigg\vert = \mathcal{O}(s^{\delta}).
\end{equs}
Indeed, for $ \delta > 0 $
sufficiently small (and uniformly over $ \tau \in (0, \tau_{*}) $ and $ s \in (0, \tau) $):
\begin{equs}
  \EE^{\PP \otimes \PP} & \bigg\vert\int_{\TT}(e^{s
H}z_{\infty}(\tau)) \cdot \Big(
\xi_{\mathrm{stat}} \reso \frac{\mathcal{I}(\xi_{\mathrm{stat}} )(s)}{\sqrt{s}}
\Big)(x)  - \sqrt{\frac{\pi}{\kappa}} \,e^{s H}z_{\infty}(\tau) \, (x) \ud x 
\bigg\vert \\
& \lesssim \EE^{\PP \otimes \PP}  \Big[ \| e^{s H} z_{\infty}(\tau) \|_{\mC^{
4\delta}}^2 \Big]^{\frac{1}{2}}
\EE \bigg[ \bigg\| \Big( \xi_{\mathrm{stat}} \reso
\frac{\mathcal{I}(\xi_{\mathrm{stat}} )(s)}{\sqrt{s}} \Big)  -
\sqrt{\frac{\pi}{\kappa}} \bigg\|_{\mC^{- 3\delta}}^2 \bigg]^{\frac{1}{2}} =
\mathcal{O}(s^{\delta}),
\end{equs}
by Lemma~\ref{lem:stochastic-product}
by using for the last term and Lemmata~\ref{lem:moment-bound-solution-1D-wn},
\ref{lem:moment-bound-invariant-msr-1D-1n} for the first term. The last step is
then to show that
\begin{align*}
  \EE^{\PP \otimes \PP} \bigg\vert \int_{\TT} \sqrt{\frac{\pi}{\kappa}}
  e^{s H} z_{\infty} (\tau) (x) \ud x - \sqrt{\frac{\pi}{\kappa}} \bigg\vert =
  \mathcal{O}(s^{\delta}).
\end{align*}
Since $ \int_{\TT} P_{s} z_{\infty}(\tau) \, (x) \ud x = 1 $ it is enough to
prove, once more for $ \delta >0$ sufficiently small, that $ \EE^{\PP \otimes \PP} \| e^{s H} z_{\infty} (\tau) -
P_{s} z_{\infty}(\tau) \|_{\mC^{\delta}} = \mathcal{O}(s^{\delta}) $. Here we can
bound
\begin{align*}
   \| e^{s H} z_{\infty} (\tau) - P_{s} z_{\infty}(\tau) \|_{\mC^{\delta}}
   \leqslant \| (e^{s H} z_{\infty} (\tau))^{\sharp} - P_{s} z_{\infty}(\tau)
   \|_{\mC^{\delta}} + \| (e^{ \cdot H} z_{\infty} (\tau)) \ppara
   \mI(\xi_{\mathrm{stat}}) \, (s) \|_{\mC^{\delta}},
\end{align*}
so that the claimed result follows along the same arguments explained above, by
applying Lemma~\ref{lem:para-rest-estimate} and
\eqref{eqn:prf-wn-small-tau-Ixi}.
\end{proof} \\
To conclude this section we establish an estimate on the paracontrolled decomposition
of the solution used in the previous lemma.
\begin{lemma}\label{lem:para-rest-estimate}
Under Assumption~\ref{assu:white-noise}, fix $ T> 0 $ and $ \gamma \in (1/2, 1)
$. For any $ \varphi \in \mC^{\frac{1}{2} + \gamma} $ and $ t \in [0, T] $ define
\begin{equs}
(e^{t H (\omega)} \varphi)^{\sharp} = e^{tH(\omega)} \varphi - \big[ (e^{\cdot
  H(\omega)} \varphi) \ppara
\mathcal{I}(\xi_{\mathrm{stat}} (\omega)) \big] (t),
\end{equs}
where $ \mathcal{I}(f) = \int_{0}^{t} P_{t-s} f \ud s. $ Then there exists a $
\delta > 0 $ such that
\begin{equs}
\EE \Big[ \sup_{0 \leqslant t \leqslant T} t^{- \big( \frac{1}{2} + \delta \big)} \| (e^{t
H(\omega) } \varphi )^{\sharp} - P_{t} \varphi \|_{\mC^{\frac{1}{2} + \delta}}
\Big] \lesssim_{T} \| \varphi \|_{\mC^{\frac{1}{2} + \gamma}}.
\end{equs}
\end{lemma}

\begin{proof}
We have that
\begin{equs}
(e^{t H } \varphi)^{\sharp} = P_{t} \varphi + \int_{0}^{t} P_{t-s} \Big[
\xi_{\mathrm{stat}} \para
e^{s H} \varphi + \xi_{\mathrm{stat}} \reso e^{s H} \varphi + C(e^{\cdot H} \varphi, \mathcal{I}(\xi_{\mathrm{stat}}) ) \Big] \ud
s,
\end{equs}
with 
\begin{equs}
C(f, g) = (\partial_{t} - \Delta) ( f \ppara g) - f 
\para ( \partial_{t} - \Delta) g.
\end{equs}
Hence we obtain, for any choice of $ \delta > 0 $:
\begin{equs}
\| (e^{t H(\omega) } \varphi )^{\sharp}&  - P_{t} \varphi
\|_{\mC^{ \frac{1}{2} + \delta }} \\
& = \bigg\| \int_{0}^{t} P_{t-s}  \Big[ \xi_{\mathrm{stat}} \para e^{s H} \varphi + \xi_{\mathrm{stat}} \reso
e^{s H} \varphi +
C(e^{\cdot H} \varphi, \mathcal{I}(\xi_{\mathrm{stat}}) ) \Big] \ud s
\bigg\|_{\mC^{ \frac{1}{2} + \delta }} \\
& \lesssim t\| \xi_{\mathrm{stat}} \para e^{ \cdot H} \varphi +
\xi_{\mathrm{stat}} \reso e^{ \cdot H}
 \varphi \|_{L^{\infty}_{T} \mC^{\frac{1}{2} + \delta}}
+ \bigg\| \int_{0}^{t} P_{t-s} C( e^{\cdot H} \varphi, \mathcal{I}(\xi_{\mathrm{stat}})) \ud s
\bigg\|_{\mC^{\frac{1}{2} + \delta} }.
\end{equs}
Next fix some $ \ve \in (0, 1/2) $ such that $ \delta_{1} = \gamma -
\frac{1}{2} - \ve >0 $, then
\begin{equs}
\| \xi_{\mathrm{stat}} \para e^{ \cdot H} \varphi + \xi_{\mathrm{stat}} \reso
e^{ \cdot H} \varphi \|_{L^{\infty}_{T} \mC^{\frac{1}{2} + \delta_{1}}} \lesssim
\| \xi_{\mathrm{stat}} \|_{\mC^{-\frac{1}{2} - \ve}} \| e^{\cdot H} \varphi
\|_{L^{\infty}_{T} \mC^{\frac{1}{2} + \gamma}}.
\end{equs}
On the other hand, for the commutator term we can use
\cite[Lemma 2.8]{GubinelliPerkowski2017KPZ}, which guarantees that:
\begin{equs}
\| C(e^{\cdot H} \varphi, \mathcal{I}(\xi_{\mathrm{stat}})) \|_{L^{\infty} \mC^{ (1 - \ve)+
(\frac{3}{2} - \ve) -  2 }} \lesssim \| e^{\cdot H} \varphi
\|_{C^{\frac{1 - \ve}{2}}_{T} L^{\infty}\cap L^{\infty}_{T} \mC^{1 - \ve}} \|
\mathcal{I}(\xi_{\mathrm{stat}}) \|_{\mC^{\frac{3}{2} - \ve }}.
\end{equs}
Then define $ \delta_{2} = 2 \ve $ and choose $ \delta = \min \{ \delta_{1},
\delta_{2} \} $. We find via Lemma~\ref{lem:schauder-estiamtes}, collecting all
the previous estimates:
\begin{equs}
\| (e^{t H(\omega) } \varphi )^{\sharp}  - P_{t} \varphi
\|_{\mC^{ \frac{1}{2} + \delta }} & \lesssim \bigg( t  + \int_{0}^{t}(t-s)^{- 2\ve} \ud
s \bigg) \| e^{\cdot H} \varphi
\|_{C^{\frac{1 - \ve}{2}}_{T} L^{\infty}\cap L^{\infty}_{T}
\mC^{\frac{1}{2} + \gamma}} \|
\xi_{\mathrm{stat}} \|_{\mC^{-\frac{1}{2} - \ve}} \\
& \lesssim t^{1 - 2 \ve} \| e^{\cdot H} \varphi
\|_{C^{\frac{1 - \ve}{2}}_{T} L^{\infty}\cap L^{\infty}_{T}
\mC^{\frac{1}{2} + \gamma}} \| \xi_{\mathrm{stat}} \|_{\mC^{-\frac{1}{2} -
\ve}}.
\end{equs}
If $ \ve $ is chosen sufficiently small we have $ 1 - 2 \ve \leqslant
\frac{1}{2} + \delta $ (of course, this choice is far from optimal), so that
the result follows now by Lemma~\ref{lem:moment-bound-solution-1D-wn}, since $
\EE \big[ \| \xi_{\mathrm{stat}} \|_{\mC^{- \frac{1}{2} - \ve}}^{p} \big]<
\infty $ for any $ p \geqslant 1 $ (for example because $
\xi_{\mathrm{stat}} $ is the distributional derivative of a Brownian motion).


\end{proof}

\section{Bulk behavior}
In this section we study the behaviour of $ \lambda(\tau) $ in the bulk $
(0, \infty) $ and establish the positivity of the Lyapunov exponent. Our
argument is based on the Bou\'e--Dupuis formula and on a perturbation argument
to construct a suitable control.

\begin{lemma}\label{lem:positivity}
  Under Assumption~\ref{assu:smooth-noise} or Assumption~\ref{assu:white-noise}
  it holds that $ \lambda(\tau) > 0 $ for all $ \tau \in (0, \infty) $.
\end{lemma}

\begin{proof}
  Since the exact value of $ \tau, \kappa >0 $ is irrelevant for this discussion, let us
  assume that $ \tau = 1 $ and $ \kappa = \frac{1}{2} $ to simplify the
notation. Moreover, since the
  Lyapunov exponent does not depend on the initial condition, we will fix $
  u_{0} = 1 $. And finally, we will work only in the case $
\xi_{\mathrm{stat}} $ being space white noise, since the arguments we will use
will work identically also under Assumption~\ref{assu:smooth-noise}. We will use the Feynman--Kac representation of the solution:
  \begin{equs}
    u(n , x) = \EE^{\QQ}_{x} \Big[ \exp \Big( \sum_{i = 0}^{n-1}
	\int_{i}^{i +1} \xi_{\mathrm{stat}}^{i}(B_{n - s})
    \ud s\Big) \Big],
  \end{equs}
  where under $ \EE^{\QQ}_{x}$ the process $ B $ is a Brownian motion started
  in $ B_{0} = x $, while under $ \EE^{\QQ} $ the process $ B $ is a Brownian
  motion started in the uniform measure on $ \TT $.
  Then the  Bou\'e--Dupuis formula \cite{Boue1998} implies that, for any fixed realization of
  the noise $ \xi $:
  \begin{equs}
    \frac{1}{n} \log{ \int_{\TT} u (n, x) \ud x } \geqslant \frac{1}{n} \sup_{u \in
    \mathbb{H}^{n}_{s}} \EE^{\QQ} \bigg[ \sum_{i = 0}^{n-1} \int_{i}^{i +1}
      \xi_{\mathrm{stat}}^{i}(X_{n-s}) \ud s - \frac{1}{2} 
    \int_{0}^{n} | u_{s} |^{2} \ud s \bigg],
  \end{equs}
  with $ X_{t} = B_{t} + \int_{0}^{t} u_{s} \ud s $ and $
  \mathbb{H}^{n}_{s} $ the space of
  controls $ u $ adapted to the filtration of $ B $ such that $
  \int_{0}^{n} | u_{s} |^{2} \ud s < \infty$ and the law of $ X_{t} $ is smooth
  for any $ t \geqslant 0 $. We observe that as long as the
  law of $ X_{t} $ is smooth the above formula makes sense also for \(
  \xi_{\mathrm{ stat}} \) space-time white noise - so that the lower bound can
  be directly derived from the Bou\'e--Dupuis formula for smooth potentials $
  \xi $ by approximation. Now, if we take $ u = 0 $ we immediately obtain non-negativity of $ \lambda $, since
  \begin{equs}
    \lambda(\tau) & \geqslant
    \limsup_{n \to \infty} \frac{1}{n} \EE^{\QQ} \bigg[ \sum_{i = 0}^{n-1} \int_{i}^{i +1}
   \xi_{\mathrm{stat}}^{i}(B_{n - s}) \ud s \bigg] = \int_{\Omega} \int_{\TT} \xi_{\mathrm{stat}}(\omega, x)\ud x \ud  \PP(\omega) =0,
  \end{equs}
  where we used that the law of $ B_{t} $ converges to the uniform measure on
  $ \TT $ for $ t \to \infty $. To prove strict positivity we have to choose a
  slightly better control. Define, for $ \ve \in (0, 1) $ that will be chosen
  small later on and any $ i \in \NN $
\begin{equs}
  u^{n}_{s} =  \ve \partial_{x} (- \Delta)^{-2} \Pi_{\times} \xi^{n- i}_{\mathrm{stat}} (X_{s}) =:
\ve Z^{n-i}(X_{s}), & \qquad  s \in (i, i+1],
\end{equs}
with $ \Pi_{\times} \xi^{i}_{\mathrm{stat}} = \xi_{\mathrm{stat}} - \langle
\xi_{\mathrm{stat}},1 \rangle$. This definition means 
that $ X $ is the unique strong solution to the SDE $ \ud X_{s} = \ve Z^{n-i}(X_{s})
(s)\ud s + \ud B_{s} $ on $ (i-1, i] $. 

Our aim will be
to prove that the cost of the control and several error terms are of order $
\mO(\ve^{2}) $ (this is why the parameter $ \ve $ is multiplying the drift):
then a zeroth order term will vanish in the limit by averaging and we will be
left with a positive leading term of order $ \mO(\ve) $. 
With this aim in mind we rewrite the quantity in question as
\begin{equs}[eqn:pos-to-estmt]
\frac{1}{n} \sum_{i = 0}^{n-1} \int_{i}^{i + 1} \int_{\TT}
\xi^{i}_{\mathrm{stat}}(y) p^{n}(n-s, y) \ud y \ud s,
\end{equs}
where $ p^{n} $ is the solution to 
\begin{equ}
(\partial_{t} - \frac{1}{2} \Delta ) p^{n} = - \ve \partial_{x}(Z^{n-1-i} \cdot
p^{n}), \qquad \forall t \in (i, i+1],
\end{equ}
for all $ i \in \NN \cap [0,n] $, with initial condition $ p^{n}(0, y) = 1.$ Now we would like to use
that $ p^{n} $ converges to an invariant $ p^{\infty} $ as $ n \to \infty$. But
for clarity let us fix first some notation. We may assume that the probability
space is as in Assumption~\ref{assu:probability-space}: then we observe that $
\ola{p}^{n}(\omega, i, \cdot) := p^{n}(\omega, n-i, \cdot)$ is just a function of the future $(\omega_{j})_{j
\geqslant i} $: in short $\ola{p}^{n}( \omega, i , \cdot) = \ola{p}^{n}( (\omega_{j})_{j
\geqslant i}, i, \cdot) $. Under this time change we can view $
\ola{p}^{n}(\omega, t) $ as the solution to
\begin{equ}[eqn:fund-sol]
(\partial_{t} + \frac{1}{2} \Delta ) \ola{p}^{n} =  \ve \partial_{x}(Z^{i} \cdot
\ola{p}^{n}),
\end{equ}
on the interval $ (i, i + 1] $,
and we observe that this definition makes sense for all $ i \in \ZZ \cap (-
\infty, n-1] $, with terminal condition $ \ola{p}^{n}(n,y)=1 $.
Now the one-force-one-solution 
principle in \cite[Theorem 3.4]{Rosati2019synchronization} (applied in the
present time-reversed setting, observing that the solution map to
\eqref{eqn:fund-sol} is strictly positive) guarantees the
existence of a $ \ola{p}^{\infty}( (\omega_{j})_{j \geqslant 0}, x) $ such that $
\PP- $almost surely
\begin{equs}
  \limsup_{n \to \infty} \frac{1}{n} \log{ d_{H}( \ola{p}^{\infty}(
      (\omega_{j})_{j \geqslant 0}, \cdot), \ola{p}^{n}( (\omega_{j})_{j \geqslant 0}, 0,
  \cdot) )} < - \mf{c} < 0,
\end{equs}
for a deterministic constant $ \mf{c} > 0$ (here $ d_{H} $ is Hilbert's
projective distance as in Section~\ref{sec:proj-invariant-measures}). In particular we define $
\ola{p}^{\infty}( \omega, i , x) = \ola{p}^{\infty} ( (\omega_{j+ i})_{j \geqslant 0}, x)
$. We observe once more, to avoid confusion, that the time arrow is running backwards when
dealing with $ \ola{p}^{\infty} $, namely $ \ola{p}^{\infty}(\omega, 0) $ is
the evolution under \eqref{eqn:fund-sol} of $ \ola{p}^{\infty}(\omega, 1)$.

Now, the synchronization principle in \cite[Theorem
3.4]{Rosati2019synchronization} (which is the same as the
one-force-one-solution principle,
only seen for fixed initial time, instead of fixed terminal time) implies that for any fixed $ \ve \in (0, 1) $
and almost all $ \omega \in \Omega$ there exists a
random constant $ \mf{d}( \omega, \ve) \in \NN $ such that
\begin{equs}
  d_{H} ( \ola{p}^{\infty} ( \omega,i, \cdot), \ola{p}^{n}( \omega, i, \cdot) ) \leqslant
  \ve^{2}, \qquad \forall i \leqslant  n - \mf{d}( \vt^{n}\omega , \ve).
\end{equs}
We remark that the constant $ \mf{d} (\vt^{n} \omega , \ve) $ depends on $
n $, but its law is invariant.
Finally, let us denote with $ S^{\ve}_{t} (\omega_{i}) p $ the solution at time
$ t \in [-1, 0] $ to \eqref{eqn:fund-sol} with terminal condition $  S^{
\ve}_{0}(\omega_{i}) p = p$, so that $ \ola{p}^{n} (\omega, i) =
S^{\ve}_{-1} (\omega_{i}) \ola{p}^{n}(\omega, i + 1) $.

We can then rewrite \eqref{eqn:pos-to-estmt} with the notation we have
introduced so far to find
\begin{equs}
  \frac{1}{n} \sum_{i = 0}^{n-1} \int_{i}^{i + 1} \int_{\TT}
  \xi^{i}_{\mathrm{stat}}(y) p^{n}(n-s, y) \ud y \ud s = \frac{1}{n} \sum_{i =
  0}^{n-1} \int_{-1}^{0} \int_{\TT} \xi_{\mathrm{stat}}(\omega_{i}, y)
  S_{s}^{\ve} (\omega_{i}) [\ola{p}^{n}(i +1, \cdot) ] (y) \ud y \ud s.
\end{equs}
We would like to replace $ \ola{p}^{n} $ with \( \ola{p}^{\infty} \), so let us
further decompose the sum into
\begin{equs}
 \frac{1}{n} \sum_{i = 0}^{n-1} \int_{-1}^{0} \int_{\TT} & \xi_{\mathrm{stat}}(\omega_{i}, y)
 S_{s}^{\ve} (\omega_{i}) [\ola{p}^{\infty}(i +1) ] (y) \ud y \ud s
 \label{eqn:frst-trm}\\
& + \frac{1}{n} \sum_{i = 0}^{n-1} \int_{-1}^{0} \int_{\TT} \xi_{\mathrm{stat}}(\omega_{i}, y)
S_{s}^{\ve}(\omega_{i}) \bigg[ \ola{p}^{\infty}(i+1) \bigg(
\frac{\ola{p}^{n}(i + 1)}{ \ola{p}^{\infty}(i+1)} -1\bigg) \bigg] (y) \ud y \ud
s. \label{eqn:rest-term}
\end{equs}
We now have to treat all the error terms, as well as the cost of the control.

\textit{Step 1: Cost of the control.} For the cost of the control we have to
prove that
\begin{equs}
\limsup_{n \to \infty} \frac{1}{n} \EE \int_{0}^{n} | u^{n}_{s} |^{2} \ud s =
\mO(\ve^{2}).
\end{equs}
Indeed we have
\begin{equs}
\frac{1}{n} \EE \int_{0}^{n} | u^{n}_{s} |^{2} \ud s = \frac{1}{n} \sum_{i = 0}^{n -1}
\ve^{2} \int_{-1}^{0} \int_{\TT} | Z^{i}(\omega_{i}, x) |^{2}
S^{\ve}_{s}(\omega_{i}) \ola{p}^{n} ( (\omega_{j} )_{j \geqslant i}, i+1, x) \ud
x \ud s,
\end{equs}
so that we can immediately bound
\begin{equs}
\limsup_{n \to \infty} \frac{1}{n} \EE \int_{0}^{n} | u^{n}_{s} |^{2} \ud s =
\mO(\ve^{2}) \lesssim \ve^{2} \EE \| Z \|_{\infty}^{2}.
\end{equs}

\textit{Step 2: Rest term.} Now let us consider the term appearing in
\eqref{eqn:rest-term}. We further divide divide this sum into
\begin{equs}
  \frac{1}{n} &\sum_{i = 0}^{n - \mf{d}(\vt^{n} \omega, \ve)-1}  \int_{-1}^{0} \int_{\TT} \xi_{\mathrm{stat}}(\omega_{i}, y)
S_{s}^{\ve}(\omega_{i}) \bigg[ \ola{p}^{\infty}(i+1) \bigg(
\frac{\ola{p}^{n}(i + 1)}{ \ola{p}^{\infty}(i+1)} -1\bigg) \bigg] (y) \ud y \ud
s \\
& + \frac{1}{n} \sum_{i = n - \mf{d}(\vt^{n} \omega , \ve)-1}^{n-1}
\int_{-1}^{0} \int_{\TT} \xi_{\mathrm{stat}}(\omega_{i}, y)
S_{s}^{\ve} (\omega_{i}) \bigg[ \ola{p}^{\infty}(i+1) \bigg(
\frac{\ola{p}^{n}(i + 1)}{ \ola{p}^{\infty}(i+1)} -1\bigg) \bigg] (y) \ud y \ud
s.
\end{equs}
As for the first term, from the definition of $ \mf{d}(\vt^{n} \omega, \ve) $
we have that for $ i \leqslant n - \mf{d}(\vt^{n} \omega, \ve)-1 $:
\begin{equs}
  \bigg\| \frac{\ola{p}^{n}(i + 1)}{ \ola{p}^{\infty}(i+1)} -1
  \bigg\|_{\infty} \lesssim \ve^{2}.
\end{equs}
In addition, define for some parameters $ \alpha \in (1/2, 1), \delta > 0 $ with $
\frac{\alpha + 1}{2} + \delta < 1$
\begin{equs}
  \eta(\omega_{i}) = \sup_{s \in [-1, 0]} |s|^{ \frac{\alpha + 1}{2} + \delta}
      \sup_{p_{0} \in \mathrm{Pr} } \sup_{\ve \in (0, 1)}  \| S_{s}^{\ve}
    (\omega_{i})  p_{0} \|_{\mC^{\alpha}} \lesssim  \sup_{s \in [-1, 0]} |s|^{
    \frac{\alpha+1}{2} + \delta}
      \sup_{p_{0} \in \mathrm{Pr} } \sup_{\ve \in (0, 1)}  \| S_{s}^{\ve}
      (\omega_{i}) p_{0} \|_{\mC^{\alpha+1}_{1}},
\end{equs}
where the previous inequality follows by Besov embedding. Then we find:
\begin{equs}
  \frac{1}{n} \sum_{i = 0}^{n - \mf{d}(\vt^{n} \omega, \ve)-1} & \int_{-1}^{0} \int_{\TT} \xi_{\mathrm{stat}}(\omega_{i}, y)
S_{s}^{\ve}(\omega_{i}) \bigg[ \ola{p}^{\infty}(i+1) \bigg(
\frac{\ola{p}^{n}(i + 1)}{ \ola{p}^{\infty}(i+1)} -1\bigg) \bigg] (y) \ud y \ud
s \\
& \lesssim_{\alpha, \delta}\frac{1}{n} \sum_{i = 0}^{n -1}  \|
\xi_{\mathrm{stat}}(\omega_{i}) \|_{\mC^{- \alpha}} \eta (\omega_{i})
\bigg\|  \ola{p}^{\infty}(i+1) \bigg(
\frac{\ola{p}^{n}(i + 1)}{ \ola{p}^{\infty}(i+1)} -1\bigg)
\bigg\|_{L^{1}} \\
& \lesssim \ve^{2} \frac{1}{n} \sum_{i = 0}^{n -1}  \|
\xi_{\mathrm{stat}}(\omega_{i}) \|_{ \mC^{- \alpha}} \eta (\omega_{i}),
\end{equs}
and by the ergodic theorem the last quantity converges, as $ n \to \infty$, to $\ve^{2} \EE \big[ \|
\xi_{\mathrm{stat}} \|_{\infty} \cdot \eta \big]$: the latter is finite by
Lemma~\ref{lem:mmt-sol-div}, since $ \EE | \eta |^{\sigma} < \infty $ for any $ \sigma
\geqslant 1$.

Instead, for the rest of the sum we bound similarly to above
\begin{equs}
  \sum_{i = n - \mf{d}(\vt^{n} \omega , \ve) 
  - 1}^{n-1} & \int_{-1}^{0} \int_{\TT} \xi_{\mathrm{stat}}(\omega_{i}, y)
S_{s}^{\ve}(\omega_{i}) \bigg[ \ola{p}^{\infty}(i+1) \bigg(
\frac{\ola{p}^{n}(i + 1)}{ \ola{p}^{\infty}(i+1)} -1\bigg) \bigg] (y) \ud y \ud
s \\
& \lesssim \sum_{i = n - \mf{d}(\vt^{n} \omega , \ve) 
- 1}^{n-1} \| \xi_{\mathrm{stat}}(\omega_{i}) \|_{\mC^{- \alpha}} \eta
(\omega_{i}) \bigg\| 
\frac{\ola{p}^{n}(i + 1)}{ \ola{p}^{\infty}(i+1)} -1\bigg\|_{\infty}.
\end{equs}
Then as an application of the ergodic theorem (see e.g.
\cite[Lemma A.4]{Rosati2019synchronization}) we have
\begin{equs}
  \lim_{n \to \infty} \frac{1}{n} \bigg\{ \sum_{i = n - \mf{d}(\vt^{n} \omega , \ve) 
    - 1}^{n-1} \| \xi_{\mathrm{stat}}(\omega_{i}) \|_{\mC^{- \alpha}}
    \eta(\omega_{i}) \bigg\| 
\frac{\ola{p}^{n}(i + 1)}{ \ola{p}^{\infty}(i+1)} -1\bigg\|_{\infty}
\bigg\} = 0 
\end{equs}
if we can prove that
\begin{equs}
  \EE \bigg[\sum_{i =  - \mf{d}( \omega , \ve) 
    - 1}^{-1} \| \xi_{\mathrm{stat}}(\omega_{i}) \|_{\mC^{- \alpha}}
    \eta(\omega_{i}) \bigg\| 
\frac{\ola{p}^{0}(i + 1)}{ \ola{p}^{\infty}(i+1)}
-1\bigg\|_{\infty}\bigg] < \infty.
\end{equs}
In fact, by independence we can bound the above through
\begin{equs}
  \EE \big[ \| \xi_{\mathrm{stat}} \|_{\mC^{- \alpha}
  } \eta \big]\sum_{i = - \infty}^{-1}  & \EE \bigg[
  \bigg\|  \frac{\ola{p}^{0}(i + 1)}{
\ola{p}^{\infty}(i+1)} -1\bigg\|_{\infty}\bigg],
\end{equs}
where the first term is finite by Lemma~\ref{lem:mmt-sol-div}.
As for the last quantity inside the sum, denoting with $ \mu $ the contraction constant of
Theorem~\ref{thm:birkhoff}, we have, via Lemma~\ref{lem:hlbrt-dst-bd}:
\begin{equs}
 \EE \bigg\|  \frac{\ola{p}^{0}(i + 1)}{
\ola{p}^{\infty}(i+1)} -1\bigg\|_{\infty} & \lesssim \EE \bigg[ \sinh \Big(
d_{H}( \ola{p}^{0}(i+1), \ola{p}^{\infty}(i+1)) \Big) \bigg] \\
& \lesssim \EE \bigg[ \sinh \Big(  \prod_{j = 0}^{|i|} \mu (S_{-1}^{
\ve} (\omega_{-j})) d_{H}( \ola{p}^{0}(0), \ola{p}^{\infty}(0)) \Big) \bigg]\\
& \lesssim \EE\bigg[  \prod_{j = 0}^{|i|} \mu(S_{-1}^{
\ve} (\omega_{-j})) d_{H}( \ola{p}^{0}(0), \ola{p}^{\infty}(0)) \exp \Big(  d_{H}(
\ola{p}^{0}(0), \ola{p}^{\infty}(0)) \Big) \bigg] \\
& \lesssim \gamma^{|i|} \EE \Big[ \exp \left\{ 2 d_{H}( \ola{p}^{0}(0),
\ola{p}^{\infty}(0)) \right\} \Big],
\end{equs}
with $ \gamma = \EE [ \mu(S_{-1}^{\ve}) ] \in (0,1) $. Here we have used
independence as well as a Taylor expansion. To conclude, since $
\ola{p}^{0}(0) = 1$ we observe that
\begin{equs}
M := \EE \big[\exp \left\{ 2 d_{H}(1,
\ola{p}^{\infty}(0)) \right\}  \big] < \infty,
\end{equs}
by Lemma~\ref{lem:Lipschitz-Invariant}, provided $ \ve $ is sufficiently small.
We therefore deduce
\begin{equs}
  \EE  \bigg[\sum_{i =  - \mf{d}( \omega , \ve) 
    - 1}^{-1} \| \xi_{\mathrm{stat}}(\omega_{i}) \|_{\mC^{- \alpha} }
    \eta(\omega_{i}) & \bigg\| 
\frac{\ola{p}^{0}(i + 1)}{ \ola{p}^{\infty}(i+1)}
-1\bigg\|_{\infty}\bigg] \leqslant \sum_{i = 0}^{\infty} \gamma^{i} M \EE \big[ \| \xi_{\mathrm{stat}}
\|_{\mC^{- \alpha}} \eta \big] < \infty,
\end{equs}
which is the desired bound. 
Hence we have obtained that
\begin{equs}
\limsup_{n \to \infty} \frac{1}{n} \sum_{i = 0}^{n-1} \int_{-1}^{0} \int_{\TT} \xi_{\mathrm{stat}}(\omega_{i}, y)
S_{s}^{\ve}(\omega_{i}) \bigg[ \ola{p}^{\infty}(i+1) \bigg(
\frac{\ola{p}^{n}(i + 1)}{ \ola{p}^{\infty}(i+1)} -1\bigg) \bigg] (y) \ud y \ud
s = \mO(\ve^{2}).
\end{equs}

\textit{Step 3: First order term.} Finally, we consider \eqref{eqn:frst-trm}.
This term converges by the ergodic theorem, since $
\ola{p}^{\infty}(i+1) $ is invariant under $ S^{\ve}_{-1} (\omega_{i}) $ and independent of $ \omega_{i} $, to
\begin{equs}
\int_{\Omega \times \Omega} & \int_{-1}^{0} \int_{\TT} \xi_{\mathrm{stat}}  ( \omega)
S^{\ve}_{s}(\omega) [ \ola{p}^{\infty}(\omega^{\prime}, \cdot )]  \ud x \ud s \ud
\PP( \omega) \ud \PP(\omega^{\prime})  = \int_{\Omega }  \int_{-1}^{0}
\int_{\TT} \Pi_{\times} \xi_{\mathrm{stat}}  ( \omega)
S^{\ve}_{s}(\omega) [ \ola{q}]  \ud x \ud s \ud
\PP( \omega),
\end{equs}
with $ \Pi_{\times} \xi_{\mathrm{stat}} = \xi_{\mathrm{stat}} - \langle
\xi_{\mathrm{stat}}, 1 \rangle$ and $ \ola{q} (x) = \EE \ola{p}^{\infty}(x) $, assuming for the moment that
all the products are well-defined. Here we used that $ \EE \langle
\xi_{\mathrm{stat}}, 1 \rangle =0$ and $ \int_{\TT} S^{\ve}_{s}(\omega)
[\ola{q}] (x) \ud x =1$ since \eqref{eqn:fund-sol} is mass preserving.
Now, define for $ s \in [-1,0] $: $$ T^{\ve}_{s} (\omega) [q] = P_{|s|} q - \ve
\partial_{x} \int_{0}^{| s |} P_{| s | - r}
\big[ Z(\omega) \cdot
P_{r}[q] \big] \ud r. $$
Then we claim that and uniformly over all $ q \in \mathrm{Pr}$
$$ \EE \int_{-1}^{0} \| S_{s}^{\ve} q - T_{s}^{\ve} q
\|_{\mC^{\frac{3}{4}}_{1}} \ud s \lesssim \ve^{2}.$$ 
Here $ \frac{3}{4} $ is an arbitrary number larger that $ \frac{1}{2}
 $, so that the product with $ \xi_{\mathrm{stat}} $ is well posed.
To prove the above estimate we can use the Duhamel representation of the solution $
S^{\ve}_{s}, s \in [-1, 0]$, so that for sufficiently small $ \delta > 0 $ and
by Lemma~\ref{lem:schauder-estiamtes}
\begin{equs}
  \| S^{\ve}_{s} (\omega) q - T^{\ve}_{s} (\omega) q
\|_{\mC^{\frac{3}{4}}_{1}} & = \ve^{2} \bigg\|
\partial_{x} \int_{0}^{| s |} P_{| s |-r} \Big[ Z(\omega)  \int_{0}^{r} \partial_{x}
P_{r -z} \big[ Z (\omega) S^{\ve}_{-z} (\omega) q \big] \ud z \Big]\ud r
\bigg\|_{\mC^{\frac{3}{4}}_{1}} \\
& \lesssim \ve^{2} \int_{0}^{| s |} (| s | - r)^{- \frac{1}{2}(
\frac{7}{4} + \delta) } \|
Z(\omega) \|_{\infty} \bigg\|\int_{0}^{r} \partial_{x}
P_{r -z} \big[ Z (\omega) S^{\ve}_{-z} (\omega) q \big] \ud z \Big]
 \bigg\|_{L^{1}} \ud r \\
& \lesssim \ve^{2} \| Z(\omega) \|_{\infty}^{2}\int_{0}^{| s |} (| s | - r)^{-
\frac{1}{2}( \frac{7}{4} + \delta)} \int_{0}^{r} (r - z)^{- \frac{1}{ 2} - \delta}\ud z \ud r
\\
& \lesssim \ve^{2} \| Z(\omega) \|_{\infty}^{2},
\end{equs}
where we used that $ \| S^{\ve}_{s} (\omega) q \|_{L^{1}} = 1 $ for $ q \in
\mathrm{Pr}.$ So the claimed bound is proven.

\textit{Step 4: Conclusion.} Overall, we have obtained that
\begin{equs}
\lambda(\tau) \geqslant \int_{\Omega} \int_{0}^{1}\int_{\TT} \Pi_{\times}
\xi_{\mathrm{stat}}(\omega,x ) T^{\ve}_{s} (\omega) [ \ola{q}](x) \ud
x \ud s \ud \PP(\omega) + \mathcal{O}(\ve^{2}).
\end{equs}
Since $ \EE \xi_{\mathrm{stat}}(x) = 0 $, we can further reduce this, with the
definition of $ T_{s}^{\ve}(\omega) $, to obtain
\begin{equs}
\lambda(\tau) \geqslant - \ve \int_{\Omega} \int_{0}^{1}\int_{\TT} \Pi_{\times}
\xi_{\mathrm{stat}}(\omega,x ) \partial_{x} \int_{0}^{s} P_{s - r} [
Z(\omega) \cdot P_{r} \ola{q}](x)  \ud r \ud
x \ud s \ud \PP(\omega) + \mathcal{O}(\ve^{2}).
\end{equs}
In addition, by Lemma~\ref{lem:Lipschitz-Invariant} and
Lemma~\ref{lem:hlbrt-dst-bd} we have that
\begin{equs}
  \| \ola{q} - 1 \|_{\infty} \lesssim \ve,
\end{equs}
so that following similar calculations to the one above we finally conclude
\begin{equs}
  \lambda(\tau) & \geqslant \ve \EE \int_{0}^{1} \int_{0}^{s} \int_{\TT} \partial_{x}
  \Pi_{\times}\xi_{\mathrm{stat}}(x) P_{s - r}[Z] (x) \ud r  \ud s \ud x +
  \mO(\ve^{2}) \\
  & \geqslant \ve \EE \int_{0}^{1} \int_{0}^{s} \int_{\TT} | P_{(s-r)/2} \partial_{x} (- \Delta)^{ -1 } \Pi_{\times}
  \xi_{\mathrm{stat}}|^{2} (x) \ud r \ud s \ud x + \mO(\ve^{2}),
\end{equs}
where in the last line we used the definition of $ Z $ together with the fact
that the heat semigroup is selfadjoint. We observe that the average appearing
in the first order term is strictly positive both under
Assumption~\ref{assu:white-noise} and under Assumption~\ref{assu:smooth-noise}
(in the latter case, because the field is chosen to be non-trivial). Hence
sending $ \ve \to 0 $ proves the desired result.
\end{proof} \\
In the previous result we required an approximation of the invariant measure $
\ola{p}^{\infty} $ for small \( \ve \). This is the content of the next
result.

\begin{lemma}\label{lem:Lipschitz-Invariant}
  Under Assumption~\ref{assu:probability-space}, and in the setting of either
  Assumption~\ref{assu:smooth-noise} or Assumption~\ref{assu:white-noise}, consider $ \ola{p}^{\infty}$
  the invariant probability measure, in the sense of
  Proposition~\ref{prop:existence-proj-invariant-measure}, to the equation
  \begin{equs}[eqn:prob-dens-back]
    (\partial_{t} + \kappa \Delta ) \ola{p} = \ve \partial_{x} \big[ [ \partial_{x} (-
    \Delta)^{- 2} \xi_{\mathrm{stat}}^{\lfloor t \rfloor} ] \ola{p}\big], \qquad t \in \RR
  \end{equs}
  for $ \ve \in (0, 1) $. Then for every $ \sigma \geqslant 1 $ there exists an
  $ \ve_{0}(\sigma) \in (0, 1) $ such that
  for all \( \ve \in (0, \ve_{0}) \), with $ d_{H} $ the Hilbert distance as in
  Section~\ref{sec:proj-invariant-measures}:
  \begin{equs}
    \EE \big[ \sinh \{ \sigma d_{H} ( \ola{p}^{\infty}, 1)\}\big] \lesssim \ve.
  \end{equs}
\end{lemma}

\begin{proof}
  Note that we should expect that $ \ola{p}^{\infty} \to 1 $ as $ \ve \to 0 $, since the
  latter is the eigenfunction associated to the top eigenvalue of the heat
  equation. Therefore, let us denote with $ S_{s}^{\ve}(\omega), S_{s}= P_{-s} $ respectively the
  solution map to \eqref{eqn:prob-dens-back} and to the heat equation started
  at time $ 0 $ and computed at time $ s < 0 $.
  We will prove that
  \begin{equs}
    \EE \Big[ \exp \{\sigma d_{H}(\ola{p}^{\infty}, 1) \}\Big] -1
    \lesssim_{\sigma} \ve.
  \end{equs}
  Following the same arguments, one can prove that $ 1-\EE \big[ \exp \{- \sigma
  d_{H}(\ola{p}^{\infty}, 1 ) \} \big] \lesssim_{\sigma} \ve $, which together
with the previous bound implies the desired result. To obtain our bound we compute:
  \begin{equs}
    d_{H}( \ola{p}^{\infty}(\vt^{-1} \omega), 1)& = d_{H}( S_{-1}^{\ve}(\omega)
    \ola{p}^{\infty}(\omega), 1) \\
    & \leqslant d_{H} \big(S_{-1}^{\ve}(\omega)
    \ola{p}^{\infty}(\omega), S_{-1} \ola{p}^{\infty}(\omega) \big) +
    d_{H}\big(S_{-1} \ola{p}^{\infty}(\omega), S_{-1}1 \big).
  \end{equs}
  Now let $ \mu = \mu (S_{-1}) \in (0, 1) $ be the contraction constant of the
  heat semigroup in the Hilbert distance, as in Theorem~\ref{thm:birkhoff}. Then
  \begin{equs}
    d_{H}( \ola{p}^{\infty}(\vt^{-1} \omega), 1) \leqslant \frac{1}{1 - \mu}
    d_{H} \big(S_{-1}^{\ve}(\omega) \ola{p}^{\infty}(\omega), S_{-1}
    \ola{p}^{\infty}(\omega) \big).
  \end{equs}
  Now, from the definition of $ d_{H} $, it suffices to prove that
  \begin{equs}[eqn:mmt-to-prv]
    \EE \bigg[ \bigg\vert \max\frac{ S_{-1}^{\ve} \ola{p}^{\infty}}{
      S_{-1} \ola{p}^{\infty}} \bigg\vert^{\frac{\sigma}{1 - \mu}} \cdot  
    \bigg\vert \max\frac{ S_{-1} \ola{p}^{\infty}}{
  S_{-1}^{\ve} \ola{p}^{\infty}} \bigg\vert^{\frac{\sigma}{1 - \mu}} \bigg] -1 
    \lesssim_{\mu} \ve.
  \end{equs}
  We can then decompose, for $ - 1 \leqslant  s \leqslant 0 $,  $ S^{\ve}_{s} (\omega) [p] = S_{s}[p] +
  R_{s}^{\ve}(\omega)[p]$, with
  \begin{equs}
    R_{s}^{\ve}(\omega)[p] = - \ve \partial_{x} \int_{0}^{|s|} P_{|s| - r} \Big[  [\partial_{x} (-
    \Delta)^{- 2} \xi_{\mathrm{stat}}(\omega_{-1}) ] S_{-r}^{\ve}(\omega)[p]\Big]\ud r.
  \end{equs}
  For the first term we thus have
  \begin{equs}
    \frac{S^{\ve}_{-1} \ola{p}^{\infty}}{S_{-1} \ola{p}^{\infty}} \leqslant 1 +
    C_{1} \| R^{\ve}_{-1} \ola{p}^{\infty} \|_{\infty},
  \end{equs}
  where we used that $ \int_{\TT} \ola{p}^{\infty}(x) \ud x =1$ and $ C_{1} > 0
  $ is a constant such that $ \texttt{p}_{1}(x) \geqslant C_{1}, \
  \forall x \in \TT $ with $ \texttt{p}_{t}(x) $ the heat kernel at time $ t > 0$.
  For the second term a Taylor expansion guarantees, for any fixed $ \alpha \in
  (0, 1) $:
  \begin{equs}
    \frac{S_{-1} \ola{p}^{\infty}}{S_{-1}^{\ve}\ola{p}^{\infty}} \leqslant 1 +
    \| R^{\ve}_{-1} \|_{\infty} \frac{ \max S_{-1} \ola{
    p}^{\infty}}{ \big( \min S_{-1}^{\ve} \ola{p}^{\infty} \big)^{2} }
    \leqslant 1 + \| R^{\ve}_{-1} \ola{p}^{\infty} \|_{\infty} C_{1}
    C_{2} e^{\ve^{2} 2 C_{3} \| \partial_{x} (- \Delta)^{-2} \xi_{\mathrm{stat}}
    \|_{\mC^{\alpha}}^{2}},
  \end{equs}
  with $ C_{1} $ as above and $ C_{2},C_{3} > 0 $ deterministic so that the fundamental
  solution $ \Gamma_{s}(x, y) $ to \eqref{eqn:prob-dens-back} (i.e. with initial
  condition $ \Gamma_{0}(x, y) = \delta_{x}(y) $ ) satisfies $
  \Gamma_{-1}(x, y) \geqslant C_{1}e^{-\ve^{2} C_{3} \| \partial_{x} (- \Delta)^{-2} \xi_{\mathrm{stat}}
  \|_{\mC^{\alpha}}^{2}}$. That such constants exist follows as in the proof of
Lemma~\ref{lem:moment-bound-invariant-msr-1D-1n}, via the results of
\cite{perkowski2020quantitative}. Then we can estimate \eqref{eqn:mmt-to-prv} as follows, for some
  deterministic $ C > 0 $
  \begin{equs}
    \EE \bigg[  \bigg\vert &\max\frac{ S_{-1}^{\ve} \ola{p}^{\infty}}{
      S_{-1} \ola{p}^{\infty}} \bigg\vert^{\frac{\sigma}{1 - \mu}} \cdot  
    \bigg\vert \max\frac{ S_{-1} \ola{p}^{\infty}}{
  S_{-1}^{\ve} \ola{p}^{\infty}} \bigg\vert^{\frac{\sigma}{1 - \mu}} \bigg] -1  \\
  & \leqslant \EE \bigg[ \bigg\vert 1 + C \| R^{\ve}_{-1} \ola{p}^{\infty}
    \|_{\infty} \Big( 1 + e^{\ve^{2} C\| \partial_{x} (- \Delta)^{-2} 
  \xi_{\mathrm{stat}}  \|_{\mC^{\alpha}}^{2}}  \Big) \bigg\vert^{\frac{2 \sigma}{1 -
\mu}} \bigg] -1 \\
& \lesssim \EE \Big[  \| R^{\ve}_{-1} \ola{p}^{\infty}
  \|_{\infty} \Big( 1 +\| R^{\ve}_{-1} \ola{p}^{\infty}
  \|_{\infty}\Big)^{M} \Big( 1 + e^{\ve^{2} C\| \partial_{x} (- \Delta)^{-2} 
\xi_{\mathrm{stat}}  \|_{\mC^{\alpha}}^{2}}  \Big)^{M}\Big] ,
\end{equs}
for some $ M (\mu, \sigma) > 0 $, where in the last step we used a Taylor expansion.
Now we observe that via Lemma~\ref{lem:schauder-estiamtes} and Lemma~\ref{lem:mmt-sol-div}, for any $ \delta > 0 $
sufficiently small
\begin{equs}
  \| R^{\ve}_{-1}(\omega) \ola{p}^{\infty} \|_{\infty} & \lesssim \ve \int_{0}^{1} \Big\|
  P_{1 - r} \Big[ [\partial_{x} (-
    \Delta)^{- 2} \xi_{\mathrm{stat}}(\omega_{-1}) ]
    S_{-r}^{\ve}(\omega_{- 1})[\ola{p}^{\infty}(\omega)]
  \Big] \Big\|_{\mC^{1 + \delta}} \ud r \\
& \lesssim \ve \int_{0}^{1} (1 -r)^{- \frac{1 + 2 \delta}{2}}  \big\|
  [\partial_{x} (-
    \Delta)^{- 2} \xi_{\mathrm{stat}}(\omega_{-1}) ]
    S_{-r}^{\ve}(\omega)[\ola{p}^{\infty}] \big\|_{\mC^{- \delta}}\ud r \\
& \lesssim \ve \int_{0}^{1} (1 -r)^{- \frac{1 + 2 \delta}{2}}  \big\|
  [\partial_{x} (-
    \Delta)^{- 2} \xi_{\mathrm{stat}}(\omega_{-1}) ]
    S_{-r}^{\ve}(\omega_{-1})[\ola{p}^{\infty}(\omega)] \big\|_{\mC^{ 1- \delta}_{1}}\ud r
\\
& \lesssim \ve \int_{0}^{1} (1 -r)^{- \frac{1 + 2 \delta}{2}} \|
\xi_{\mathrm{stat}} (\omega_{-1}) \|_{\mC^{- \frac{1}{2} - \delta}}
s^{- \frac{1 - 3 \delta}{2}} \eta (\omega_{-1})\ud r \lesssim \ve \eta (\omega_{-1})
\| \xi_{\mathrm{stat}} (\omega_{-1}) \|_{\mC^{- \frac{1}{2} - \delta}}
\end{equs}
where in the third line we used Besov embeddings. Hence an application of
Fernique's theorem as well as  Lemma~\ref{lem:mmt-sol-div} guarantees the
required estimate.
\end{proof}\\
Next we show a moment estimate for the solution map to
\eqref{eqn:prob-dens-back}.

\begin{lemma}\label{lem:mmt-sol-div}
  Let $ S_{s}^{\ve} p_{0} $ be the solution to \eqref{eqn:prob-dens-back} with
  initial condition $ p_{0} \in \mathrm{Pr} $ (defined as in
  Section~\ref{sec:proj-invariant-measures}) at time $ s < 0 $. Then for any $ \sigma
  \geqslant 1 $ and $ \alpha \in (0, 2), \delta > 0 $ such that $
  \frac{\alpha}{2} + \delta < 1 $
  \begin{equs}
    \EE \Big[ \Big(  \sup_{s \in [-1, 0]} |s|^{ \frac{\alpha}{2} + \delta}
      \sup_{p_{0} \in \mathrm{Pr} } \sup_{\ve \in (0, 1)}  \| S_{s}^{\ve} p_{0}
  \|_{\mC^{\alpha}_{1}} \Big)^{\sigma} \Big] < \infty .
  \end{equs}
\end{lemma}

\begin{proof}
  First, we observe that by mass conservation $ S^{\ve}_{s} p_{0} \in
  \mathrm{Pr} $ for all $ s < 0 $. The parameter $ \delta > 0 $ is arbitrary
  small and chose only because we will embed $ L^{1} \subseteq \mC^{- \delta}_{1} $.  Now, let us first assume that $ \alpha + \delta< 1 $, so
  that using Duhamel's formula and
  Lemma~\ref{lem:schauder-estiamtes} we obtain
  \begin{equs}
    \| S^{\ve}_{s} p_{0} \|_{\mC^{\alpha}_{1}} & \lesssim | s |^{-
    \frac{\alpha + \delta}{2}} \| p_{0} \|_{L^{1}}+ \int_{0}^{| s |}(| s | - r )^{-
    \frac{\alpha + 1 + \delta}{2}} \| [ \partial_{x} \Delta^{- 2}
    \xi_{\mathrm{stat}} ] S^{\ve}_{-r} p_{0} \|_{\mC^{- \delta}_{1}} \ud r \\
    & \lesssim |s|^{-\frac{\alpha+ \delta}{2} } + \| \xi_{\mathrm{stat}}
    \|_{\mC^{-1}},
  \end{equs}
  where the regularity $ -1 $ far from an optimal choice, but the associated
  norm is finite under both Assumption~\ref{assu:smooth-noise} and
  Assumption~\ref{assu:white-noise}.
  Now one can iterate the bound to obtain the result for any $ \alpha \in
  (0, 1) $:
  \begin{equs}
    \sup_{s \in [-1, 0]} |s|^{ \frac{\alpha}{2} + \delta}
      \sup_{p_{0} \in \mathrm{Pr} } \sup_{\ve \in (0, 1)} \| S^{\ve}_{s} p_{0}
      \|_{L^{1}} \lesssim 1 + \| \xi_{\mathrm{stat}} \|_{\mC^{-1}}^{2}, 
  \end{equs}
  which implies the desired result.
\end{proof} \\
To conclude this section, let us note the following elementary result. 
\begin{lemma}\label{lem:hlbrt-dst-bd}
  Consider the distance $ d_{H} $ as in
  Section~\ref{sec:proj-invariant-measures} and let $ \sinh(x) =
  \frac{1}{2} (e^{x} - e^{- x}) $. Then, for any $ f, g \in
  \mathrm{Pr}$:
  \begin{equs}
    \| f/g -1 \|_{\infty} \leqslant 2 \sinh (d_{H}(f, g))
  \end{equs}
\end{lemma}

\begin{proof}
  We have the upper bound
    $\frac{f(x)}{g(x)} - 1 \leqslant e^{ \log{ \big( \max \frac{f}{g}} \big) }-1 \leqslant
    e^{d_{H}(f, g)} -1,$
  and the lower bound
    $1 - \frac{f(x)}{g(x)} \leqslant 1 - e^{ \log{ \big( \min \frac{f}{g}
    \big)}} \leqslant 1 - e^{- d_{H}(f, g)}.$
  Combining the two bounds proves the claim.
\end{proof}

\section{Behavior near $\infty$}
In this section we provide a short proof of the convergence for $ \tau \to
\infty $ described in the main results. A key tool will be Doob's $
H$-transform, which has its roots in the Krein-Rutman theorem. 

\begin{lemma}\label{lem:eigenfunction}
Under Assumption~\ref{assu:smooth-noise} or Assumption~\ref{assu:white-noise}
consider, for $ \omega \in \Omega, $ $ H(\omega) $ as in
Definition~\ref{def:hamiltonians}. Then there exists a unique $ \psi (\omega)
\in C(\TT) $ with $ \psi (\omega, x) > 0 , \ \forall x \in \TT $ and $ \int_{\TT}
\psi( \omega, x) \ud x =1 $ such that for some $ \zeta(\omega) \in \RR $
\begin{equs}
e^{t H (\omega)} \psi (\omega) = e^{ t \zeta(\omega) } \psi (\omega), \qquad
\forall t > 0.
\end{equs}
In particular, $ \zeta(\omega) = \max \sigma(H (\omega)) $. Finally, there
exists an \( \alpha > 1 \) so that $ \| \psi(\omega) \|_{\mC^{\alpha}} < \infty $ \(
\PP- \)almost surely.
\end{lemma}

\begin{proof}
This is a simple consequence of the Krein-Rutman theorem and the strong maximum
principle for parabolic equations. In fact, under both possible
assumptions, for fixed $ \omega $, $ s \mapsto e^{s H (\omega)} $ is a
compact semigroup on $ C (\TT) $ (see e.g.
Lemma~\ref{lem:moment-bound-solution-1D-wn} for the white noise case, which
implies also the required regularity estimate).
\end{proof} \\
We can use the eigenvalue--eigenfunction pair $(\zeta, \psi)$ as just
constructed to introduce Doob's H-transform.
\begin{lemma} \label{lem-doobs-H-transform}
Under Assumption~\ref{assu:smooth-noise} or Assumption~\ref{assu:white-noise}
consider the pair \( (\zeta, \psi) \colon \Omega \to \RR \times C(\TT)\) as
defined in Lemma~\ref{lem:eigenfunction}. One can decompose the semigroup $
e^{t H} $ as
 \[ e^{t H} u_0 = e^{t \zeta } \psi e^{t \overline{H}} (u_0 / \psi), \]
  where $e^{t \bar{H}}$ is the semigroup associated to the Hamiltonian
  \[ \overline{H} = \Delta \varphi + 2 \frac{\nabla \psi}{\psi} \nabla \varphi . \]
\end{lemma}

\begin{proof}
First we observe that by Lemma~\ref{lem:eigenfunction}, since $ \| \psi
\|_{\mC^{\alpha}} < \infty $ for some $ \alpha > 1 $ and $ \psi > 0 $, the
definition of $ \overline{H} $ makes sense. The derivation of the $
H-$transform is classical, but we provide the salient points for the reader.
  For any smooth $\varphi$ we have, from the definition of $ \psi$:
  \begin{equs}
    \frac{1}{\psi} [H - \zeta] (\varphi \cdot \psi) & = 
    \frac{1}{\psi} (\Delta + \xi - \zeta) (\varphi \cdot \psi )=  \Delta \varphi + \frac{\varphi}{\psi}(\Delta + \xi - \zeta) \psi + 2 \frac{\nabla
    \psi}{\psi} \cdot \nabla \varphi\\
    & =  \Delta \varphi + 2 \frac{\nabla \psi}{\psi} \nabla \varphi,
  \end{equs}
  meaning that
  \[ \overline{H} = \frac{1}{\psi} [H - \zeta] \psi = M_{\psi^{-1}} [H - \zeta]
M_{\psi}, \]
where $M_{\psi}$ is the operator defined by point-wise multiplication with $\psi$.
  In particular, we have the decomposition
  \[ e^{t H} u_0 = e^{t \zeta} e^{t M_{\psi} \bar{H} M_{\psi^{- 1}}}
     u_0, \]
  Now since $M_{\psi^{- 1}} = M_{\psi}^{- 1}$ commutes with $ \overline{H} $,
we eventually find
  \begin{equs}
    e^{t H} u_0  =  e^{t \zeta} e^{t M_{\psi} \overline{H}  M_{\psi^{-
    1}}} u_0  =  e^{t \zeta} M_{\psi} e^{t \overline{H}} M^{- 1}_{\psi} u_0  =
e^{t \zeta} \psi e^{t \bar{H}} (u_0 / \psi),
  \end{equs}
which concludes the proof.
\end{proof}\\
The next result is a moment estimate on the Hilbert distance (cf.
Section~\ref{sec:proj-invariant-measures}) between $ \psi $ and $
1 $, the latter intended as the constant function.
\begin{lemma}\label{lem:bdd-avrg-mu}
  In the same setting as above, define
  \[ \mu (\omega) = \log \Big( \max_{x \in \TT } \psi(\omega, x) \Big) - \log \Big( \min_{x \in
     \TT} \psi (\omega, x) \Big) = d_{H}(\psi (\omega), 1 ) \in (0, \infty) . \]
  Then
  \[ 0 <\mathbf{E} [\mu]  < \infty. \]
\end{lemma}
\begin{proof}
Consider $v = \log (\psi) .$ Since $\psi$ solves $  \partial_x^2 \psi + \xi
\psi - \zeta \psi =  0$, we obtain that $v$ solves
  \begin{equ}
    \partial_x^2 v + \xi - \zeta =  - (\partial_x v)^2 .
  \end{equ}
Now we view $ v $ as a periodic function on $ \RR $. We can choose $ x_{0} \in
\RR$ so that $ \partial_{x} v (x_{0}) = 0 $ and for every $ x \in \TT $ there exists a $
z_{+}(x) \in \RR , z_{+}(x) > x_{0} $ so that $ \partial_{x} v (x) =
\partial_{x} v (z_{+}(x)) $, and we can bound $ | z_{+} (x) | \leqslant c $ for all
$ x $, for some $ c > 0 $. We find:
  \begin{equs}
    \partial_x v (x)  =  \int_{x_{0}}^{z(x)} \partial^2_x v (y) \ud y =
\int_{x_{0}}^{z(x)} - \xi (y) - | \partial_x v |^2 (y) + \zeta \ud y
\lesssim_{c} \| \Xi \|_{\infty} + |\zeta|,
  \end{equs}
with $ \Xi $ a primitive of $ \xi $ with $ \Xi (x_{0}) = 0 $.
  Similarly we also find a $ z_{-}(x) < x_{0} $ such that $
\partial_{x} v (x) = \partial_{x} v(z_{-}(x)) $, implying:
  \begin{equs}
    - \partial_x v (x)  =  \int_{z_{-}(x)}^{x_{0}} \partial^2_x v (y) \ud y =
\int_{z_{-}(x)}^{x_{0}}  - \xi (y) - | \partial_x v |^2 (y) + \zeta 
    \ud y  \lesssim_{c} \| \Xi \|_{\infty} + |\zeta|.
  \end{equs}
  Now we can bound
  \[ \max_{x \in \TT} v (x) - \min_{x \in \TT} v (x) \leqslant \|
     \partial_x v \|_{\infty} \lesssim | \zeta | + \| \Xi \|_{\infty}
     . \]
To conclude we have to guarantee that $ \EE \big[ | \zeta | + \| \Xi
\|_{\infty} \big] < \infty $. Clearly the second term is bounded (under
Assumption~\ref{assu:white-noise} $ \Xi $ is a Brownian motion). That $ \EE |
\zeta | < \infty$ follows, under Assumption~\ref{assu:white-noise}, from
Corollary~\ref{cor:bound-eigenvalue} (under Assumption~\ref{assu:smooth-noise}
one can use a simpler argument, through a maximum principle).
\end{proof} \\
The following result establishes the behaviour of $ \lambda (\tau)$ for
large $\tau$.

\begin{proposition}
In the setting of Lemma~\ref{lem:furst}, with $ \mu $ as in
Lemma~\ref{lem:bdd-avrg-mu} and $ \zeta $ as in Lemma~\ref{lem:eigenfunction} we have:
  \begin{equation}
   \EE [\zeta] \geqslant \lambda(\tau) \geqslant \liminf_{t \rightarrow \infty} \frac{1}{t} \log
(\min_{x \in \TT} u (t, x)) \geqslant \mathbf{E} \left[\zeta - \frac{\mu}{\tau}
    \right]. \label{eqn:bound-lyapunov}
  \end{equation}
  In particular
  \[ \lim_{\tau \rightarrow \infty} \lambda (\tau) =\mathbf{E}
     [\zeta] =\mathbf{E} [\max \sigma (\Delta + \xi_{\mathrm{stat}})]. \]
\end{proposition}

\begin{proof}
The proof follows by representing the solution via Doob's $H-$transform and an
  application of maximum principles and the ergodic theorem. We indicate with
$ (\zeta^{i}, \psi^{i} ) $ the eigenvalue-eigenfunction pair as in
Lemma~\ref{lem:eigenfunction}, associated to the Hamiltonian $ H^{i} $ as
\ref{def:hamiltonians}. Furthermore, we restrict to considering the limit $
\lim_{n \to \infty} \frac{1}{n \tau} \min_{x \in \TT} u (n \tau, x)  $. Then
extension to arbitrary $ t $ is straightforward. 
 For $n \in \NN$, using Doob's transformation as
  in Lemma \ref{lem-doobs-H-transform} we can represent $u (n \tau,
  x)$ by
  \[ u (n \tau, x) = \left[ e^{\tau \sum_{i = 1}^n \zeta^i} \psi^n e^{\tau
     \bar{H}^n} (\psi^{n - 1} / \psi^n) e^{\tau \bar{H}^{n - 1}} \cdots
     e^{\tau \bar{H}^2} (\psi^1 / \psi^2) e^{\tau \bar{H}^1} (u_0 / \psi^1)
     \right] (x), \]
  with
  \[ \bar{H}^i (\varphi) = \Delta \varphi + 2 \frac{\nabla \psi^i}{\psi^i}
     \nabla \varphi . \]
  Now for any $i \in \NN$ the semigroup $e^{t \bar{H}^i}$ is strictly positive:
  \[ \bigg( \varphi (x) \geqslant 0 \quad \forall x \in \TT, \ \int_{\TT} \varphi(x) \ud x >
0 \bigg) \quad \Longrightarrow \quad (e^{t \bar{H}^i} \varphi) (x) \geqslant 0 \quad
     \forall x \in \TT . \]
  In addition for any $ c \in \RR $ (we identify the constant $c$ with
  the constant function $ c(x)=c $) $ e^{t \bar{H}^i} c = c$. In particular we observe that for any continuous $\varphi$:
  \begin{equ}
      \min_{x \in \TT}  (e^{t \bar{H}} \varphi) (x)  \geqslant 
      \min_{x \in \TT} \varphi (x), \qquad
      \max_{x \in \TT}  (e^{t \bar{H}} \varphi) (x)  \leqslant 
      \max_{x \in \TT} \varphi (x) .
  \end{equ}
The last maximum principles provide
  the estimate:
  \begin{equs}
    \min_{x \in \TT } u ( n \tau , x) & \geqslant  e^{ \tau \sum_{i =
1}^{n} \zeta^{i}} \prod_{i = 1}^{n (t)} \left[ \frac{\min_{x \in \TT}
    \psi^i (x)}{\max_{x \in \TT} \psi^i (x)} \right] \geqslant  \exp
\left( \tau \sum_{i = 1}^{n
    } \zeta^{i} - \frac{ \mu^i}{\tau} \right) .
  \end{equs}
  where $ \mu^{i} = \log (\max_{x \in \TT} \psi^i (x)) - \log (\min_{x \in
    \TT } \psi^i (x))$.
  In particular, by the ergodic theorem $ \PP- $almost surely 
  \begin{equ}
    \lim_{t \rightarrow \infty} \frac{1}{\tau n } \log \left( \tau \sum_{i = 1}^{n
    } \zeta^{i} - \frac{\mu^i}{\tau} \right)  =  \mathbf{E} \left[
\lambda - \frac{\mu}{\tau} \right].
  \end{equ}
As for the upper bound in \eqref{eqn:bound-lyapunov}, it's a simple consequence
of the inequality $ \|u(n \tau) \|_{L^{1}} \leqslant \| u(n \tau)
\|_{L^{2}} \lesssim \exp(\sum_{i = 1}^{n} \tau \zeta^{i})$.
  \end{proof} \\
We conclude by proving that the average top eigenvalue is positive.
\begin{lemma}\label{lem:positive-avrg-lyap}
Consider $ \xi_{\mathrm{stat}} $ as in Assumption~\ref{assu:smooth-noise}
or~\ref{assu:white-noise}. Then
\begin{equs}
\EE \big[ \max \sigma( \Delta + \xi_{\mathrm{stat}}) \big] > 0.
\end{equs}
\end{lemma}

\begin{proof}
Consider a smooth random function $\psi \colon \Omega \to C^{\infty}(\TT)$ 
such that
\begin{equs}
\EE [\langle \xi, \psi \rangle] > 0, \qquad \| \psi \|_{C^{2}} = 1.
\end{equs}
We observe that it's always possible to construct such $ \psi $ under
both Assumption~\ref{assu:smooth-noise} and Assumption~\ref{assu:white-noise}.
We want to use the representation
  \[ \max \sigma(\Delta + \xi_{\mathrm{stat}}) = \sup_{ \eta \in C^{\infty}  
\colon  \| \eta \|_{L^2 (\TT)} = 1}
     \langle (\Delta + \xi_{\mathrm{stat}} ) \eta, \eta \rangle, \]
which follows from the fact that $ C^{\infty} $ is dense in the domain of the
Hamiltonians in Definition~\ref{def:hamiltonians}. Then for $ \alpha \in
(0, 1) $ define
\begin{equs}
\eta_{\alpha}(x) = c_{\alpha} \left( 1 + \alpha \psi (x) \right), \qquad
\forall x \in \TT,
\end{equs}
with $ c_{\alpha} > 0 $ so that $ \| \eta_{\alpha} \|_{L^{2}} =1 $: in particular
\[ (1 + \alpha)^{-1} \leqslant (1 + \alpha \| \psi \|_{L^{2}})^{-1} \leqslant c_{\alpha} \leqslant (1 - \alpha \| \psi \|_{L^{2}})^{-1}  \leqslant (1 -
\alpha)^{-1}. \] 
Then
\begin{equs}
\EE[ \langle (\Delta + \xi_{\mathrm{stat}}) \eta_{\alpha}, \eta_{\alpha} \rangle
] = -\alpha^{2}  \EE \big[ c_{\alpha}^{2} \| \nabla \psi \|^{2}_{L^{2}} \big] + 2 \alpha \EE
\big[ c_{\alpha}^{2} \langle \xi_{\mathrm{stat}}, \psi \rangle \big] +
\alpha^{2} \EE \big[ c_{\alpha}^{2}\langle \xi_{\mathrm{stat}} , \psi^{2}
\rangle \big],
\end{equs}
so that by the previous bound, since $ \lim_{\alpha \to 0} c_{\alpha} =1 , $ we
obtain $$ \lim_{\alpha \to 0} \alpha^{-1} \EE \big[ \langle
(\Delta + \xi_{\mathrm{stat}}) \eta_{\alpha}, \eta_{\alpha} \rangle
\big] = 2 \EE \big[ \langle \xi_{\mathrm{stat}}, \psi \rangle \big] > 0,$$
which proves the desired result.
\end{proof}

\section{Projective invariant measures}\label{sec:proj-invariant-measures}

In this section we study the projective space
\[ \mathrm{Pr} = \left\{ \varphi \in C (\TT ; [0, \infty)) \quad
   s.t. \quad \varphi (x) \geqslant 0, \quad \forall x \in \TT \quad
   \mathrm{and} \quad \int_{\TT} \varphi (x) \ud x = 1 \right\}, \]
and some of its fundamental properties.
This space, endowed with Hilbert's projective distance
\[ d_H (\varphi, \psi) = \max_{x \in \TT} \log (\varphi (x) / \psi (x)) -
   \min_{x \in \TT} \log (\psi (x) / \varphi (x)), \]
is a complete metric space (see e.g. \cite[Lemma 2.2]{Rosati2019synchronization}). Our purpose
is to understand properties of the invariant measure associated to
\eqref{eqn:time-Anderson}, when seen as a cocycle on $
\mathrm{Pr} $. For the needs of this section it will be convenient to work with
product probability spaces (a stronger requirement than in
Assumptions~\ref{assu:smooth-noise},\ref{assu:white-noise}, but we
can always modify the probability space $ \Omega $ to meet this requirement).

\begin{assumption}\label{assu:probability-space}
Under either Assumption~\ref{assu:smooth-noise} or
Assumption~\ref{assu:white-noise}, consider a probability space $ (\Omega_{\mathrm{stat}}, \mF_{\mathrm{stat}},
\PP_{\mathrm{stat}}) $ supporting a random variable $ \xi_{\mathrm{stat}} $
with law as required by \ref{assu:smooth-noise} or \ref{assu:white-noise}.
Then assume that the probability space $ (\Omega, \mF, \PP)$ is the following product space endowed with the product sigma-algebra
and the product measure:
\begin{equs}
\Omega =  \Omega_{\mathrm{stat}}^{\otimes \ZZ}, \quad \mF =
 \mF_{\mathrm{stat}}^{\otimes \ZZ}, \quad \PP = \PP^{\otimes
\ZZ}_{\mathrm{stat}}.
\end{equs}
In this way every $ \omega \in \Omega $ is of the form $ \omega =
(\omega_{i})_{i \in \ZZ}, $ with $ \omega_{i} \in \Omega_{\mathrm{stat}} $, and
we can assume that the maps $ \xi^{i}_{\mathrm{stat}} $ of
Definition~\ref{def:tau-potential} are given by:
\begin{equs}
\xi^{i}_{\mathrm{stat}}(\omega) = \xi_{\mathrm{stat}} (\omega_{i}).
\end{equs}
Finally define the map $ \vt \colon \ZZ \times \Omega \to \Omega $ by
\begin{equs}
\vt (z, \omega) = \vt^{z}(\omega) = (\omega_{i + z})_{i \in \ZZ}.
\end{equs}
\end{assumption}
In this setting we associate to any strictly
positive (meaning $ A \varphi (x) > 0, \ \forall x \in \TT $ for all $
\varphi \in \mathrm{Pr}$ ) and bounded
operator $ A \colon C(\TT) \to C(\TT)$ a projective map
\begin{equs}
A^{\pi} \colon \mathrm{Pr} \to \mathrm{Pr}, \qquad A^{\pi}(\varphi) =
\frac{A \varphi}{\int_{\TT} A(\varphi)(x) \ud x}.
\end{equs}
The reason why we consider the distance $ d_{H} $ is the following contraction
property.
\begin{theorem}[Birkhoff's contraction]\label{thm:birkhoff}
  If $ A $ is a strictly positive operator on $ C(\TT) $, then there exists a constant $
  \mu(A) \in [0, 1) $ such that
  \begin{equs}
    d_{H}(A^{\pi} \varphi, A^{\pi} \psi) \leqslant \mu(A) d_{H}(\varphi, \psi),
    \qquad \forall \varphi, \psi \in \mathrm{Pr}.
  \end{equs}
\end{theorem}
In particular, a consequence of this proposition is the following result.

\begin{proposition}\label{prop:existence-proj-invariant-measure}
  Under Assumption~\ref{assu:smooth-noise} or Assumption~\ref{assu:white-noise}
and in the setting of Assumption~\ref{assu:probability-space}, for any $\tau > 0$ there
  exists a unique map
  $ z_{\infty} (\tau) : \Omega \rightarrow \mathrm{Pr} $
  that satisfies either of the following for all $ \omega \in \Omega$ outside a
$ \vt- $invariant null-set:
  \begin{enumerate}
    \item \textbf{(Synchronization)} For any $\varphi \in \mathrm{Pr}$:
    \[ \limsup_{n \rightarrow \infty} \frac{1}{n} \log \left( d_H \left(
       z_{\infty} (\tau, \vartheta^n \omega), \left( \prod_{i = 1}^n e^{\tau
       H^i (\omega)} \right)^{\pi} \varphi \right) \right) < 0. \]
    \item \textbf{(Invariance)} For every $\ell, n \in \NN \cup \{
0 \}$
    \[ \left( \prod_{i = n + 1}^{n + \ell} e^{\tau H^i (\omega)}
\right)^{\pi} z_{\infty} (\tau, \vartheta^n \omega) = z_{\infty} (\tau, \vartheta^{n
       + \ell} \omega) . \]
\end{enumerate} 
In addition, for every $ \omega \in \Omega $ with $ \omega =
(\omega_{i})_{i \in \ZZ} $ define $ \omega_{\leqslant} =
(\omega_{i})_{i \leqslant 0} \in \Omega^{-} : = \prod_{i \leqslant 0}
\Omega_{\mathrm{stat}} $ and $ \omega_{>} =
(\omega_{i})_{i > 0} \in \Omega^{+} : = \prod_{i > 0}
\Omega_{\mathrm{stat}}. $ Then there exists a map $ z_{\infty}^{\leqslant}
(\tau) \colon \Omega_{\leqslant} \to \mathrm{Pr} $ such that $$ z_{\infty}^{\leqslant}
(\tau, \omega_{\leqslant}) = z_{\infty}(\tau, \omega).$$
\end{proposition}
We refer the reader to
\cite[Theorem 3.4]{Rosati2019synchronization} for a
proof of the proposition above and in general for a more detailed discussion also of Theorem~\ref{thm:birkhoff}.  The
``convergence in direction'' that the previous proposition proves is useful
to derive some classical statements regarding Lyapunov exponents. We
start with a proof of Lemma~\ref{lem:furst}.

\vspace{1em}

\begin{proof}(of Lemma~\ref{lem:furst})
From the subadditive ergodic theorem we have that
\begin{equs}
\limsup_{t \to \infty} \frac{1}{t} \log \left( \int_{\TT} u (t, x) \ud x
\right) \in [- \infty, \infty ),
\end{equs}
since $ \EE \sup_{0 \leqslant t \leqslant 1} \Big( \log \left( \int_{\TT} u (t, x)
\ud x \right) \vee 0 \Big) < \infty $ by calculations simpler than those in
Lemma~\ref{lem:moment-bound-invariant-measure} for regular noise and
Lemma~\ref{lem:moment-bound-solution-1D-wn} for white noise. Next, let us prove
Equation~\eqref{eqn-furstenberg-formula}, which proves that the limsup is
really a limit. Let $ \overline{\lambda}  \in [- \infty, \infty) $ be defined
by the following limit (up to taking a subsequence of $ n $, to ensure that the limit exists):
  \begin{equ}
    \bar{\lambda}  =  \lim_{n \to \infty} \frac{1}{ \tau
    n} \log \left( \int_{\TT} u (n \tau, x) \ud x \right) = \frac{1}{\tau} \lim_{n \rightarrow \infty} \frac{1}{n}  \sum_{i =
    1}^n \log \left( \int_{\TT} e^{\tau H^i} (z^{i - 1}) (x) \ud x \right),
  \end{equ}
  where we have just rewritten the first term via a telescopic sum with:
  \[ z^i (x) = \frac{u (i \tau, x)}{\int_{\TT} u (i \tau, x) \ud x} . \]
  Now consider $ z_{\infty} $ as in Proposition~\ref{prop:existence-proj-invariant-measure} and let $
z_{\infty}^{i} (\omega) = z_{\infty}(\vt^{i} \omega) $ for $ i \in \NN $. Then for almost all $
\omega \in \Omega $ there exist some $b (\omega), c (\omega) > 0$ such that
  \[ d_H (z^i (\omega), z_{\infty}^i (\omega)) \leqslant b (\omega) e^{- c (\omega) i} . \]
We can thus rewrite the terms in the telescopic sum, for almost all $ \omega \in
\Omega $, as:
  \begin{equs}
    \frac{1}{n}  \sum_{i = 1}^n \log \left( \int_{\TT} e^{\tau H^i} (z^i) (x)
    \ud x \right) & =  \frac{1}{n}  \sum_{i = 1}^n \log \left( \int_{\TT}
    e^{\tau H^i} \left( z_{\infty}^{i - 1} \frac{z^{i - 1}}{z_{\infty}^{i -
    1}} \right) (x) \ud x \right)\\
    & \leqslant  \frac{1}{n}  \sum_{i = 1}^n \left[ \log \left( \int_{\TT}
    e^{\tau H^i} (z_{\infty}^{i - 1}) (x) \ud x \right) + \log \left(
    \max_{x \in \TT } \frac{z^{i - 1} (x)}{z_{\infty}^{i - 1} (x)} \right)
    \right]\\
    & \leqslant  \frac{1}{n} \sum_{i = 1}^n \left[ \log \left( \int_{\TT}
    e^{\tau H^i} (z_{\infty}^{i - 1}) (x) \ud x \right) \right] +
    \frac{1}{n} \sum_{i = 1}^n b e^{- c i} .
  \end{equs}
  So that passing to the limit, using independence and the ergodic theorem:
  \begin{equs}
    \lim_{n \rightarrow \infty} \frac{1}{n}  \sum_{i = 1}^n \log \left(
    \int_{\TT} e^{\tau H^i} (z^i) (x) \ud x \right) & \leqslant  \lim_{n
    \rightarrow \infty} \frac{1}{n} \sum_{i = 1}^n \left[ \log \left( \int_{\TT}
    e^{\tau H^i} (z_{\infty}^{i - 1}) (x) \ud x \right) \right]\\
    & \leqslant  \int_{\Omega \times \Omega} \log \left( \int_{\TT}
    e^{\tau H^1 (\omega)} (z_{\infty}^0 (\omega')) (x) \ud x \right) \ud
    \PP (\omega) \ud \PP (\omega'),
  \end{equs}
which is the required upper bound for~\eqref{eqn-furstenberg-formula}. The
lower bound follows analogously, so that $ \overline{\lambda} (\tau) =
\lambda(\tau) $ and~\eqref{eqn-furstenberg-formula} is
proven. We are left with proving that $ \lambda (\tau) > - \infty $. In the
case of regular noise, this now follows by Furstenberg's formula and similar calculations as in
Lemma~\ref{lem:moment-bound-solution-1D-wn}, while for white noise this is
implied for example by Corollary~\ref{cor:bound-eigenvalue}.
\end{proof} \\
The following results establishes instead the continuity of the Lyapunov exponent.

\begin{lemma}\label{lem:continuity}
Under Assumption~\ref{assu:smooth-noise} or Assumption~\ref{assu:white-noise}
and in the setting of Assumption~\ref{assu:probability-space} the map $ \lambda
\colon (0, \infty) \to \RR $ as in Lemma~\ref{lem:furst} is continuous.
\end{lemma}

\begin{proof}
If suffices to establish the continous dependence on \(\tau\) of
Equation~\eqref{eqn-furstenberg-formula}. First observe that for any $ \sigma
\in (0, \infty), \ \lim_{\tau \to \sigma} z_{\infty}(\tau) =
z_{\infty}(\sigma)$ in distribution in $ \mathrm{Pr} $. Indeed, by
Lemmata~\ref{lem:moment-bound-invariant-measure},~\ref{lem:moment-bound-invariant-msr-1D-1n} the sequence $ \{
z_{\infty}(\tau) \}_{| \tau - \sigma | \leqslant 1} $ is tight in $ C(\TT) $
and one can easily check that any limit point for $ \tau \to \sigma$ must satisfy the invariance
property of Proposition~\ref{prop:existence-proj-invariant-measure} (for $
\tau = \sigma $). Since only $ z_{\infty}(\sigma)$ satisfies this property we deduce the required
convergence. Using the independence between $ H(\omega)$, and $z_{\infty}(\tau,
\omega^{\prime})$, together with Lemma~\ref{lem:moment-bound-solution-1D-wn}
(and a similar but simpler result if the noise is regular) we also observe that 
\begin{equs}
\lim_{\tau \to \sigma} e^{\tau H (\omega)} z_{\infty}(\tau, \omega^{\prime})
= e^{\sigma H (\omega)} z_{\infty}(\sigma, \omega^{\prime})
\end{equs}
in distribution. Then the claimed convergence holds by uniform integrability, observing that for any $
p \geqslant 1 $
\begin{equs}
\sup_{| \tau - \sigma | \leqslant 1} \EE^{\PP \otimes \PP} \Big\vert  \log
    \left( \int_{\TT} e^{\tau H (\omega)} (z_{\infty} (\tau, \omega')) (x)
    \ud x \right) \ud \PP (\omega) \ud \PP (\omega') 
\Big\vert^{p} < \infty.
\end{equs}
This follows from Corollary~\ref{cor:bound-eigenvalue} under
Assumption~\ref{assu:white-noise} and by analogous but simpler calculations under
Assumption~\ref{assu:smooth-noise}.
\end{proof} \\
Next we study some properties of the invariant measure that we will need for
our results. In particular we prove certain moment bounds and study the
convergence of the behaviour of the measures for $ \tau \to 0 $. The results as
well as their proofs will be slightly different under
either Assumption~\ref{assu:smooth-noise} or Assumption~\ref{assu:white-noise}.
Therefore we distinguish the two settings, starting with the latter.

\subsection{Moment estimates for regular noise}

Now we concentrate on moment estimates and on the convergence for $ \tau \to 0 $ of the invariant
measure associated to \eqref{eqn:time-Anderson}. We start with the regular
setting of Assumption~\ref{assu:smooth-noise}.

\begin{lemma}\label{lem:moment-bound-invariant-measure}
Under Assumption~\ref{assu:smooth-noise}, for any $p\geqslant 1$ and $\alpha \in
  (0, 2)$ one can bound
  \[ \sup_{\tau \in (0, \infty)} \EE \| z_{\infty} (\tau)
     \|_{\mC^{\alpha}}^p < \infty . \]
  In particular, the sequence $\{z_{\infty} (\tau)\}_{\tau \in (0, 1)} $ is tight in $C (\TT)$ and any limit
  point for $\tau \rightarrow 0$ is supported in $ \mathrm{Pr} $.
\end{lemma}

\begin{proof}
  Let us fix $\tau > 0$ and $ n \in \NN $, then by
Proposition~\ref{prop:existence-proj-invariant-measure}, for almost every $ \omega \in \Omega$
  \begin{equs}
    z_{\infty} (\tau, \vartheta^n \omega) & = & \frac{\prod_{i = 1}^n e^{\tau
    (\Delta + \xi_{\mathrm{stat}}^{i})} z_{\infty} (\tau, \omega)}{\int_{\TT} \big[ \prod_{i = 1}^n
    e^{\tau (\Delta + \xi_{\mathrm{stat}}^{i})} z_{\infty} (\tau, \omega ) \big] (x) \ud x} .
  \end{equs}
  To lighten the notation we avoid writing explicitly the dependence on $ \omega, $ as long
as no confusion can arise. Using a maximum principle we can bound the denominator by:
  \[ \int_{\TT} \prod_{i = 1}^n \Big[ e^{\tau (\Delta +
\xi_{\mathrm{stat}}^{i})} z_{\infty} (\tau
     ) \Big]( x) \ud x \geqslant e^{- \tau \sum_{i = 1}^n \|
     \xi^i_{\mathrm{stat}} \|_{\infty}} \int_{\TT} z_{\infty} (\tau, x)
     \ud x = e^{- \tau \sum_{i = 1}^n \| \xi^i_{\mathrm{stat}} \|_{\infty}},
  \]
  while for the nominator we observe that
  \[ \left[ \prod_{i = 1}^n e^{\tau (\Delta + \xi_{\mathrm{stat}}^{i})} z_{\infty} (\tau)
     \right] (x) = u (n \tau, x), \]
  where the latter is the solution to
  \[ \partial_t u = \Delta u + \xi^{\tau} u, \qquad u (0, x) = z_{\infty} (\tau,
      x) . \]
  Let us write $n (t)$ for the smallest integer such that $\tau n (t) \geqslant
t$. Via Duhamel's formula
\begin{equs}
u(t) = P_{t} z_{\infty}(\tau) + \int_{0}^{t} P_{t -s}
(\xi^{\tau}(s) u(s)) \ud s.
\end{equs}
Then by the Schauder estimates in Lemma~\ref{lem:schauder-estiamtes} we have
for any $ \alpha \in (0, 2) $ and $ \ve>0 $ such that $ \alpha + \ve < 2 $:
\begin{equs}
\| u (t) \|_{B^{\alpha}_{1, \infty}} \lesssim t^{-
\frac{\alpha + \ve}{2}} \| z_{\infty}(\tau) \|_{L^{1}} + \int_{0}^{t}
(t-s)^{- \frac{\alpha + \ve}{2}} \| \xi^{\tau}(s) \|_{\infty} \| u
(s) \|_{L^{1}} \ud s.
\end{equs}
By a maximum principle and since $ \| z_{\infty}(\tau) \|_{L^{1}} = 1 $ we have $ \| u (s)
\|_{L^{1}} \leqslant  \exp \Big( \sum_{i = 1}^{n(t)} \tau \|
\xi_{\mathrm{stat}}^{i} \|_{\infty} \Big)$, so that
\begin{equs}
\| u (t) \|_{B^{\alpha}_{1, \infty}} \lesssim t^{-
\frac{\alpha + \ve}{2}}  + \exp \Big( C \sum_{i = 1}^{n(t)} \tau \|
\xi_{\mathrm{stat}}^{i} \|_{\infty} \Big) \int_{0}^{t} (t-s)^{- \frac{\alpha + \ve}{ 2}}
\| \xi^{\tau}(s) \|_{\infty}  \ud s.
\end{equs}
Hence overall for some $ C> 0 $
\begin{equs}
\EE \| z_{\infty} (\tau) \|_{B^{\alpha}_{1, \infty}}^{p} & \lesssim \EE \bigg[
 \exp \Big( C p \sum_{i = 1}^{n(1)} \tau \|
\xi_{\mathrm{stat}}^{i} \|_{\infty} \Big) \bigg( 1 +  \int_{0}^{\tau
n(1)} ( \tau n(1) -s)^{- \frac{\alpha + \ve}{ 2}}
\| \xi^{\tau}(s)\|_{\infty}  \ud s \bigg)^{p}\bigg] \\
& \lesssim \EE \bigg[
 \exp \Big( 2 C p \sum_{i = 1}^{n(1)} \tau \|
\xi_{\mathrm{stat}}^{i} \|_{\infty} \Big) \bigg]^{\frac{1}{2}} \EE\bigg[ \bigg( 1 +  \int_{0}^{\tau
n(1)} ( \tau n(1) -s)^{- \frac{\alpha + \ve}{ 2}}
\| \xi^{\tau}(s)\|_{\infty}  \ud s \bigg)^{2 p}\bigg]^{\frac{1}{2}}.
\end{equs}
The first term is bounded by the exponential moment bound of
Assumption~\ref{assu:smooth-noise}. Similarly the second term, since:
\begin{equs}
\EE\bigg[ \bigg( \int_{0}^{\tau
n(1)} ( \tau n(1) -s)^{- \frac{\alpha + \ve}{ 2}}
\| \xi^{\tau}(s)\|_{\infty}  \ud s \bigg)^{2 p}\bigg]^{\frac{1}{2}} \lesssim \int_{0}^{\tau
n(1)} ( \tau n(1) -s)^{- \frac{\alpha + \ve}{ 2}}
\EE \| \xi_{\mathrm{stat} }\|_{\infty}^{2p}  \ud s \lesssim 1.
\end{equs}
This is not quite enough, since our aim is to bound $ z_{\infty}(\tau) $ in
\(\mathcal{C}^{\alpha}\) and not in $ B^{\alpha}_{1, \infty} $. But we can
iterate the argument by using the bound we established, together with the
Besov-Sobolev embedding $ B^{\alpha}_{1 , \infty} \subseteq \mC^{\alpha -
\frac{1}{2}} $ (in dimension $ d=1 $).
\end{proof} \\
From the previous tightness we can deduce the convergence for $ \tau \to 0 $.

\begin{proposition}\label{prop:averaging-invariant-measure}
  The following convergence holds in distribution, as a sequence of random
variables with values in $ \mC^{\alpha}( \TT)$ for any $\alpha \in (0, 2)$:
  \[ \lim_{\tau \rightarrow 0} z_{\infty} (\tau) = 1. \]
\end{proposition}

\begin{proof}
  We have already proven in Lemma~\ref{lem:moment-bound-invariant-measure}
  that the sequence $\{ z_{\infty} (\tau) \}_{\tau \in (0, 1)}$ is tight
  in $ \mC^{\alpha} (\TT)$. To establish the limit as $\tau \rightarrow 0$ we
  observe that for any $t > 0$ and $ n(t) $ the smallest integer such that $
\tau n (t) \geqslant t $:
  \[ z_{\infty} (\tau) \overset{d}{=} S_{\tau} ( \tau n (t)) z_{\infty} (\tau) . \]
  Here $S_{\tau} (t) u_0$ is the solution at time $t \geqslant 0$ to 
Equation~\eqref{eqn:time-Anderson} with $\xi = \xi^{\tau}$ and initial
condition $ u_{0}$, and is chosen to be independent of $ z_{\infty}(\tau) $ on
the right-hand side. Now let
  $z_{\infty} (0)$ be any limit point of the sequence $z_{\infty} (\tau)$ for
$ \tau \to 0 $.
Then by Lemma~\ref{lem:averaging-smooth-op} and by the independence of $
S_{\tau} $ and $ z_{\infty}( \tau)$, for every $ t \geqslant 0 $ there exists a
subsequence $ \tau_{k} > 0, \tau_{k} \to 0 $, such that 
  \begin{equation}
    \lim_{k \to \infty } S_{\tau_{k}} (\tau_{k} n (t)) z_{\infty} (\tau_{k}) = P_t
    z_{\infty} (0). 
  \end{equation}
Hence we conclude
  \[ z_{\infty} (0) \overset{d}{=} P_t z_{\infty} (0) . \]
  Since the Dirac measure in the function that is constantly $1$ is the
  only invariant measure for $P_t$ we have proven our result.
\end{proof}\\
We conclude the subsection with two lemmata used in the previous proof. We
start with an averaging result for the solution map to
\eqref{eqn:time-Anderson}.

\begin{lemma}\label{lem:averaging-smooth-op}
  Fix any $\alpha \in (0, 2)$ and consider a bounded sequence $(t_{\tau})_{\tau \in (0,
  1)}$ of positive real numbers $t_{\tau} > 0$ such that for some $t > 0$
  \[ \lim_{\tau \rightarrow 0} t_{\tau} = t. \]
  Then for every $u_0 \in \mC^{\alpha} (\TT)$
  \[ S_{\tau} (t_{\tau}) u_0 \rightarrow P_t u_0 \ \ \text{ in
probability in } \  \mC^{\alpha}(\TT).\]
\end{lemma}

\begin{proof}
 Consider the process $\mI(\xi^{\tau})$ solving $(\partial_t -
\Delta) \mI(\xi^{\tau})
  = \xi^{\tau}$, with $\mI(\xi^{\tau})(0) = 0$:
  \[ \mI(\xi^{\tau})(t) = \int_0^t P_{t - s} [\xi^{\tau} (s)] \ud s. \]
 We can write $S_{\tau} (t) u_0 = e^{\mI(\xi^{\tau}) (t)} Q_{\tau} (t) u_0,$ where
  $Q_{\tau} (t) u_0$ is the solution to
  \[ (\partial_t - \Delta) [Q_{\tau} (t) u_0] = 2 (\partial_x
\mI(\xi^{\tau })_{t} ) \partial_x [Q_{\tau} (t) u_0] +  (\partial_x
\mI(\xi^{\tau} )_{t} )^2 [Q_{\tau} (t) u_0], \qquad Q_{\tau} (0) u_0 = u_0 . \]
We observe that by the second statement of
Lemma~\ref{lem-convergence-to-zero-Xi} we have (by Besov embedding, choosing $
p \geqslant 1 $ sufficiently large) $ \lim_{\tau \to 0} \|
\mathcal{I}(\xi^{\tau}) (t_{\tau})\|_{\mC^{\alpha}}= 0$ in probability, so
that our result follows if we show that $ Q_{\tau}(t_{\tau}) u_{0} \to P_{t}
u_{0} $ in probability in $ \mC^{\alpha}. $ Let us fix a $ T > \sup_{\tau}
 t_{\tau}$. Now by Lemma \ref{lem-convergence-to-zero-Xi} we have that for any $ p
\geqslant 2 $, $ \lim_{\tau \to 0}
M (\tau, p) =0 $ in probability, where 
\begin{equs}
M(\tau, p) = \| \mathcal{I}(\xi^{\tau}) \|_{L^{p}_{T} \mC^{\alpha}}.
\end{equs}
In particular, assuming $ \alpha \in (1, 2) $, there exists
a $ q(p) \geqslant 2 $ such that
\begin{equs}
\| \partial_{x} \mI (\xi^{\tau}) \|_{L^{p}_{T} \mC^{\alpha -1}} \lesssim
M(\tau, p), \qquad \| ( \partial_{x} \mI (\xi^{\tau}) )^{2} \|_{L^{p}_{T}
\mC^{\alpha-1}} \lesssim_{\beta} M (\tau, 2p)^{2}.
\end{equs}
These estimates provide us with an a priori bound on $ Q_{\tau}(t) u_{0} $. We
find for any $ \ve>0 $ such that $ \frac{\alpha + \ve}{ 2} < 1 $, and assuming
again that $ \alpha \in (1, 2) $, via Lemma~\ref{lem:schauder-estiamtes}:
\begin{equs}
\| Q_{\tau} & (t) u_{0}  \|_{\mC^{\alpha}} \\
&  \lesssim \| u_{0} \|_{\mC^{\alpha}} +
\int_{0}^{t}(t -s)^{- \frac{\alpha + \ve}{2}} \Big[ \|  (\partial_x
\mI(\xi^{\tau })_{s} ) \partial_x [Q_{\tau} (s) u_0] \|_{\mC^{- \ve}} + \|(\partial_x
\mI(\xi^{\tau} )_{s} )^2 [Q_{\tau} (s) u_0] \|_{\mC^{- \ve}} \Big] \ud s \\
& \lesssim \| u_{0} \|_{\mC^{\alpha}} +
\sup_{0 \leqslant s \leqslant t} \| Q_{\tau}(s) u_{0} \|_{\mC^{\alpha}}\int_{0}^{t}(t -s)^{- \frac{\alpha + \ve}{2}} \Big[ \|  (\partial_x
\mI(\xi^{\tau })_{s} )  \|_{\mC^{\alpha -1}} + \|(\partial_x
\mI(\xi^{\tau} )_{s} )^2  \|_{\mC^{\alpha-1}} \Big] \ud s\\
& \lesssim \| u_{0} \|_{\mC^{\alpha}} +t^{\zeta} 
\sup_{0 \leqslant s \leqslant t} \| Q_{\tau}(s) u_{0} \|_{\mC^{\alpha}}
\Big( M(\tau, p) + M (\tau, 2p)^{2} \Big), \label{eqn:prf-homogenization-support}
\end{equs}
for some $ \zeta > 0 $, which is obtained by applying H\"older's inequality,
assuming $ p \geqslant 1 $ is sufficiently large.
Using a Gronwall argument we conclude that
\begin{equs}
\sup_{0 \leqslant s \leqslant t}\| Q_{\tau}(s) u_{0} \|_{\mC^{\alpha}} \lesssim
\| u_{0} \|_{\mC^{\alpha}} \exp\Big(C(T) \big(1 + M(\tau, p) + M (\tau,
2p)^{2} \big)^{\zeta^{\prime}} \Big), 
\end{equs}
for some $ \zeta^{\prime} > 0 $.
In particular, we can now bound by \eqref{eqn:prf-homogenization-support}:
\begin{equs}
\| Q_{\tau}(t) u_{0} - P_{t} u_{0} \|_{\mC^{\alpha}} & = \bigg\| \int_{0}^{t} P_{t-s} \Big(
  2 (\partial_x \mI(\xi^{\tau })_{s} ) \partial_x [Q_{\tau} (s) u_0] + 2 (\partial_x
\mI(\xi^{\tau} )_{s} )^2 [Q_{\tau} (s) u_0]\Big) \ud s \bigg\|_{\mC^{\alpha}}
\\
& \lesssim \Big( M(\tau, p) + M (\tau, 2p)^{2}\Big) \exp\Big(C(T) \big(1 + M(\tau, p) + M (\tau,
2p)^{2} \big)^{\zeta^{\prime}} \Big).
\end{equs}
Since $ M(\tau, p) \to 0  $ in probability, our result follows.
\end{proof} \\
Finally we establish some bounds for the solution to the linear equation.
\begin{lemma}
  \label{lem-convergence-to-zero-Xi} The process $\mI(\xi^{\tau})$ defined for any
  $\tau > 0$ by
  \[ \mI(\xi^{\tau})(t)= \int_0^t P_{t - s} [\xi^{\tau} (s)] \ud s, \qquad t > 0
  \]
  satisfies for any $T > 0$, any $\alpha \in (0, 2)$ and any $ p \geqslant 2 $:
\begin{equs}
\lim_{\tau \to 0} \EE \| \mathcal{I}(\xi^{\tau}) \|_{L^{p}([0,T];
B^{\alpha}_{p,p})}^{p} = 0.
\end{equs}
Moreover, for any bounded sequence $ (t_{\tau})_{\tau \in (0, 1)} $ of positive
real numbers, we have
\begin{equs}
\lim_{\tau \to 0} \EE \| \mI(\xi^{\tau})(t_{\tau})
\|_{B^{\alpha}_{p,p}}^{p} = 0.
\end{equs}
\end{lemma}

\begin{proof}
Let us define $ K^{x}_{j}(y) = \mF^{-1} [\varrho_{j} ](x - y). $ Then we can
rewrite the $ B^{\alpha}_{p,p} $ norm as:
\begin{equs}
\| \varphi \|_{B^{\alpha}_{p,p}}^{p}  = \sum_{j \geqslant  -1} 2^{\alpha j p}
\int_{\TT} | \langle \varphi, K_{j}^{x} \rangle |^{p} \ud x,
\end{equs}
therefore, by Fubini, our objective will be to bound $ \EE | \langle
\mI(\xi^{\tau}) (t), K_{j}^{x} \rangle |^{p} $ uniformly over $ x \in \TT$.
As before, let $ n(t) $ be the smallest integer such that $ \tau n(t) \geqslant
t. $ We can use Rosenthal's inequality \cite[Theorem 2.9]{Petrov1995}, which we can apply because
the sequence $ \{ \xi_{\mathrm{stat}}^{i} \}_{i \in \NN} $ is independent, and $ \EE P_{t}
\xi_{\mathrm{stat}} = 0 $, to bound for $ p
\geqslant 2 $ 
\begin{equs}
\EE | \langle & \mI(\xi^{\tau}) (t), K_{j}^{x} \rangle |^{p}  = \EE \bigg\vert
\sum_{i = 0}^{n(t)-1} \int_{i \tau}^{(i + 1) \tau \wedge t}  \langle
P_{t - s} [\xi^{\tau} (s)], K_j^x \rangle \ud s \bigg\vert^{p} \\
& = \EE \bigg\vert
\sum_{i = 0}^{n(t)- 1} \int_{i \tau}^{(i + 1) \tau \wedge t} \langle
P_{t - s} [\xi_{\mathrm{stat}}^{i}], K_j^x \rangle \ud s \bigg\vert^{p} \\
& \lesssim \sum_{i = 0}^{n(t)-1} \EE \bigg\vert \int_{i \tau}^{(i + 1) \tau \wedge t} \langle
P_{t - s} [\xi_{\mathrm{stat}}^{i}], K_j^x \rangle \ud s \bigg\vert^{p} + \bigg(
\sum_{i= 0}^{n(t)-1} \EE \bigg\vert \int_{i \tau}^{(i + 1) \tau \wedge t} \langle
P_{t - s} [\xi_{\mathrm{stat}}^{i}], K_j^x \rangle \ud s \bigg\vert^{2}
\bigg)^{\frac{p}{2}}.
\end{equs}
We observe that in addition the following estimate holds for any $ \ve> 0 $
and $ \alpha \in (0, 2) $ with $ \frac{\alpha + \ve}{2} < 1 $:
\begin{equs}
\bigg\vert\int_{i \tau}^{(i + 1) \tau \wedge t} \langle
P_{t - s} [\xi_{\mathrm{stat}}^{i}], K_j^x \rangle \ud s \bigg\vert & \lesssim \int_{i \tau}^{(i+1) \tau \wedge t} \| P_{t-s}
\xi_{\mathrm{stat}^{i}} \|_{\mC^{\alpha+ \ve}} \| K_{j}^{x} \|_{B^{-
\alpha}_{1, 1}} \ud s \\
& \lesssim \| \xi_{\mathrm{stat}}^{i} \|_{\infty} \| K_{j}^{x} \|_{B^{-
\alpha}_{1, 1}} \int_{i \tau}^{(i + 1) \tau \wedge t} (t -s)^{- \frac{\alpha + \ve}{2}}
\ud s.
\end{equs}
In particular, if we now define $ G_{\tau}(i,t) = \int_{i \tau}^{(i+1) \tau \wedge
t} (t -s)^{- \frac{\alpha + \ve}{ 2}} \ud s $, we obtain:
\begin{equs}
\EE | \langle \mI(\xi^{\tau}) (t), K_{j}^{x} \rangle |^{p} \lesssim \| K_{j}^{x}
\|_{B^{- \alpha}_{1,1}}^{p} \EE \| \xi^{1}_{\mathrm{stat}}
\|_{\infty}^{p} \bigg( \sum_{i = 0}^{n(t)-1} G_{\tau}^{p}(i, t) + \bigg(
\sum_{i = 0}^{n(t)-1} G_{\tau}^{2} (i, t)
\bigg)^{\frac{p}{2}} \bigg).
\end{equs}
At this point, using the inequality $ \sum_{i} |a_{i}|^{p} \leqslant \big(
\sum_{i} | a_{i} | \big)^{p}$ and since $ \sum_{i = 0}^{n(t)-1}
G_{\tau}(i, t) = \int_{0}^{t} (t -s)^{- \frac{\alpha + \ve}{2}} \ud s $, we
conclude
\begin{equs}[eqn:prf-homogenization-1]
\EE | \langle \mI(\xi^{\tau}) (t), K_{j}^{x} \rangle |^{p} \lesssim \| K_{j}^{x}
\|_{B^{- \alpha}_{1,1}}^{p} \EE \| \xi^{1}_{\mathrm{stat}} \|_{\infty}^{p}
t^{p - p \frac{\alpha + \ve}{ 2}} \lesssim 2^{- \alpha j p},
\end{equs}
by \eqref{eqn:bound-K-j}. On the other hand, we can also bound
\begin{equs}
\bigg\vert\int_{i \tau}^{(i + 1) \tau \wedge t} \langle
P_{t - s} [\xi_{\mathrm{stat}}^{i}], K_j^x \rangle \ud s \bigg\vert & \lesssim
\tau \| \xi_{\mathrm{stat}}^{i} \|_{\infty} \| K_{j}^{x} \|_{L^{1}},
\end{equs}
which leads, following the previous steps, to the bound
\begin{equs}[eqn:prf-homogenization-2]
\EE | \langle \mI(\xi^{\tau}) (t), K_{j}^{x} \rangle |^{p} \lesssim \| K_{j}^{x}
\|_{L^{1}}^{p} \EE \| \xi^{1}_{\mathrm{stat}} \|_{\infty}^{p} \Big(
\tau^{p -1} t + (\tau t)^{\frac{p}{2}} \Big) \lesssim \tau^{p -1} t + (\tau t)^{\frac{p}{2}} .
\end{equs}
Now, interpolating between \eqref{eqn:prf-homogenization-1} and
\eqref{eqn:prf-homogenization-2} delivers all the required results.
\end{proof}

\subsection{Moment estimates for white noise}

We now treat similar bounds as in the previous subsection in the case of space
white noise. This will require some more involved estimates. We start with a
bound on the solution map.

\begin{lemma}\label{lem:moment-bound-solution-1D-wn}
Under Assumption~\ref{assu:white-noise}, for any $
\gamma \in (0,1) $ consider $ u_{0} \in
\mC^{\frac{1}{2} + \gamma } $. Then the solution to
\begin{equ}
\partial_{t} u = \Delta u + \xi_{\mathrm{stat}} u, \quad u(0, \cdot) =
u_{0}(\cdot)
\end{equ}
satisfies the following moment bound for any $ p \geqslant 1, T >0 $:
\begin{align*}
\EE \bigg[ \| u \|_{C^{\frac{1/2 + \gamma}{2}}_{T}
L^{\infty}}^{p} +  \| u \|_{ L^{\infty}_{T} \mC^{\frac{1}{2} +
\gamma}}^{p} \bigg] \lesssim_{p, T} \| u_{ 0} \|_{\mC^{\frac{1}{2} +
\gamma}}^{p}.
\end{align*}
\end{lemma}

\begin{proof}
For every $ \omega \in \Omega$ we can apply
Lemma~\ref{lem:sub-gaussian-sol-map} with $ \xi(t,x) =
\xi_{\mathrm{stat}}(\omega, x) $ and parameters $ \alpha =
-\frac{1}{2} - \ve,$ for some $ \ve \in (0, 1) $, and $\beta = \beta_{0} = \frac{1}{2} + \gamma $ to find that
for any $ \delta \in (0, 1) $ for a deterministic constant $ C(\delta, T)>0 $
\begin{equs}
 \| u (\omega) \|_{C^{\frac{1/2 + \gamma}{2}}_{T}
L^{\infty}}^{p} + \| u(\omega) \|_{ L^{\infty}_{T} \mC^{\frac{1}{2} +
\gamma}}^{p} \lesssim C \exp \Big( C p \| \xi_{\mathrm{stat}}(\omega) \|_{\mC^{- \frac{1}{2} - \ve }}^{
\frac{4}{3-2 \ve - \delta}}  \Big).
\end{equs}
If $ \ve, \delta $ are chosen sufficiently small we find that \(
\frac{4}{3 - 2 \ve - \delta} < 2 \). In this case we obtain $$ \EE \Big[\exp
\Big( C p \| \xi_{\mathrm{stat}} \|_{\mC^{- \frac{1}{2} - \ve }}^{\frac{4}{3-2 \ve - \delta}}
\Big) \Big] < \infty$$by Fernique's theorem \cite[Theorem 4.1]{Ledoux} since $ \xi_{\mathrm{stat}} $ is a Gaussian random
variable that satisfies $ \EE \| \xi_{\mathrm{stat}}
\|_{\mC^{-\frac{1}{2} - \ve}}^{2} < \infty $ for any $ \ve>0 $ (this follows
along standard computations, or by observing that $ \xi_{\mathrm{stat}} $ is
the distributional derivative of a periodic Brownian motion).
\end{proof} \\
Next we establish a uniform bound on the projective invariant measure.

\begin{lemma}\label{lem:moment-bound-invariant-msr-1D-1n}
Under Assumption~\ref{assu:white-noise}, for any $ \tau >0 $ let $
z_{\infty}(\tau) \colon \Omega \to \mathrm{Pr} $ be defined as in
Proposition~\ref{prop:existence-proj-invariant-measure}. Then for any $ p
\geqslant 2,  \gamma \in
(0, 1)$
\begin{equs}
\sup_{\tau \in (0, \infty)} \EE \| z_{\infty}(\tau) \|_{\mC^{\frac{1}{2} + \gamma}}^{p} < \infty.
\end{equs}
\end{lemma}

\begin{proof}
From the invariance of $ z_{\infty} $ we know that for any $ t > 0 $, $ z_{\infty}(\tau) =
u(n(t) \tau ) / \| u (n (t) \tau) \|_{L^{1}} $ in distribution, where $ n(t) =
\min \{ n \in \NN  \ \colon \  \tau n \geqslant t \} $ and for every $ \omega \in \Omega $,  $ u(\omega) $ is the solution to the equation
\begin{equs}
\partial_{t} u(\omega) = \Delta u(\omega) + \xi^{\tau}(\omega) u(\omega) ,
\qquad u(\omega, 0, \cdot) = z_{\infty}(\tau, \omega).
\end{equs}
Hence our result follows from an upper bound on the moments of $ \| u(n(t) \tau)
\|_{\mC^{\frac{1}{2} + \gamma}} $ and a lower bound on the moments of $ \|
u(n(t) \tau) \|_{L^{1}} $, for some $ t>0$ appropriately chosen, and uniformly
over $ \tau $. Here the point is
that at any positive time the heat semigroup has smoothened the initial
condition, from $ L^{1} $ to $ \mC^{\frac{1}{2} + \gamma} $, while at the same time
our bounds show that the total mass may decrease at most by a factor $ \exp(- C
n(t) \tau| x|^{2}) $, where $ x $ is roughly some linear
functional of the Gaussian noise and $ C>0 $ deterministic, so $ t $ needs
to be small enough to ensure the integrability of negative moments of this
quantity.

\textit{Step 1: Lower bound on $ \| u \|_{L^{1}} $.} Recall that by the
strong maximum principle $ u(\omega, \tau n(t), x) >0, \ \forall x \in \TT.$ We will show
that for every $ \omega \in \Omega $
\begin{equs}
\int_{\TT} u(\omega, \tau n(t), x) \ud x \geqslant C(\omega, t, \tau) \int_{\TT} u(0, x) \ud
x = C(\omega, t, \tau),
\end{equs}
where we used that \( \int_{\TT} z_{\infty}(\tau, x) \ud x =1 \), and
the crux of the argument will be that $ C(\omega, t, \tau) $ satisfies $
\sup_{\tau \in (0, 1)}  \EE \frac{1}{| C(t, \tau) |^{p}} < \infty $ for certain
combinations of $ p $ and $ t $.

In the following calculations we consider $
\omega \in \Omega $ fixed, so we omit writing the dependence on it.
Let $ \mathcal{I}(\xi^{\tau})(t) = \int_{0}^{t} P_{t-s} (\xi^{\tau}(s)) \ud s
$. Then we can decompose $ u_{t} = e^{\mI(\xi^{\tau})_{t}} w_{t} $, with $
w $ the solution to
\begin{equs}
\partial_{t} w = \Delta w + 2 (\partial_{x} \mI(\xi^{\tau}))
\partial_{x} w + ( \partial_{x} \mI(\xi^{\tau}))^{2} w, \qquad w(0, x) =
z_{\infty}(\tau, x).
\end{equs}
By Lemma~\ref{lem:bounds-linear-solution-wn}, $ \mI(\xi^{\tau}) $ takes values
in $ \mC^{1 + \gamma} , $ for any $ \gamma \in (0, 1/2) $. Hence let us define
$ A_{t} = \| \mI(\xi^{\tau}) \|_{L^{\infty}_{t}
\mC^{1 + \gamma}(\TT)} $. By comparison we find that \( w_{t} (x) \geqslant
\widetilde{w}_{t}(x), \) with $ \widetilde{w} $ the solution to
\begin{equs}
\partial_{t} \widetilde{w} = \Delta \widetilde{w} + 2(\partial_{x}
\mI(\xi^{\tau})) \partial_{x} \widetilde{w}, \qquad \widetilde{w}(0, x) =
z_{\infty}(\tau, x).
\end{equs}
We can write $ \widetilde{w} (t, x) = \int_{\TT} \Gamma_{t}(x, y)
z_{\infty}(\tau, y) \ud y $, where $ \Gamma_{t} $ is the fundamental solution
to the previous PDE:
\begin{equs}
\partial_{t} \Gamma = \Delta_{x} \Gamma + 2 (\partial_{x} \mI(\xi^{\tau}))
\partial_{x} \Gamma , \qquad \Gamma_{0}(x, y) = \delta_{y}(x),
\end{equs}
with \( \delta_{y} \) the Dirac delta function centered at $ y. $ Now one can
find quantitative lower bounds to $ \Gamma $ in terms of the heat kernel, see
e.g.\ \cite[Theorem 1.1]{perkowski2020quantitative}. The quoted article
considers the more complicated setting of a
distribution valued drift on infinite volume, but the same arguments show that
\begin{equs}
\Gamma_{t}(x, y) \geqslant C_{1} \exp \Big( - t C_{2} A_{t}^{2} \Big)
p_{ \kappa t}(x-y),
\end{equs}
where $ C_{1}, C_{2}, \kappa > 0 $ are deterministic constants and $
p_{t}(x) $ is the periodic heat kernel. In particular, we obtain
\begin{equs}
\int_{\TT} \widetilde{w}_{t}(x) \ud x \geqslant C_{1} e^{- C_{2} t
A_{t}^{2}} \int_{\TT} z_{\infty}(y) \ud y = C_{1} e ^{- C_{2} t A_{t}^{2}}.
\end{equs}
Overall, we have obtained that
\begin{equs}
\int_{\TT} u_{t} (x) \ud x & \geqslant e^{- t A_{t}} \int_{\TT} w_{t}(x) \ud x
\geqslant \exp \Big( - t  A_{t} \Big) \int_{\TT}
\widetilde{w}_{t}(x) \ud x \geqslant C_{1} \exp \Big( - t \Big( A_{t} + C_{2} A_{t}^{2} \Big) \Big).
\end{equs}
Now, by Lemma~\ref{lem:bounds-linear-solution-wn} there exists a $
\sigma(\gamma) > 0 $ such that for any $ p
\geqslant 1 $ and $ t_{*}(p) = \frac{\sigma}{2 (1 + C_{2} ) p} \wedge 1 $ we have
\begin{equs}[eqn:prf-invariant-wn-1]
\sup_{\tau \in (0, \infty)} \EE \sup_{0 \leqslant t \leqslant t_{*}(p)} & \| u_{t}
\|_{L^{1}}^{- 2 p} \\
& \lesssim \sup_{\tau \in (0, \infty)} \EE \exp \Big(
\sigma  \| \mI(\xi^{\tau}) \|_{C([0, 1]; \mC^{1 +\gamma} ) }^{2}  \Big) < \infty.
\end{equs}

\textit{Step 2: Upper bound on $ \| u \|_{\mC^{\frac{1}{2} + \gamma}} $.} 
Let us start by observing that for any $ \ve > 0, \ q \geqslant 1 $
\begin{equs}[eqn:prf-wn-invariant-uniform-bound-Lq]
\sup_{\tau \in (0, \infty)}\EE \int_{0}^{1} \| \xi^{\tau}(s) \|^{q}_{\mC^{-
\frac{1}{2} - \ve}} \ud s = \EE \| \xi_{\mathrm{stat}} \|_{\mC^{-
\frac{1}{2} - \ve}}^{q} < \infty.
\end{equs}
Hence we see that $ \xi^{\tau} $ takes values in $ L^{q}([0, 1]; \mC^{-
\frac{1}{2} - \ve}) $ for all $ \ve>0 , q \geqslant 1 $. In particular, for any
$ t \in (0,1) $ we can apply
Lemma~\ref{lem:sub-gaussian-sol-map} to obtain that for all $ \ve >0 $
sufficiently small
\begin{equs}
\| u_{t/2} \|_{B^{\frac{3}{2} - 2 \ve }_{1, \infty}} \leqslant C_{3}(\ve,
t) \| u_{0} \|_{L^{1}} \exp \Big( C_{4}(\ve) \| \xi^{\tau} \|^{\frac{4}{3-4
\ve}}_{L^{q(\ve)}( [0, 1]; \mC^{- \frac{1}{2} - \ve} )} \Big),
\end{equs}
where $ C_{3}, C_{4} > 0 $ are deterministic constants and we allow $ C_{3} $
to depend on $ t $ to incorporate the explosion at time $ t = 0 $. Now, by
Besov embedding we have that $ B^{\frac{3}{2} - 2 \ve}_{1, \infty} \subseteq
\mC^{\frac{1}{2} - 2 \ve} $, so that we can follow the same argument on the
interval $ [t/2, t] $ to obtain (up to increasing the value of $ C_{3},
C_{4} $):
\begin{equs}
\| u_{t} \|_{\mC^{\frac{3}{2} - 2 \ve}} & \leqslant C_{3}(\ve,
t) \| u_{t/2} \|_{L^{\infty}} \exp \Big( C_{4}(\ve) \| \xi^{\tau} \|^{\frac{4}{3-4
\ve}}_{L^{q(\ve)}([0,1]; \mC^{- \frac{1}{2} - \ve} )} \Big) \\
& \leqslant  C_{3}(\ve,
t)^{2} \| u_{0} \|_{L^{1}} \exp \Big( 2 C_{4}(\ve) \| \xi^{\tau} \|^{\frac{4}{3-4
\ve}}_{L^{q(\ve)}( [0, 1]; \mC^{- \frac{1}{2} - \ve} )} \Big).
\end{equs}
Since $ \| u_{0} \|_{L^{1}} = 1 $, and since for $ \ve $ small both $
\frac{4}{3-2 \ve} < 2 $ and $ \frac{1}{2} + \gamma \leqslant \frac{3}{2}
- 2 \ve $, by Fernique's theorem \cite[Theorem 4.1]{Ledoux} and
\eqref{eqn:prf-wn-invariant-uniform-bound-Lq}, for any $ a > 0 $
\begin{equs}
\sup_{\tau \in (0, \infty)} \EE \sup_{t \in (a,1)}  \| u_{t} \|_{\mC^{\frac{1}{2} +
\gamma}}^{2p} < \infty.
\end{equs}

\textit{Step 3: Conclusion.} Now there exists a $ \tau_{*} (p) $ such that for
all $ \tau \in (0, \tau_{*}(p)) $ we have
\begin{equs}
\tau n( t_{*}(2p)/2) \leqslant t_{*}(2p).
\end{equs}
Then the results at the previous points imply:
\begin{equs}
\sup_{\tau \in (0, \tau_{*}(p)]} \EE \| z_{\infty}(\tau) \|_{\mC^{\frac{1}{2} +
\gamma}}^{p} & = \sup_{\tau \in (0, \tau_{*}(p)]} \EE \|
u_{\tau n ( \tau_{*}(2p)/2 ) } / \| u_{\tau n ( \tau_{*}(2p)/2 ) } \|_{L^{1}}
\|_{\mC^{\frac{1}{2} + \gamma}}^{p} \\
& \leqslant \sup_{\tau \in (0, \tau_{*}(p)]} \Big( \EE \Big[ \| u_{\tau n ( \tau_{*}(2p)/2 ) } \|_{\mC^{\frac{1}{2} + \gamma}}^{2p} \Big]
\Big)^{\frac{1}{2}} \Big( \EE \Big[ \| u_{\tau n ( \tau_{*}(2p)/2 )} \|_{L^{1}}^{- 2 p} \Big]
\Big)^{\frac{1}{2}} < \infty.
\end{equs}
To conclude the proof of the lemma we have to consider the case $ \tau \in
(\tau_{*}(p), \infty) $. Here we observe that in
all the bounds in steps $ 1 $ and $ 2 $ we did not use any other information on
the initial condition $ z_{\infty} (\tau) $ than $ \| z_{\infty} (\tau)
\|_{L^{1}}  = 1.$ Then for $ \zeta =  \tau_{*}(p) \wedge t_{*}(2p) $ we have
$ z_{\infty}(\tau) \stackrel{d}{=} S_{\tau} (\zeta) u_{0}  / \| S_{\tau}
(\zeta) u_{0} \|_{L^{1}} $, where $ u_{0} = S_{\tau}(\tau - \zeta) z_{\infty}(\tau) / \| S_{\tau}
(\tau - \zeta) z_{\infty}(\tau) \|_{L^{1}}   $ and $ S_{\tau} $ is the solution map to
Equation~\eqref{eqn:time-Anderson} with $ \xi = \xi^{\tau} $, chosen
independent of $ z_{\infty}(\tau) $. Then we can follow verbatim the
calculations above, by using that $ \| u_{0} \|_{L^{1}} =1 $ to obtain the
required result. 
\end{proof} \\
A consequence of this result if the following bound on the largest eigenvalue
of the operator $ \Delta + \xi_{\mathrm{ stat}} $.

\begin{corollary}\label{cor:bound-eigenvalue}
Under Assumption~\ref{assu:white-noise}, let $ \gamma $ be the largest eigenvalue of the operator $ \Delta +
\xi_{\mathrm{stat}} $, as in Lemma~\ref{lem:eigenfunction}. Then there exists a
$ \sigma > 0 $ such that
\begin{equs}
 \EE e^{\sigma | \gamma |} < \infty.
\end{equs}
\end{corollary}

\begin{proof}
We can bound $ \EE e^{\sigma \gamma} < \infty$ for all $ \sigma > 0 $ by
similar calculations as in the upper bound presented in Step $ 2 $ of the proof
of Lemma~\ref{lem:moment-bound-invariant-msr-1D-1n}. In addition, there exists
a $ \sigma_{*} $ such that for all \( \sigma \in (0, \sigma_{*}) \) we have $
\EE e^{- \sigma \gamma} < \infty $ by following the same arguments that lead to
\eqref{eqn:prf-invariant-wn-1}.
\end{proof} \\
The previous bound builds on the following estimate on the linear equation with
additive noise.
\begin{lemma}\label{lem:bounds-linear-solution-wn}
Consider $ \mathcal{I}(\xi^{\tau}) $ defined by
\begin{equs}
\mathcal{I}(\xi^{\tau})(t) = \int_{0}^{t} P_{t-s}( \xi^{\tau} (s)) \ud s.
\end{equs}
Then for any $ \gamma \in (0, 1) $ and $ T>0 $ there exists a $
\sigma(\gamma, T) > 0 $ such that:
\begin{equs}
 \sup_{\tau \in (0, \infty) }  \EE \exp \Big( \sigma(\gamma, T) \|
\mathcal{I}(\xi^{\tau})  \|^{2}_{C([0, T];
\mC^{\frac{1}{2} + \gamma})} \Big) < \infty.
\end{equs}
\end{lemma}

\begin{proof}
Our aim is to apply the Kolmogorov continuity criterion to control the time
continuity of $ \mathcal{I}(\xi^{\tau}) $. We can decompose an increment of the process as:
\begin{equs}
\mI(\xi^{\tau})(t) - \mI(\xi^{\tau})(r) = \int_{r}^{t}
P_{t-s}(\xi^{\tau}(s)) \ud s + (P_{t-r} - \mathrm{Id}) \int_{0}^{r}
P_{r-s}(\xi^{\tau}(s))  \ud s.
\end{equs}
For any $ \delta \in (0, 1) $ define $ \zeta =
\frac{1}{2} + \gamma + 2 \delta $. Then by Lemma~\ref{lem:schauder-estiamtes}:
\begin{equs}
\EE \| \mathcal{I}(\xi^{\tau})(t) - \mI(\xi^{\tau})(r)
\|_{\mC^{\frac{1}{2} + \gamma}}^{p} \lesssim & \EE \bigg( \int_{r}^{t}\|
P_{t-r}(\xi^{\tau}(s))\|_{\mC^{\frac{1}{2} + \gamma}}  \ud s \bigg)^{p} \\
& + (t-s)^{\delta p} \EE \bigg( \bigg\| \int_{0}^{r} P_{r-s}(\xi^{\tau}(s))
 \ud s \bigg\|_{\mC^{\zeta}}^{p}\bigg).
\end{equs}
Now we can assume $ \delta >0$ sufficiently small and $ p \geqslant 2 $
sufficiently large, so that for some $ \ve \in (0, 1) $ and $ \zeta^{\prime} =
\frac{3}{2} - 3 \ve $ we have the continuous embedding $ B^{\frac{3}{2} -
\ve}_{p,p} \subseteq \mC^{\zeta}  $. Then
\begin{equs}
\EE \bigg( \bigg\| \int_{r}^{t} P_{t-s} \xi^{\tau}(s) \ud s \bigg\|_{\mC^{\zeta}}^{p}
\bigg) & \lesssim \EE \bigg( \bigg\|
\sum_{i = n(r)-1}^{n(t)-1} \int_{i \tau \vee r}^{(i+1) \tau
\wedge t} P_{t-s} 
\xi_{\mathrm{stat}}^{i} \ud s \bigg\|_{B^{\zeta^{\prime}}_{p,p}}^{p}\bigg) \\
& \lesssim \sum_{j \geqslant -1} 2^{ j \zeta^{\prime} p}\int_{\TT}  \EE \bigg( \bigg| \sum_{i =
n(r)-1}^{n(t)-1} \int_{i \tau \vee r}^{(i+1) \tau \wedge t} \langle P_{t-s} 
\xi_{\mathrm{stat}}^{i}, K_{j}^{x} \rangle \ud s \bigg|^{p}\bigg) \ud
x,
\end{equs}
where $ n(t) = \min \{ n \in \NN  \ \colon \ \tau n \geqslant t\} $ and $
K_{j}^{x} $ is as in \eqref{eqn:paley-block}.
Next, since over $ i $ we have a sum of
independent random variables, we can use Rosenthal's inequality
\cite[Theorem 2.9]{Petrov1995} to bound uniformly in $ x , j $, similarly to the proof of
Lemma~\ref{lem-convergence-to-zero-Xi}:
\begin{equs}
\EE \bigg( \bigg| \sum_{i =
n(r)-1}^{n(t)-1} \int_{i \tau \vee r}^{(i+1) \tau \wedge t} & \langle P_{t-s} 
\xi_{\mathrm{stat}}^{i}, K_{j}^{x} \rangle \ud s \bigg|^{p}\bigg) \\
 & \lesssim_{p} 2^{ - \big( \frac{3}{ 2} - 2 \ve \big) j p} \bigg\{ \sum_{i = n(r)-1}^{n(t)-1}  \EE \Big[ \|
\xi_{\mathrm{stat}} \|_{\mC^{- \frac{1}{2} - \ve} }^{p} \Big] \bigg( \int_{i \tau \vee r}^{(i+1) \tau
\wedge t} (t-s)^{- \frac{2 - \ve}{2} } \ud s \bigg)^{p} \\
& \qquad \quad \quad \quad + \bigg( \sum_{i = n(r)-1}^{n(t)-1}\EE \Big[ \|
\xi_{\mathrm{stat}} \|_{\mC^{- \frac{1}{2} - \ve} }^{2} \Big] \bigg( \int_{i \tau \vee r}^{(i+1) \tau
\wedge t} (t-s)^{- \frac{2 - \ve}{ 2} } \ud s \bigg)^{2} \bigg)^{\frac{p}{2}} \bigg\} \\
& \lesssim  2^{ - \big( \frac{3}{ 2} - 2 \ve \big) j p} \bigg(
\int_{r}^{t}(t-s)^{- \frac{2 - \ve}{ 2} } \ud s \bigg)^{p}\simeq  2^{ - \big(
\frac{3}{ 2} - 2 \ve \big) j p} (t-r)^{p \ve /2},
\end{equs}
where we used the bound $ \sum_{i} a_{i}^{p} \leqslant \Big( \sum_{i} a_{i}
\Big)^{p} $ for $ a_{i} \geqslant 0 $ together with the bound
\begin{equs}
| \langle P_{t - s} \xi_{\mathrm{stat}}^{i}, K_{j}^{x} \rangle | \lesssim \|
P_{t- s} \xi_{\mathrm{stat}}^{i} \|_{\mC^{\frac{3}{2} - 2 \ve}} 2^{ -
( \frac{3}{2}- 2 \ve)j} \lesssim (t -s)^{- \frac{2 - \ve}{2}} \|
\xi_{\mathrm{stat}}^{i} \|_{\mC^{-\frac{1}{2} - \ve}}.
\end{equs}
Putting all the bounds together, we have proven that for $ p \geqslant 1
$ sufficiently large there exists a $ \delta^{\prime} > 0 $ such that
\begin{equs}
\sup_{\tau \in (0, \infty)} \EE \| \mathcal{I}(\xi^{\tau})(t) - \mI(\xi^{\tau})(r) \|_{\mC^{\frac{1}{2} +
\gamma}}^{p} \lesssim (t-r)^{1 + \delta^{\prime} p}.
\end{equs}
Since $ \mI(\xi^{\tau})(0) = 0 $ by the Kolmogorov continuity criterion this implies
\begin{equs}[eqn:prf-mI-uniform-bound] 
\sup_{\tau \in (0, \infty)} \EE \| \mI(\xi^{\tau}) \|_{C([0, T]; \mC^{\frac{1}{2}
+ \gamma})}^{2} < \infty.
\end{equs}
Finally, since $ \mI(\xi^{\tau}) $ is a Gaussian process, an application of Fernique's theorem
\cite[Theorem 4.1]{Ledoux} together with the uniform bound \eqref{eqn:prf-mI-uniform-bound}
complete the proof of the result.
\end{proof}

\section{An analytic estimate}\label{sec:an-analytic-estimate}

In this section we prove an analytic bound that is useful to control the
invariant projective measure associated to Assumption~\ref{assu:white-noise}
uniformly over small $ \tau $.

\begin{lemma}\label{lem:sub-gaussian-sol-map}
Consider $ \alpha \in (0, 1) $ and $ T > 0 $. Let $ u $ be the unique solution to
\begin{equs}
(\partial_{t} - \Delta) u (t,x) = \xi (t,x) u (t,x), \qquad u(0, x) =
u_{0}(x),
\end{equs}
on $ [0, T] \times \TT $, with $ \xi \in L^{q}([0,T]; \mC^{-\alpha}) $ for all $ q \geqslant
1 $, and with $ u_{0} \in \mC^{\beta_{0}}_{p} $, for some $
\beta_{0} \geqslant -\alpha $ and \( p \in [1, \infty]\). Then for any $ \alpha <
\beta < 2- \alpha, \ \zeta = \frac{( \beta - \beta_{0})_{+}}{2} \in [0, 1)$ and any $ \delta \in (0, 2 - \alpha) $ there exists a $q= q(\alpha, \beta, \delta) \geqslant 1 $ such that 
\begin{equs}
\| t \mapsto t^{\zeta}u_{t} \|_{C^{\frac{\beta}{2}}_{T} L^{p}} + \| t \mapsto t^{\zeta} u_{t}
\|_{L^{\infty}_{T} \mC^{\beta}_{p}} \leqslant C \| u_{0} \|_{\mC^{\beta_{0}}_{p}} \exp \Big( C \| \xi \|_{L^{q}([0, T];
\mC^{-\alpha})}^{\frac{2}{2 - \alpha - \delta }  } \Big),
\end{equs}
for a constant $ C = C(T, p, q, \alpha, \beta, \delta) > 0 $ independent of $ \xi $ and $
u_{0} $.
\end{lemma}

\begin{remark}\label{rem:optimal-analytic-bound}
This result is a bit in the spirit of the quantitative estimates in
\cite{perkowski2020quantitative}.
That the exponential growth depends on the regularity of $ \xi $ should be
expected. For example, our estimate is in line -- up to the small factor $
\delta $ -- with the fact that the solution
to the $ 2D $ Anderson model driven by space white noise is integrable only for
short times, see e.g. \cite{gu2018moments}. In this case we would have $ \alpha = 1 - \ve $ for any $ \ve > 0
$, so the above estimate would deliver slightly more than quadratic bounds --
albeit if $ \alpha < 1 $ the estimate would need to incorporate additional
stochastic terms used to construct solutions with, say, regularity structures.
\end{remark}

\begin{proof}
The solution $ u $ as in the statement can be represented in mild form:
\begin{equs}
u_{t} = P_{t} u_{0} + \int_{0}^{t} P_{t-s} ( u_{s} \xi_{s}) \ud s.
\end{equs}

\textit{Step 1.} We start by establishing a bound on the spatial regularity,
i.e.\ on $ t^{\zeta}\| u_{t} \|_{\mC^{\beta}_{p}}. $
To lighten the notation, for a Banach space $ X $ and a map \( f \colon [0, \infty) \to X \) and $ t
> 0, \zeta \in [0, 1) $ we write $ \| u \|_{X_{\zeta, t}} $ for the norm
\begin{equs}
\| f \|_{X_{\zeta, t}} = \sup_{0 \leqslant s \leqslant t } s^{\zeta} \| f_{s} \|_{X}.
\end{equs}
If $ \zeta =0 $ we omit writing the dependence on it.
With this notation we can use the product estimates of
Lemma~\ref{lem:paraproduct-estimate} (observing that $ \beta - \alpha >0 $, so
that the resonant product below is well-defined) together with the Schauder estimates
of Lemma~\ref{lem:schauder-estiamtes} to bound, for some $ q \geqslant 1 $ sufficiently large: 
\begin{align*}
t^{\zeta} \| u_{t} \|_{\mC^{\beta}_{p}} & \lesssim \| u_{0}
\|_{\mC^{\beta_{0}}_{p}} + t^{\zeta} \int_{0}^{t} \| P_{t - s} (u_{s} \para \xi_{s})
\|_{\mC^{\beta}_{p}} + \| P_{t-s} (u_{s} (\reso + \rpara) \xi_{s}) \|_{\mC^{\beta}_{p}}  \ud s \\
& \lesssim \| u_{0} \|_{\mC^{\beta_{0}}_{p}} + t^{\zeta} \int_{0}^{t}(t
-s)^{-\frac{1}{2} \big(\beta + \alpha \big)} s^{- \zeta} \| u \|_{L^{p}_{\zeta, t}} \|
\xi_{s} \|_{\mC^{-\alpha}} \\
& \qquad \qquad \qquad \qquad + (t -s)^{- \frac{\alpha}{2}} s^{- \zeta} \| u
\|_{\mC^{\beta}_{p, \zeta, t}} \| \xi_{s} \|_{\mC^{-
\alpha}} \ud s \\
 & \lesssim \| u_{0} \|_{\mC^{\beta_{0}}_{p}} + t^{ \frac{1}{q^{\prime}} - \frac{\beta +
\alpha}{2}} \| u \|_{L^{p}_{\zeta, t}} A + t^{ \frac{1}{q^{\prime}} -
\frac{\alpha}{2} } \| u \|_{\mC^{\beta}_{p, \zeta, t}} A,
\end{align*} 
with $ A = \| \xi \|_{L^{q}_{T} \mC^{- \alpha}},$ and $ \frac{1}{q^{\prime}} +
\frac{1}{q} =1$. Here we have used that for
any $ \mu \in (0, 1)$ and $ f \colon [0, T] \to [0, \infty) $, by H\"older and a change of variables:
\begin{equs}
t^{\zeta}\int_{0}^{t} (t-s)^{- \mu} s^{- \zeta} f_{s} \ud s & \leqslant
t^{\zeta} \bigg(\int_{0}^{t} (t-s)^{- \mu q^{\prime}} s^{- \zeta q^{\prime}}
\ud s \bigg)^{\frac{1}{q^{\prime}}} \| f \|_{L^{q}_{t}}\\
&  \leqslant  t^{\frac{1}{q^{\prime}} - \mu} \bigg(
\int_{0}^{1} (1 - s)^{- \mu q^{\prime} } s^{- \zeta q^{\prime}} \ud s \bigg)^{\frac{1}{q^{\prime}}}
\| f \|_{L^{q}_{t}}  \lesssim t^{\frac{1}{q^{\prime}}  - \mu} \| f
\|_{L^{q}_{t}},
\end{equs}
where in the last step we assumed that $ q = q (\mu, \zeta) \geqslant 1 $ is
large enough so that the time integral we wrote is finite.
Since the right hand-side of our previous bound is increasing in $ t $ we can take the supremum over
all times up to $ t $ on the left hand-side to obtain 
\begin{equs}
\| u \|_{\mC^{\beta}_{p, \zeta, t}} & \lesssim \| u_{0}
\|_{\mC^{\beta_{0}}_{p}} + t^{ \frac{1}{q^{\prime}} - \frac{\beta +
\alpha}{2} } \| u \|_{L^{p}_{\zeta, t}} A + t^{\frac{1}{q^{\prime}} -
\frac{\alpha}{2} } \| u \|_{\mC^{\beta}_{p, \zeta, t}} A.
\end{equs}
In particular, let $ \alpha_{1} = \frac{1}{q^{\prime}} - \frac{\beta +
\alpha}{2} $. Then, since $ t^{\frac{1}{q^{\prime}} - \frac{\alpha}{2}}
\leqslant T^{\frac{\beta}{2}} t^{\alpha_{1}} $, there
exists a $ t_{*}(T) $, independent of $ A $, such that
\begin{equs}
\| u \|_{\mC^{\beta}_{p, \zeta, t}} \lesssim \| u_{0}
\|_{\mC^{\beta_{0}}_{p}}, \qquad  \forall t \in \left(0,
\frac{t_{*}}{A^{\frac{1}{\alpha_{1}}}} \right].
\end{equs}
This estimate guarantees that 
\begin{equs}[eqn:prf-moment-1d-small-time]
\sup_{0 < t \leqslant t_{*}/ A^{\frac{1}{\alpha_{1}}}} t^{\zeta}\| u_{t}
\|_{\mC^{\beta}_{p}} \lesssim \| u_{0}
\|_{\mC^{\beta_{0}}_{p}},
\end{equs}
so that in particular $ \| u_{t_{*} / A^{\frac{1}{\alpha_{1}}}}
\|_{\mC^{\beta}_{p}} \lesssim A^{\frac{\zeta}{\alpha_{1}}} \| u_{0}
\|_{\mC^{\beta_{0}}_{p}} $,
which tells us that by time $ t_{*} / A^{\frac{1}{\alpha_{1}}} $ the heat
semigroup has smoothened the initial condition and the regularity of the
solution is now governed by the forcing. Following this idea, we bound
the solution for times larger than $ t_{*} / A^{\frac{1}{\alpha_{1}}} $
differently (assuming $ t_{*} / A^{\frac{1}{\alpha_{1}}} \leqslant T $,
otherwise the proof is complete). Let us define $ v(t) = u (t_{*}/ A^{\frac{1}{\alpha_{1}}} + t) $.
We can follow the previous steps with $ \beta_{0} = \beta $ and
\(\zeta = 0\) to obtain for $q$ sufficiently large:
\begin{equs}[eqn:prf-moment-1d-1]
\| v \|_{\mC^{\beta}_{p,t}} & \lesssim \| v_{0}
\|_{\mC^{\beta}_{p}} + t^{\alpha_{1}} \| v \|_{L^{p}_{t}} A + t^{\alpha_{2}} \| v
\|_{\mC^{\beta}_{p, t}} A,
\end{equs}
with $ \alpha_{2} = \frac{1}{q^{\prime}} - \frac{\alpha}{2} $.
Now we would like to use Gronwall to obtain a bound that
depends only on $ v_{0} $ and $ A $. But this would lead to an estimate of the kind: \( \| v_{t}
\|_{\mC^{\frac{1}{2} + \gamma}} \lesssim C_1  \| v \|_{0} e^{ C_2 t
A^{ \frac{1}{\alpha_{1}}}}\) (see the discussion below). Since $
\alpha_{1} \simeq 1 - \frac{\beta + \alpha}{2} $ (for large $ q $) this is not of the correct
order for our result. Hence we have to take better care of
the $ \| v \|_{L^{p}_{t}} $ norm, to obtain roughly that the leading order term above
(for small $ t $) is of the order $ t^{\alpha_{2}} $, which would lead to
the required exponential bound. We find for any $ \ve > 0 $:
\begin{equs}
\| v_{t} \|_{L^{p}} & \lesssim \| v_{0} \|_{\mC^{\beta}_{p}} +
\int_{0}^{t} \| P_{t-s}( v_{s} \para \xi_{s}) \|_{L^{p}} + \| P_{t-s}
v_{s} (\reso + \rpara) \xi_{s} \|_{L^{p}} \ud s \\
& \lesssim \| v_{0} \|_{\mC^{\beta}_{p}} +\int_{0}^{t}(t-s)^{-
\frac{\alpha + \ve}{2} }  \| v \|_{L^{p}_{t}} \| \xi_{s}
\|_{\mC^{- \alpha}}  + \| v \|_{\mC^{\beta}_{p,
t}} \| \xi_{s} \|_{\mC^{-\alpha}} \ud s ,
\end{equs}
thus leading, for $ \ve >0 $ sufficiently small and $ q \geqslant 1 $
sufficiently large, to the bound:
\begin{equs}[eqn:prf-moment-1d-2]
& \| v \|_{L^{p}_{t}} \lesssim \| v_{0} \|_{\mC^{\beta}_{p}} +
t^{\frac{1}{q^{\prime}} - \frac{\alpha + \ve}{2} } \| v \|_{L^{p}_{t}}A + t \|
v \|_{\mC^{\beta}_{p, t}} A.
\end{equs}
Now define $ \alpha_{2}(\ve ) = \frac{1}{q^{\prime}} - \frac{\alpha +
\ve}{2}$ and fix $ \ve >0 $ small and $ q \geqslant 1 $ large so that
(for $ \delta $ as in the statement of the lemma):
\begin{equs}
\alpha_{2}(\ve) > 1 - \frac{\alpha + \delta/2}{2} : = \alpha_{3}. 
\end{equs}
In particular, we can fix \( n \in \NN \) such that $$ \alpha_{1} + n
\alpha_{2}(\ve) \geqslant  (n+1) \alpha_{3}.$$
We can then improve \eqref{eqn:prf-moment-1d-1} by plugging in the estimate
\eqref{eqn:prf-moment-1d-2} on the $ L^{p} $ norm and obtain:
\begin{equs} 
\| v \|_{\mC^{\beta}_{p, t}} & \lesssim \| v_{0}
\|_{\mC^{\beta}_{p}}( 1 + t^{\alpha_{1}}A) + \| v \|_{L^{p}_{t}}t^{\alpha_{1} +
\alpha_{2}(\ve) }  A^{2} +  \| v \|_{\mC^{\beta}_{p,
t}}(t^{\alpha_{2}} A + t^{1 + \alpha_{1}} A^{2}).
\end{equs}
In the rest of this calculation we can assume that $ t \in (0, 1) $, and that
$  \alpha_{2}(\ve) \leqslant \alpha_{2} $. Then, if we iterate this procedure another $ n -1 $ times by substituting
\eqref{eqn:prf-moment-1d-2} into the above bound, we obtain:
\begin{equs}[eqn:prf-moment-1d-almost-complete]
\| v \|_{\mC^{\beta}_{p, t}}  \lesssim &\| v_{0}
\|_{\mC^{\beta}_{p}} \Big( 1 + \sum_{i = 1}^{n} t^{
\alpha_{1} + (i-1) \alpha_{2}(\ve) }A^{i} \Big) + \| v \|_{L^{p}_{t}}t^{\alpha_{1} +
n \alpha_{2}(\ve) }  A^{n+1} \\
& \quad \quad \quad +  \| v \|_{\mC^{\beta}_{p, t}}\Big(t^{\alpha_{2}(\ve)} A +
\sum_{i = 1}^{n} t^{1+ \alpha_{1} +(i-1) \alpha_{2}(\ve) }A^{i+ 1}\Big), \qquad
\forall t \in (0, 1).
\end{equs}
Now, let us first work under the assumption
\begin{equs}
\beta \leqslant  \alpha + \delta/2.
\end{equs}
From the definition of $ \alpha_{1} $ we then find, provided $ q $ is
sufficiently large:
\begin{equs}
1 + \alpha_{1} \geqslant  1+ \frac{1}{q^{\prime}}  -\alpha - \delta /4 \geqslant 2 \Big(
1 - \frac{\alpha + \delta/2}{2} \Big) = 2 \alpha_{3}.
\end{equs}
We can then obtain from \eqref{eqn:prf-moment-1d-almost-complete}, for $ t \in
(0, 1) $:
\begin{equs}
\| v \|_{\mC^{\beta}_{p, t}} & \lesssim_{T, n} \| v_{0}
\|_{\mC^{\beta}_{p}} (1 + A^{n}) + \| v \|_{\mC^{\beta}_{p, t}} \Big[ 
t^{\alpha_{1}+(n-1) \alpha_{2}(\ve)} A^{n+1} + t^{\alpha_{2}(\ve)}A +
(t^{ \frac{1 + \alpha_{1}}{2} } A)^{2} \sum_{i = 1}^{n} (t^{
\alpha_{2}(\ve )} A)^{i-1} \big] \\
& \leqslant C(T, n) \| v_{0} \|_{\mC^{\frac{1}{2} + \gamma}_{p}} (1 + A^{n}) +
C(T, n)\| v \|_{\mC^{\frac{1}{2}
+ \gamma}_{p, t}} \Big[ ( t^{\alpha_{3}} A )^{n+1} + t^{\alpha_{3}}A +
(t^{\alpha_{3}} A)^{2} \sum_{i = 1}^{n} (t^{ \alpha_{3}} A)^{i-1} \big],
\end{equs}
for some $ C(T, n) >0 $. In particular, we can find a $ t_{*}^{\prime} >0  $
such that for $ 0 < t \leqslant (t_{*}^{\prime} /A^{\frac{1}{\alpha_{3}}} ) \wedge
1$ and uniformly over $ A $ one has (up to increasing the value of $
C(T, n) $)
\begin{equs}
\| v \|_{\mC^{\beta}_{p, t}} \leqslant C(T, n)(1 + A^{n}) \| v_{0}
\|_{C^{\beta}_{p}}(1 +t^{\alpha_{3}} A), \qquad \forall t,A >0
\ \colon \ t\leqslant (t_{*}^{\prime} /A^{\frac{1}{\alpha_{3}}}) \wedge 1.
\end{equs}
Now, using the linearity of the equation and iterating this bound on small
intervals of length $ (t_{*}^{\prime} / A^{\frac{1}{\alpha_{3}}}) \wedge 1 $, one finds (once again up to increasing the value of $
C(T, n) $)
\begin{equs}
\| v \|_{\mC^{\beta}_{p,T}} \lesssim \| v_{0}
\|_{\mC^{\beta}} \exp \Big( C(T, n) (1 + A^{\frac{1}{\alpha_{3}}}) \log{(1 + A^{n})}\Big).
\end{equs}
We can now use the definition of $ v $ together with the small-times bound
\eqref{eqn:prf-moment-1d-small-time} on $ u $ to deduce that (up to taking a
larger $ n $):
\begin{equs}[eqn:prof-moment-1d-final]
\| u \|_{\mC^{\beta}_{p, \zeta, T}} & \lesssim_{T} \| u_{0}
\|_{\mC^{\beta}_{p}} \exp \Big( C(T, n)(1 + A^{\frac{1}{\alpha_{3}}})
\log{(1 + A^{n})}\Big)\\ 
& \lesssim \| u_{0} \|_{\mC^{\beta}_{p}} \exp \Big( C(T, n)
A^{\frac{2}{2 - \alpha - \delta}}\Big),
\end{equs}
where in the last step we chose a possibly larger $ C(T, n) $ and used that
\begin{equs}
A^{\frac{1}{\alpha_{3}}} \log{(1 + A^{n})} \lesssim 1 +
A^{\frac{1}{\alpha_{3}- \delta/4} } = 1 + A^{\frac{2}{2 - \alpha - \delta}}.
\end{equs}
This concludes the proof of the result in the case $ \beta \leqslant \alpha +
\delta/2 $. We can build on this result to complete the proof of the bound on
the spatial regularity. Fix any $ \beta \in (\alpha, 2 - \alpha)$, then, by the bound we
just proved and \eqref{eqn:prf-moment-1d-1} we have
\begin{equs}
\| v \|_{\mC^{\beta}_{p,t}} & \lesssim_{T} \| v_{0}
\|_{\mC^{\beta}_{p}} +  B(A) + t^{\alpha_{2}} \| v
\|_{\mC^{\beta}_{p, t}} A,
\end{equs}
with $ B(A) =  \| u_{0} \|_{\mC^{\beta}_{p}} A \exp \Big( C(T, n)(1 + A^{\frac{1}{\alpha_{3}}})
\log{(1 + A^{n})}\Big)$. Hence, once more by a Gronwall type argument we
obtain:
\begin{equs}
\| v \|_{\mC^{\beta}_{p,t}} & \lesssim \Big( \| v_{0}
\|_{\mC^{\beta}_{p}} + B(A) \Big) \exp \Big( C(T) A^{\frac{1}{\alpha_{2}}}
\Big).
\end{equs}
Since $ \frac{1}{\alpha_{2}} \leqslant \frac{2}{2 - \alpha - \delta}$, for $
q \geqslant 1 $ sufficiently large, our claim now follows along the same
arguments that led to \eqref{eqn:prof-moment-1d-final}, which completes the
proof of the bound for the spatial regularity.

\textit{Step 2.} Finally, the bound on the temporal regularity can be deduced from the
spatial bound we just proved, by applying for instance
\cite[Lemma 6.6]{GubinelliPerkowski2017KPZ}.
\end{proof}

\section{Stochastic estimates}\label{sec:stochastic-estimates}
In this section we establish the stochastic estimates necessary for the Taylor
expansion of the Furstenberg formula near zero.
The next result establishes the key stochastic estimate for the proof of
Lemma~\ref{lem:regular-noise-small-time}.

\begin{lemma}\label{lem:moment-estimate-for-zeta} In the setting
of Assumption~\ref{assu:smooth-noise} and Definition~\ref{def:hamiltonians}, define for any $ \tau \in (0, 1)$
and $ (\omega, \omega^{\prime}) \in \Omega \times \Omega $ the random variable $\zeta(\tau,
\omega, \omega^{\prime})$ as:
  \[ \zeta (\tau, \omega, \omega') = \int_0^{\tau} \int_{\TT} H
     (\omega) e^{s H (\omega)} (z_{\infty} (\tau, \omega')) (x) \ud
     x \ud s, \]
  where $z_{\infty}(\tau, \omega^{\prime}) $ is defined as in
Proposition~\ref{prop:existence-proj-invariant-measure}. Then there exists a $
\gamma>0 $ such that:
  \[ \EE^{\PP \otimes \PP} \left[ \zeta (\tau) -
     \frac{\zeta^2 (\tau)}{2} \right] =
     \tau^2 \mf{r}(\tau) +\mathcal{O} \left( \tau^{2 + \gamma} \right),
  \]
  with:
  \[ \mathfrak{r} (0) \assign \lim_{\tau \rightarrow 0^+} \mathfrak{r}
     (\tau) = \frac{1}{4} \int_{\TT^2} \EE [| \xi (x) - \xi
     (y) |^2] \ud x \ud y > 0. \]
\end{lemma}

\begin{proof}
  \textit{Step 1: Estimate for the first moment.} First we take the
  expectation over $\omega'$. Hence define
  \[ \bar{z}_{\infty} (\tau, x) = \int_{\Omega} z_{\infty}
     (\tau, \omega', x) \ud \PP (\omega'), \]
  and observe that $\bar{z}_{\infty} \in L^1 (\TT)$ with $\bar{z}_{\infty}
  (\tau, x) \geqslant 0$ for all $x \in \TT$ and by Fubini $\int_{\TT}
  \bar{z}_{\infty} (\tau, x) \ud x = 1$. Then by integration by
  parts we obtain
  \begin{equs}
    \mathbf{E}^{ \PP \otimes \PP} [\zeta (\tau)] & = 
    \mathbf{E}^{ \PP } \left[ \int_0^{\tau} \int_{\TT} H (\omega)
    e^{s H (\omega)} (\bar{z}_{\infty} (\tau)) (x) \ud x \ud s
    \right]\\
    & =  \int_0^{\tau} \int_{\TT} \int_{\Omega} \xi_{\mathrm{stat}}
    (\omega, x) [e^{s H (\omega)} (\bar{z}_{\infty} (\tau))] (x) \ud
    \PP (\omega) \ud x \ud s.
  \end{equs}
  Now we use the Feynman Kac formula to represent the semigroup $e^{s H
  (\omega)} (z)$. Here $\mathbf{E}_x^{\mathbf{Q}}$ indicates the average
  w.r.t. a periodic Brownian motion $B_t$ started in $x \in \TT$, so that
  \begin{equs}
     \xi_{\mathrm{stat}} (\omega, x) [e^{s H (\omega)}
    (\bar{z}_{\infty} (\tau))] (x) & = 
    \xi_{\mathrm{stat}} (\omega, x) \mathbf{E}_x^{\mathbf{Q}} \left[
    \bar{z}_{\infty} (\tau, B_s) \exp \left( \int_0^s \xi_{\mathrm{stat}}
    (\omega, B_r) \ud r \right) \right]  \\
    =  \xi_{\mathrm{stat}} (\omega, x) &
    \mathbf{E}_x^{\mathbf{Q}} \left[ \bar{z}_{\infty} (\tau, B_s)
    \left( 1 + \int_0^s \xi_{\mathrm{stat}} (\omega, B_r) \ud r + R_1 (s, B,
    \omega) \right) \right] .  
  \end{equs}
  For the rest term
  \[ R_1 (s, B, \omega) = \exp \left( \int_0^s \xi_{\mathrm{stat}} (\omega, B_r)
     \ud r \right) - \left( 1 + \int_0^s \xi_{\mathrm{stat}} (\omega, B_r)
     \ud r \right) \]
  we can use Taylor to estimate
  \[ | R_1 (s, \omega) | \leqslant \frac{1}{2} \exp (s \| \xi_{\mathrm{stat}}
     (\omega) \|_{\infty}) (s \| \xi_{\mathrm{stat}} (\omega) \|_{\infty})^2, \]
  so that
  \[ \left| \int_{\Omega} \xi_{\mathrm{stat}} (\omega, x)
     \mathbf{E}_x^{\mathbf{Q}} [\bar{z}_{\infty} (B_s) R_1 (s, \omega)] \ud \PP
(\omega)
     \right| \lesssim s^2 \mathbf{E} [\| \xi_{\mathrm{stat}} \|_{\infty}^3 \exp
     (s \| \xi_{\mathrm{stat}} \|_{\infty})] \mathbf{E}_x^{\mathbf{Q}}
     [\bar{z}_{\infty} (\tau, B_s)] . \]
  And the average is finite uniformly over $s$ in a bounded set by the moment
bound in Assumption~\ref{assu:smooth-noise}. Now we observe that the Lebesgue
measure on $ \TT$ is invariant for $B_s$, so that
  \begin{equation}
    \int_{\TT} \mathbf{E}_x^{\mathbf{Q}} [\bar{z}_{\infty} (\tau, B_s)]
    \ud x = \int_{\TT} \bar{z}_{\infty} (\tau, x) \ud x = 1,
    \label{eqn:zeta-bound-proof-z-inf-L1-bound}
  \end{equation}
  from the definition of $\bar{z}_{\infty}$, so that
\begin{equs}
\Big| \int_0^{\tau} \int_{\TT} \xi_{\mathrm{stat}} (\omega, x)
    \mathbf{E}_x^{\mathbf{Q}} [\bar{z}_{\infty} (\tau, B_s)  R_1 (s,
    \omega)] \ud \PP (\omega) \ud x \ud s \Big| & \lesssim \int_0^{\tau} s^2 \int_{\TT} \mathbf{E}_x^{\mathbf{Q}}
    [\bar{z}_{\infty} (\tau, B_s)] \ud x \ud s\\
    & =  \mathcal{O} (\tau^3) .
\end{equs}
  Next we would like to replace \( \int_{0}^{s} \xi_{\mathrm{stat}}(
B_{r}) \ud r \) by \( s \xi_{\mathrm{stat}}(x) \). We follow two different
approaches, depending on the regularity of $ \xi_{\mathrm{stat}} $. First
assume that \eqref{eqn:holder-noise} holds. Then for any $ \ve \in (0, 1/2) $
and $ s \in [0, 1] $
\begin{equs}
\bigg\vert \int_{0}^{s} \xi_{\mathrm{stat}}(B_{r}) -
\xi_{\mathrm{stat}}(x) \ud r \bigg\vert  \leqslant \| \xi_{\mathrm{stat}}
\|_{C^{\alpha}} \int_{0}^{s} | B_{r} -x |^{\alpha} \ud s &  \leqslant \| \xi_{\mathrm{ stat}} \|_{C^{\alpha}} \| B
\|_{C^{\frac{1}{2} - \ve}([0, 1])} \int_{0}^{s} r^{\alpha(\frac{1}{2}
- \ve)} \ud r \\
& \lesssim \| \xi_{\mathrm{ stat}} \|_{C^{\alpha}} \| B
\|_{C^{\frac{1}{2} - \ve}([0, 1])} s^{1 + \alpha(\frac{1}{2} - \ve)}.
\end{equs}
Hence, using that $ \EE^{\QQ}_{x} \| B \|_{C^{\frac{1}{2} - \ve}(\TT)} <
\infty $, we obtain
\begin{equs}
\Big| \mathbf{E}_x^{\mathbf{Q}} \Big[ \bar{z}_{\infty} (\tau,
    B_s) \Big( \int_0^s \xi_{\mathrm{stat}} (\omega, B_r) - & \xi_{\mathrm{stat}}
    (\omega, x) \ud r \Big) \Big] \Big|  \lesssim s^{1 + \alpha(\frac{1}{2} - \ve)} \| \xi_{\mathrm{stat}} \|_{\infty} \|
\overline{z}_{\infty}(\tau) \|_{\infty}. \label{eqn:prf-moment-bound-smooth}
\end{equs}
Instead, if we assume that \eqref{eqn:piecewise-noise} holds, then $ \TT 
= \bigcup_{i = 1}^{\mf{n}} A_{i} $, with $ A_{i} $ disjoint intervals. In this
case, for every $ x \in \TT $ there exists an $ i $ such that $ x \in A_{i} $ and we
can define $ p(x) \in \partial A_{i} $ the nearest boundary point of $
A_{i} $ to $ x $. Then
\begin{equs}
\Big| \mathbf{E}_x^{\mathbf{Q}} \Big[ \bar{z}_{\infty} (\tau,
    B_s) \Big( \int_0^s & \xi_{\mathrm{stat}} (\omega, B_r)  - \xi_{\mathrm{stat}}
    (\omega, x) \ud r \Big) \Big] \Big|  \\
    & \lesssim s \| \xi_{\mathrm{stat}} \|_{\infty} \| \bar{z}_{\infty}
    (\tau) \|_{\infty}  \PP \left( \sup_{0 \leqslant r
    \leqslant s} | B_r - x | \geqslant \left| x - p(x) \right| \right),
\end{equs}
  which holds true because $\xi_{\mathrm{stat}}$ is constant on any interval $
A_{i} $. Now Doob's martingale inequality guarantees that:
  \[ \PP \left( \sup_{0 \leqslant r \leqslant s} | B_r - x | \geqslant
     \left| x - p(x) \right| \right) \leqslant \exp \left( -
     \frac{\left| x - p(x) \right|^2}{2 s} \right) . \]
  So that
\begin{equs}
    \bigg\vert \int_{\TT} \xi_{\mathrm{stat}} (\omega, x) &
    \mathbf{E}_x^{\mathbf{Q}} \left[ \bar{z}_{\infty} (\tau, B_s) \left( \int_0^s
    \xi_{\mathrm{stat}} (\omega, B_r) - \xi_{\mathrm{stat}} (\omega, x) \ud r
    \right) \right] \ud x \bigg\vert \\
    & \lesssim \| \xi_{\mathrm{stat}}(\omega) \|_{\infty}^2 \|
    \bar{z}_{\infty} (\tau) \|_{\infty}  s \int_{\TT}
    \exp \left( - \frac{\left| x - p(x) \right|^2}{2 s} \right) \ud
    x \\
    & = \sum_{i = 1}^{\mf{n}}\| \xi_{\mathrm{stat}} \|_{\infty}^2 \|
    \bar{z}_{\infty} (\tau) \|_{\infty}  s \int_{A_{i}} \exp \left( - \frac{| x
- p(x)|^2}{2 s} \right) \ud x \\
    & \leqslant 2 \mf{n} \| \xi_{\mathrm{stat}} \|_{\infty}^2 \|
    \bar{z}_{\infty} (\tau) \|_{\infty} s^{1 +
    \frac{1}{2}} \int_0^{\infty} \exp \left( - \frac{| x |^2}{2}
    \right) \ud x \\
    & \lesssim \| \xi_{\mathrm{stat}} \|_{\infty}^2 \|
    \bar{z}_{\infty} (\tau) \|_{\infty} s^{1 + \frac{1}{2}} .
\label{eqn:prf-moment-bound-piecewise}
\end{equs}
At this point we can conclude the estimate on the first moment of $ \zeta $.
Via \eqref{eqn:prf-moment-bound-smooth} by defining $ \gamma = \alpha \big(
\frac{ 1}{2} - \ve \big) $ or \eqref{eqn:prf-moment-bound-piecewise} with $
\gamma= \frac{1}{2} $ (depending on the assumption on the noise) together with
the moment assumption on $ \xi_{\mathrm{stat}} $ and
Lemma~\ref{lem:moment-bound-invariant-measure} for the moments of $
z_{\infty} $, we obtain:
  \begin{equs}
    \mathbf{E}^{ \PP \otimes \PP} [\zeta (\tau)] & = 
    \int_0^{\tau} \int_{\TT} \int_{\Omega} s (\xi_{\mathrm{stat}}
    (\omega, x))^2 \mathbf{E}_x^{\mathbf{Q}} [\bar{z}_{\infty} (\tau,
    B_s)] \ud \PP (\omega) \ud x \ud s +\mathcal{O} \left(
    \tau^{2 + \gamma } \right) .\\
    & =  \frac{\tau^2}{2}  \int_{\TT} \kappa (x, x)
    \mathbf{E}_x^{\mathbf{Q}} [\bar{z}_{\infty} (\tau, B_s)] \ud x
    +\mathcal{O} \left( \tau^{2 + \gamma} \right),
  \end{equs}
where we have defined $ \kappa(x,y) = \EE \big[ \xi_{\mathrm{stat}}(x)
\xi_{\mathrm{stat}}(y) \big] $.
  We have completed the estimate for the first moment of $ \zeta $.

  \textit{Step 2: Estimate for the second moment.} Let us fix any sequence
  $z (\tau)$ of functions such that for every $\tau > 0$ $z
  (\tau) \geqslant 0$ with $\int_{\TT} z (\tau, x) \ud x = 1$
  and concentrate on estimating the following (later we will replace
$z(\tau)$ by the random $z_{\infty}(\omega^{\prime} , \tau )$):
\begin{equs}
   \mathbf{E}^{\PP}  \bigg[ \bigg( \int_0^{\tau} \int_{\TT} H
    (\omega) e^{s H (\omega)} & (z (\tau)) (x) \ud x \ud s
    \bigg)^2 \bigg]  =  \int_{[0, \tau]^2} \int_{\TT^2}
    \int_{\Omega} F (x, y, s, r, \omega) \ud \PP
    (\omega) \ud x \ud y \ud s \ud r.
\end{equs}
 Here $F$ is defined by:
  \[ F (x, y, s, r, \omega) = \xi_{\mathrm{stat}} (\omega, x) [e^{s
     H (\omega)} (z (\tau))] (x) \xi_{\mathrm{stat}} (\omega, y)
     [e^{r H (\omega)} (z (\tau))] (y),\]
  and we can expand $F$ similarly to the previous step:
\begin{equs}
F (x_1, x_2, s_1, s_2, \omega) & =  \prod_{i = 1}^2 \xi_{\mathrm{stat}}
       (\omega, x_i) \mathbf{E}_{x_i}^{\mathbf{Q}} \left[ z (\tau,
       B_{s_i}) \exp \left( \int_0^{s_i} \xi_{\mathrm{stat}} (\omega, B_r)
       \ud r \right) \right]\\
       & =  \prod_{i = 1}^2 \xi_{\mathrm{stat}} (\omega, x_i)
       \mathbf{E}_{x_i}^{\mathbf{Q}} [z (\tau, B_{s_i}) (1 + R_2 (s_i,
       B, \omega))],
\end{equs}
  with $R_2$ bounded by:
\begin{equs}
| R_2 (s_i, B, \omega) | & =  \left| \exp \left( \int_0^{s_i}
    \xi_{\mathrm{stat}} (\omega, B_r) \ud r \right) - 1 \right|\\
    & \leqslant s_i \exp (s_i \| \xi_{\mathrm{stat}} (\omega) \|_{\infty}) .
\end{equs}
  Hence if we define
  \[ \bar{F} (x_1, x_2, s_1, s_2, \omega) = \prod_{i = 1}^2 \xi_{\mathrm{stat}}
     (\omega, x_i) \mathbf{E}_{x_i}^{\mathbf{Q}} [z (\tau, B_{s_i})],
  \]
  we obtain
  \begin{equs}
    \int_{[0, \tau]^2} \int_{\TT^2} \int_{\Omega} | F - \bar{F}
    | \ud \PP (\omega) \ud x \ud s & \lesssim  \mathbf{E}
    [e^{\| \xi_{\mathrm{stat}} \|_{\infty}} \| \xi_{\mathrm{stat}} \|_{\infty}^2]
    \| z (\tau) \|_{\infty}^2 \int_{[0, \tau]^2} \int_{\TT^2}
    s_1 + s_2 \ud x\ud s\\
    & \lesssim  \mathbf{E} [e^{\| \xi_{\mathrm{stat}} \|_{\infty}} \|
    \xi_{\mathrm{stat}} \|_{\infty}^2] \| z (\tau) \|_{\infty}^2
    \tau^3 .
  \end{equs}
  Altogether, we have found that
  \begin{equ}
    \mathbf{E}^{\PP} \left[ \left( \int_0^{\tau} \int_{\TT} H
    (\omega) e^{s H (\omega)} (z (\tau)) (x) \ud x \ud s
    \right)^2 \right]  =  \int_{[0, \tau]^2} \int_{ \TT^2}
    \int_{\Omega} \bar{F} \ud \PP (\omega) \ud x \ud s
    + \| z (\tau) \|_{\infty}^2 \mathcal{O} (\tau^3) .
  \end{equ}
  Now the average of $ \overline{F} $ with respect to $ \PP$ is given by
  \begin{equ}
    \int_{[0, \tau]^2} \int_{\TT^2} \int_{\Omega} \bar{F}
    \ud \PP (\omega) \ud x \ud s  =  \int_{[0,
    \tau]^2} \int_{ \TT^2} \kappa (x, y) \prod_{i = 1}^2
    \mathbf{E}_{x_i}^{\mathbf{Q}} [z (\tau, B_{s_i})] \ud x
    \ud s.
  \end{equ}
  Finally, replacing \( z(\tau) \) with \( z_{\infty}(\tau) \) and using
Lemma~\ref{lem:moment-bound-invariant-measure} we obtain:
\begin{equs}
   \mathbf{E}^{\PP \otimes \PP} & \left[ \zeta (\tau) -
    \frac{\zeta^2 (\tau)}{2} \right] \\
   & =  \frac{\tau^2}{2} \int_{\Omega}  \int_{\TT} \kappa (x, x) \mathbf{E}_x^{\mathbf{Q}} [z_{\infty}
    (\tau, \omega^{\prime} , B_s)] \ud x - \int_{ \TT^2} \kappa (x_1, x_2) \prod_{i
    = 1}^2 \mathbf{E}_{x_i}^{\mathbf{Q}} [ z_{\infty} (\tau, \omega^{\prime},
B_{s_i})] \ud x_{1} \ud x_{2} \ud \PP(\omega^{\prime} ) \\
    &  \qquad +\mathcal{O} \left( \tau^{2 + \gamma} \right)\\
& = \frac{\tau^{2}}{4} \int_{\Omega} \int_{ \TT^{2}} \EE \big[
|\xi_{\mathrm{stat}}(x_{1}) - \xi_{\mathrm{stat}}(x_{2}) |^{2} \big] \prod_{i
    = 1}^2 \mathbf{E}_{x_i}^{\mathbf{Q}} [z (\tau, \omega^{\prime} ,  B_{s_i})]\ud x_{1} \ud
x_{2} \ud \PP(\omega^{\prime}) + \mathcal{O}(\tau^{2 + \gamma}).
\end{equs}
With this the proof is essentially complete. The last step is proving the
convergence
\begin{equs}
\lim_{\tau \to 0} \int_{\Omega} \int_{ \TT^{2}} \EE  \big[
|\xi_{\mathrm{stat}}(x_{1}) - & \xi_{\mathrm{stat}}(x_{2}) |^{2} \big] \prod_{i
    = 1}^2 \mathbf{E}_{x_i}^{\mathbf{Q}} [z (\tau, \omega^{\prime} ,  B_{s_i})]\ud x_{1} \ud
x_{2} \ud \PP(\omega^{\prime}) \\
&  = \int_{\TT^{2}} \EE \big[ |\xi_{\mathrm{stat}}(x_{1}) -
\xi_{\mathrm{stat}}(x_{2}) |^{2} \big]\ud x_{1} \ud x_{2},
\end{equs}
which is a consequence of Proposition~\ref{prop:averaging-invariant-measure}
and Lemma~\ref{lem:moment-bound-invariant-measure}.
\end{proof} \\
The following result is instead essential for the proof of
Lemma~\ref{lem:taylor-white-noise}.

\begin{lemma}\label{lem:stochastic-product}
Let $ \xi $ be white noise on the torus $ \TT $ as in
Assumption~\ref{assu:white-noise} and let $
\mathcal{I (\xi)} \colon \Omega \times [0, \infty) \times \TT \to \RR$ be defined by:
\begin{equs}
\mathcal{I}(\xi) (t) = \int_{0}^{t} P_{t-s} \xi \ud s.
\end{equs}
Then there exists a $ \delta_{*} > 0 $ such that for every $ \delta \in
(0, \delta_{*})$ and $ p \geqslant 1 $:
\begin{equs}
 \EE \bigg\| \xi \reso
\frac{\mathcal{I}(\xi)(s)}{\sqrt{s}}  - \sqrt{\frac{ \pi}{\kappa}} \bigg\|_{\mC^{-
3 \delta}}^{p} = \mathcal{O}(s^{\delta}).
\end{equs}
\end{lemma}

\begin{proof}
  Our aim will be to prove for any $ p \geqslant 1 $, that
$\EE \Big\| \xi \reso
\frac{\mathcal{I}(\xi)(s)}{\sqrt{s}}  - \sqrt{\frac{\pi}{\kappa}}\Big\|_{B^{-
2 \delta}_{p,p}}^{p} = \mathcal{O}(s^{\delta}).$
Since $ p $ is arbitrary large, the Besov
embedding $ B^{- 2\delta}_{p,p} \subseteq B^{-2 \delta - \frac{1}{p}}_{\infty,
\infty} $ imples the desired result. From the definition of the $ B^{-
\delta}_{p,p} $ norm, and up to replacing $ \delta $ by $ 2 \delta $, it suffices to prove that uniformly over \( l \in
\NN \cup \{ -1 \} \)
\begin{equs}[eqn:prf-stochastic-bound-wn-aim]
\sup_{x \in \TT} \EE \Big| \Delta_{l} \Big(  \xi \reso
\frac{\mathcal{I}(\xi)(s)}{\sqrt{s}}  - \sqrt{\frac{\pi}{\kappa}}\Big)(x) \Big|^{p} \lesssim
2^{\delta l p} t^{ \frac{\delta}{2} p}.
\end{equs}
In addition, since \( \xi \) is Gaussian, by hypercontractivity it suffices to
check the last claim for $ p = 2 $.
We can write the product under consideration in Fourier coordinates. Here we have
\begin{equs}
\mF \left(\xi \reso \frac{\mathcal{I}(\xi) (t)}{\sqrt{t}}\right)(k) & = \sum_{k_{1} + k_{2}=k } \sum_{| i-j | \leqslant 1}
\varrho_{i}(k_{1}) \varrho_{j}(k_{2}) \hat{\xi}(k_{1}) \frac{1}{\sqrt{t}
}  \int_{0}^{t} e^{- (t-s) \kappa | k_{2} |^{2}}\hat{\xi}(k_{2}) \ud s,
\end{equs}
where $ \hat{\xi}(k) = \mF (\xi)(k) $.
In particular, for any $ l \in \NN \cup \{ -1 \}$ the $ l-
$th Paley block is given by
\begin{equs}
\Delta_{l}& \left(\xi \reso \frac{\mathcal{I}(\xi) (t)}{\sqrt{t}}\right)(x) \\
& = \sum_{k \in \ZZ} e^{2 \pi i k \cdot x}\varrho_{l}(k) \sum_{k_{1} + k_{2}=k }
\psi_{0}(k_{1}, k_{2}) \hat{\xi}(k_{1}) \int_{0}^{t}
e^{-(t-s) \kappa | k_{2} |^{2}} \hat{\xi}(k_{2}) \ud s,
\end{equs}
with $ \psi_{0} (k_{1}, k_{2}) = \sum_{| i - j | \leqslant 1}
\varrho_{i}(k_{1}) \varrho_{j}(k_{2}) $. Note that $ \{ \hat{\xi} (k)
\}_{k \in \ZZ} $ is a set of complex Gaussians with covariance $ \EE
\hat{\xi}(k_{1}) \hat{\xi} (k_{2}) = 1_{\{ k_{1} + k_{2} = 0 \}} $. So we can
write the Ito chaos decomposition for $ f \colon \ZZ^{2} \to \RR $ with $
\sum_{k_{1}, k_{2} \in \ZZ} | f(k_{1}, k_{2}) |^{2} + \sum_{k_{1} \in \ZZ} |
f(k_{1}, -k_{1}) |< \infty $:
\begin{align*}
\sum_{k_{1}, k_{2} \in \ZZ} f(k_{1}, k_{2}) \hat{\xi} (k_{1}) \hat{\xi} (
k_{2}) = \int_{\ZZ^{2}} f(k_{1}, k_{2}) \hat{\xi}( \ud k_{1}, \ud k_{2}) +
\sum_{k_{1} \in \ZZ} f(k_{1}, - k_{1}),
\end{align*}
the first term on the right hand-side being a multiple stochastic integral in
the sense of \cite[Section 1.1.2]{NualartMalliavin}. For our purposes, we can
decompose
\begin{equs}
\Delta_{l}& \left(\xi \reso \frac{\mathcal{I}(\xi) (t)}{\sqrt{t}}\right)(x) =\\
& \int_{\ZZ^{2}} e^{2 \pi \iota (k_{1} + k_{2}) \cdot x}\varrho_{l}(k_{1}+
k_{2}) \psi_{0}(k_{1}, k_{2}) \frac{1}{\sqrt{t}}\int_{0}^{t}
e^{-(t-s) \kappa | k_{2} |^{2}}  \ud s \ \hat{\xi}( \ud k_{1}, \ud k_{2}) +
\mf{s}(t, l),
\end{equs}
where
\begin{equs}
\mf{s}(t, l) = \sum_{k \in \ZZ} \frac{1}{\sqrt{t}}\int_{0}^{t} e^{-(t-s)
  \kappa | k |^{2}} \ud s 1_{\{ l = 0 \}}
\end{equs}
is the zeroth chaos (we will write $ \mf{s}(t) \assign \mf{s}(t, 0) $). Now if we look at the second chaos, we find:
\begin{equs}
\EE \bigg\vert &\int_{\ZZ^{2}} e^{2 \pi \iota (k_{1} + k_{2}) \cdot x}\varrho_{l}(k_{1}+
k_{2}) \psi_{0}(k_{1}, k_{2}) \frac{1}{\sqrt{t}}\int_{0}^{t}
e^{-(t-s) \kappa | k_{2} |^{2}}  \ud s \ \hat{\xi}( \ud k_{1}, \ud k_{2})
\bigg\vert^{2} \ \ \label{eqn:prf-1dwn-moment-secoon-chaos} \\ 
& \lesssim \int_{\ZZ^{2}} \varrho_{l}^{2}(k_{1}+
k_{2}) \psi_{0}^{2}(k_{1}, k_{2}) \frac{1}{t} \bigg(\int_{0}^{t}
e^{-(t-s) \kappa | k_{2} |^{2}}\ud s\bigg)^{2}  \ud k_{1} \ud k_{2},
\end{equs}
and for any $ \ve \in (0, 1) $ we can bound
\begin{equs}
\int_{0}^{t} e^{-(t-s)\kappa | k |^{2}} \ud s & \lesssim \int_{0}^{t}
\frac{1}{(t-s)^{1 - \ve} | k |^{2 (1 - \ve)} } \ud s  \lesssim  \frac{t^{\ve}}{| k |^{2(1 - \ve)}}.
\end{equs}
Using this bound with $ \ve = \frac{1}{2} + \frac{1}{2} \delta$, for some small $ \delta >
0 $ we can bound
\eqref{eqn:prf-1dwn-moment-secoon-chaos} uniformly over $ t > 0 $ by
\begin{equs}
\int_{\ZZ^{2}}  \varrho_{l}^{2}(k_{1}+
k_{2}) \psi_{0}^{2}(k_{1}, k_{2}) t^{\delta} \frac{1}{| k_{2} |^{2 - \delta} }
\ud k_{1} \ud k_{2}  & = \int_{\ZZ^{2}} \varrho_{l}^{2}(k_{1}) \psi_{0}^{2}(k_{1}- k_{2}, k_{2})
t^{\delta} \frac{1}{| k_{2} |^{2 - \delta} } \ud k_{1} \ud k_{2}\\
& \lesssim \int_{\ZZ^{2}} \frac{1}{| k_{1} |^{1 -2 \delta}} \varrho_{l}^{2}(k_{1}) \psi_{0}^{2}(k_{1}- k_{2}, k_{2})
t^{\delta} \frac{1}{| k_{2} |^{1 + \delta} } \ud k_{1} \ud k_{2},
\end{equs}
where in the first step we have changed variables and in the second we have
used that $ | k_{1} - k_{2}  | \simeq | k_{2} | \gtrsim | k_{1} | $ on the
support of $ \psi_{0} $. Now since for $ \delta >0 $ the integral over $
k_{2} $ is convergent we are left with
\begin{equs}
t^{\delta} \sum_{k \in \ZZ} \frac{\varrho_{l}^{2}(k)}{| k |^{1 - 2\delta}}
\lesssim t^{\delta} 2^{ 2 \delta l}.
\end{equs}
If we now prove that the limit \( \lim_{t \to 0} \mf{s}(t) = \sqrt{\pi/\kappa} +
\mathcal{O}(t^{\frac{\delta}{2}}) \) we have completed the proof of
\eqref{eqn:prf-stochastic-bound-wn-aim}.
We can rewrite \( \mf{s} \) as:
\begin{equs}
\mf{s}(t) & = \frac{1}{\sqrt{t}}\int_{0}^{t} \frac{1}{\sqrt{t-s}} \sum_{k \in
\ZZ} \sqrt{t-s}  e^{- (t-s) \kappa |k|^{2}} \ud s  =\frac{1}{\sqrt{t}}\int_{0}^{t} \frac{1}{\sqrt{t-s}} J_{\sqrt{t-s}}\ud s .
\end{equs}
Here we have defined for $ \ve>0 $: $J_{\ve} = \sum_{k \in \ve \ZZ} \ve e^{-
  \kappa | k |^{2}}$. The latter is a Riemann approximation of the integral $ \int_{\RR}
  e^{- \kappa | x |^{2}} \ud x $ and we can bound, uniformly over $ \ve \in
  (0, 1)$:
\begin{equs}
  \bigg\vert J_{\ve} - \int_{\RR} e^{- \kappa| x |^{2}} \ud x \bigg\vert \lesssim \ve,
\end{equs}
so that $\mf{s}(t) = \sqrt{\pi / \kappa} + \mathcal{O}(\sqrt{t})$, which
concludes the proof, as we can assume $ \delta <1 $.
\end{proof}

\

\begin{appendix}

\section{Paraproducts}\label{sec:paraproducts}
Let us start by defining the Fourier transform and its inverse (here $ \iota = \sqrt{-1} $):
\begin{equs}
\mF \varphi (k) & =  \int_{\TT} e^{- 2 \pi \iota x k} \varphi (x) \ud x, \qquad
& \forall k \in \ZZ, \quad \varphi \in \mS^{\prime}( \TT ),& \\
\mF^{-1} \psi (x) & =  \sum_{k \in \ZZ} e^{ 2 \pi \iota x k} \psi (k), \qquad & \forall x \in
\TT, \quad \psi \colon \ZZ \to \RR.&
\end{equs}
Next fix two smooth functions with compact support $ \varrho_{- 1}, \varrho_{0} $ that generate a dyadic
partition of the unity in the sense of
\cite[Proposition 2.10]{BahouriCheminDanchin2011FourierAndNonLinPDEs}. Namely, such that for $ j \in \NN
$, defining  $ \varrho_{j}(k) = \varrho(2^{j} k) $, we have $ 1 = \sum_{j \geqslant -1}
\varrho_{j} (k), \ \forall k \in \ZZ. $ Then define the $ j-$th Paley block
\begin{equs}[eqn:paley-block]
\Delta_{j} \varphi (x) = \mF^{-1} [ \varrho_{j} \mF \varphi] (x) = \langle
\varphi, K_{j}^{x} \rangle, \quad \forall
j \geqslant -1, \quad \varphi \in \mS^{\prime} (\TT).
\end{equs}
Here $ K_{j}^{x}(y) = \mF^{-1} \varrho_{j} (x-y) $ and we observe that a
scaling argument guarantees that uniformly over $ j, x $
\begin{equs}[eqn:bound-K-j]
\| K_{j}^{x} \|_{B^{\alpha}_{1, 1}} \lesssim 2^{\alpha j}.
\end{equs}
Next, note that every Paley block is a smooth function on $ \TT $. We can thus define the
periodic Besov spaces $ B^{\alpha}_{p, q} \subseteq
\mS^{\prime}(\TT) $ for $ \alpha \in \RR, p, q \in [1, \infty]$ by:
\begin{equs}
\| \varphi \|_{B^{\alpha}_{p, q} } = \Big( \sum_{j \geqslant -1} 2^{\alpha j q}
\| \Delta_{j} \varphi \|_{L^{p}}^{q} \Big)^{\frac{1}{q}}.
\end{equs}
We can decompose a product of periodic distributions, formally, as 
\begin{equs}
\varphi \cdot \psi = \psi \para \psi + \varphi \reso \psi + \varphi\rpara
\psi,
\end{equs}
where 
\begin{equs}
\varphi \para \psi = \sum_{j \geqslant -1} \sum_{i \leqslant j -2} \Delta_{i} \psi \Delta_{j} \psi,
\qquad \varphi \reso \psi = \sum_{|i - j| \leqslant 1} \Delta_{i} \varphi
\Delta_{j} \psi.
\end{equs}
We refer to the first as the paraproduct and to the last term as
the resonant product between $ \varphi $ and $ \psi $. We will also require the
following modified paraproduct for time-dependent functions $ \varphi, \psi
\colon [0, \infty) \to \mS^{\prime} (\TT)$:
\begin{equs}[eqn:modified-paraproduct]
\varphi \ppara \psi \, (t)= \sum_{j \geqslant -1} \bigg( \int_{0}^{t} 2^{2 j}
\mf{f}(2^{2j}(t -s))\bigg(  \sum_{i \leqslant j -2} \Delta_{i} \varphi \bigg)
(s)   \ud s \, \bigg) \Delta_{j} \psi(t),
\end{equs}
where $ \mf{f} \colon \RR \to [0, \infty) $ is a fixed smooth function with
compact support $ \supp( \mf{f}) \subseteq [0, \infty) $ and such that $
\int_{\RR} \mf{f}(s) \ud s =1 $.
To lighten the notation
we write, for $ \alpha \in \RR, p \in [1, \infty] $
\begin{equs}
\| \varphi \|_{\mC^{\alpha}} = \| \varphi \|_{B^{\alpha}_{\infty,
\infty}}, \qquad \| \varphi \|_{\mC^{\alpha}_{p}} = \| \varphi
\|_{B^{\alpha}_{p, \infty}}.
\end{equs}
The following estimates hold for paraproducts.
\begin{lemma}\label{lem:paraproduct-estimate}
For $ \alpha, \beta \in \RR $ 
\begin{equs}
\| \varphi \para \psi \|_{\mC^{\beta + \alpha \wedge 0}} \lesssim \| \varphi
\|_{\mC^{\alpha}} \| \psi \|_{\mC^{\beta}}.
\end{equs}
If $ \alpha + \beta >0 $ one has in addition
\begin{equs}
\| \varphi \reso \psi \|_{\mC^{\alpha + \beta}} & \lesssim \| \varphi
\|_{\mC^{\alpha}} \| \psi \|_{\mC^{\beta}}, \qquad \| \varphi \cdot \psi
\|_{\mC^{\alpha \wedge \beta}}  \lesssim \| \varphi \|_{\mC^{\alpha}} \| \psi
\|_{\mC^{\beta}}.
\end{equs}
\end{lemma}
See for example \cite[Theorems 2.82,
2.85]{BahouriCheminDanchin2011FourierAndNonLinPDEs} for a proof. Next consider the periodic heat semigroup $ P_{t} = e^{t \kappa \Delta} $. The
following regularity estimates hold.
\begin{lemma}\label{lem:schauder-estiamtes}
For any $ \kappa > 0, \, \alpha \in \RR, p \in [1, \infty] $ and $ \beta \geqslant 0 $:
\begin{equs}
\| P_{t} \varphi \|_{\mC^{\alpha+ \beta}_{p}} \lesssim_{\kappa} t^{- \frac{\beta}{2} }
\| \varphi \|_{\mC^{\alpha}_{p}}, \qquad \| P_{t} \varphi - \varphi
\|_{\mC^{\alpha}_{p}} \lesssim_{\kappa} t^{\frac{\beta}{2}} \| \varphi
\|_{\mC^{\alpha + \beta}_{p}}, \qquad \forall \varphi \in \mS^{\prime}
(\TT).
\end{equs}
\end{lemma}
These estimates follow from the representation of the semigroup as a
Fourier multiplier, see e.g. \cite[Lemma A.7]{GubinelliPerkowski2017KPZ}. 
\end{appendix}

\newcommand{\etalchar}[1]{$^{#1}$}

\end{document}